\documentclass[10pt]{article}
\usepackage{amsthm,amsfonts,amsmath,amscd,amssymb,epsfig,epsf,verbatim}
\title{First steps in stable Hamiltonian topology}
\author{K.~Cieliebak and E.~Volkov\footnote{Supported by DFG
    grant MO 843/3-3, FNRS grant FRFC 2.4655.06, and the ESF
    Research Networking Programme {\em Contact and Symplectic Topology
      (CAST)}}}
\theoremstyle{plain}
\newtheorem{theorem}{Theorem}[section]
\newtheorem{thm}[theorem]{Theorem}
\newtheorem{corollary}[theorem]{Corollary}
\newtheorem{cor}[theorem]{Corollary}
\newtheorem{proposition}[theorem]{Proposition}
\newtheorem{prop}[theorem]{Proposition}
\newtheorem{lemma}[theorem]{Lemma}

\theoremstyle{remark}

\newtheorem{question}[theorem]{Question}
\newtheorem{remark}[theorem]{Remark}
\newtheorem{example}[theorem]{Example}

\newtheorem*{definition}{Definition}

\newcommand{\id}{{{\mathchoice {\rm 1\mskip-4mu l} {\rm 1\mskip-4mu l}
{\rm 1\mskip-4.5mu l} {\rm 1\mskip-5mu l}}}}

\newcommand{\ddtt}{\frac{d}{dt}\bigl|_{t=0}}

\newcommand{\p}{\partial}

\newcommand{\om}{\omega}
\newcommand{\Om}{\Omega}
\newcommand{\eps}{\varepsilon}
\newcommand{\into}{\hookrightarrow}
\newcommand{\la}{\langle}
\newcommand{\ra}{\rangle}
\newcommand{\N}{{\mathbb{N}}}
\newcommand{\Z}{{\mathbb{Z}}}
\newcommand{\R}{{\mathbb{R}}}
\newcommand{\C}{{\mathbb{C}}}
\newcommand{\Q}{{\mathbb{Q}}}

\newcommand{\im}{{\rm im }}        
\newcommand{\st}{{\rm st}}
\newcommand{\vol}{{\rm vol}}

\newcommand{\const}{{\rm const}}

\newcommand{\curl}{{\rm curl\,}}

\renewcommand{\min}{{\rm min}}
\renewcommand{\max}{{\rm max}}

\newcommand{\inn}{{\rm int\,}}
\newcommand{\tb}{{\rm tb}}

\newcommand{\loc}{{\rm loc}}

\newcommand{\ot}{{\rm ot}}
\newcommand{\nv}{{\rm nv}}
\newcommand{\Hel}{{\rm Hel}}
\newcommand{\Leg}{{\rm Leg}}
\newcommand{\trans}{{\rm trans}}

\newcommand{\Diff}{{\rm Diff}}
\newcommand{\DD}{\mathcal{D}}
\newcommand{\EE}{\mathcal{E}}

\newcommand{\CC}{\mathcal{C}}
\newcommand{\LL}{\mathcal{L}}
\newcommand{\FF}{\mathcal{F}}

\renewcommand{\AA}{\mathcal{A}}
\newcommand{\UU}{\mathcal{U}}

\newcommand{\RR}{\mathcal{R}}
\newcommand{\SFT}{\mathcal{SFT}}
\newcommand{\SHS}{\mathcal{SHS}}
\newcommand{\HS}{\mathcal{HS}}
\newcommand{\Cont}{\mathcal{CF}}
\begin{document}
\maketitle
In this paper we study topological properties of stable Hamiltonian
structures. In particular, we prove the following results in dimension
three: The space of stable Hamiltonian structures modulo homotopy is
discrete; there exist stable Hamiltonian structures that are not
homotopic to a positive contact structure; stable Hamiltonian
structures are generically Morse-Bott (i.e.~all closed orbits are Bott
nondegenerate) but not Morse; the standard contact structure on $S^3$
is homotopic to a stable Hamiltonian structure which cannot be
embedded in $\R^4$. Moreover, we derive a structure theorem in
dimension three and classify stable Hamiltonian structures supported
by an open book. We also discuss implications for the foundations of
symplectic field theory.
\tableofcontents

\section{Introduction}\label{sec:intro}

A stable Hamiltonian structure is a generalization of a contact
structure as well as a taut foliation defined by a closed form. 
Stable Hamiltonian structures first appeared in~\cite{HZ} as a 
condition on hypersurfaces for which the Weinstein conjecture can be
proved. Later, they attained importance as the structure on a
manifold needed for the compactness result in symplectic field
theory~\cite{EGH,BEHWZ,CM}. Further interest in stable Hamiltonian
structures arises from the recent proof of the Weinstein conjecture in
dimension three by Hutchings and Taubes~\cite{HT}, and from their
relation to Ma\~n\'e's critical values~\cite{CFP}. Stable Hamiltonian
structures also appear in work of Eliashberg, Kim and
Polterovich~\cite{EKP} and Wendl~\cite{We}.  
 
In this paper we take first steps in
studying the topology of stable Hamiltonian structures. Besides their
intrinsic interest, some of these topological questions are also
relevant for the foundations of symplectic field theory. 

\begin{definition}
A {\em Hamiltonian structure (HS)} on an oriented $(2n-1)$-dimensional
manifold $M$ is a closed 2-form $\om$ of maximal rank, i.e.~such that
$\om^{n-1}$ vanishes nowhere. Associated to $\om$ is its 1-dimensional
{\em kernel distribution (foliation)} $\ker(\om):=\{v\in TM\mid
i_v\om=0\}$. We orient $\ker(\om)$ using the orientation on $M$
together with the orientation on the local transversal to $\ker\om$
given by $\om^{n-1}$. 
A {\em stabilizing 1-form} for $\om$ is a 1-form $\lambda$ such that 
\begin{equation}\label{eq:def}
\lambda\wedge\om^{n-1}>0 \qquad\text{and}\qquad  
\ker(\om)\subset\ker(d\lambda).
\end{equation}
A Hamiltonian structure $\om$ is called {\em stabilizable} if
it admits a stabilizing 1-form $\lambda$, and the pair $(\om,\lambda)$
is called a {\em stable Hamiltonian structure (SHS)}. A SHS
$(\om,\lambda)$ induces a canonical {\em Reeb vector field} $R$
generating $\ker(\om)$ and normalized by $\lambda(R)=1$. Note that if
$(\om,\lambda)$ is a SHS, then $(\om,-\lambda)$ is a SHS inducing the
opposite orientation. 

We call an oriented 1-foliation $\LL$ a {\em stable Hamiltonian
  foliation} if it is the kernel foliation of a SHS. 
\end{definition}

Note that there is no analogue for (stable) Hamiltonian structures of the
stability results of Moser and Gray for symplectic resp.~contact
structures (see e.g.~\cite{MS}) because the dynamics of the kernel
foliation can change drastically under small
perturbations. By ``stable Hamiltonian topology'' we mean the study
of stable Hamiltonian structures up to homotopy in the following
sense. 

\begin{definition}
A {\em homotopy of Hamiltonian structures} is a smooth family $\om_t$, for 
$t$ in some interval $I\subset\R$, such that the cohomology class of
$\om_t$ remains constant. The homotopy is called {\em
  stabilizable} if it admits a smooth family of stabilizing 1-forms
$\lambda_t$, and the pair $(\om_t,\lambda_t)$ is called a  
{\em homotopy of stable Hamiltonian structures}, or simply {\em stable
  homotopy}. 
\label{homdef}
\end{definition}

\begin{remark}\label{rem:exact}
Let us emphasize that for a homotopy of HS $\om_t$ we always require
$\dot\om_t$ to be {\em exact}. This is the relevant notion for many
reasons that will become clear in this paper (e.g.~to get invariance
of symplectic field theory under stable homotopies, see
Section~\ref{subsec:sft}). We will exclusively use this notion of
homotopy, with the exception of Sections~\ref{subsec:linstab}
and~\ref{subsec:pert-contact} where we will briefly consider
deformations $\om_t$ with varying cohomology class (to which we will
explicitly refer as ``non-exact deformations''). 
\end{remark}

{\bf Discreteness. }
For a de Rham cohomology class $\eta\in H^2(M;\R)$, we denote by
$\SHS_\eta(M)$ and $\HS_\eta(M)$ the spaces of SHS resp.~HS with
$[\om]=\eta$, equipped with the $C^k$-topology for some $2\leq
k\leq\infty$. 

\begin{thm}[cf.~Theorem~\ref{thm:discrete}]\label{thm:discrete-intro} 
Every SHS $(\om,\lambda)$ on a closed 3-manifold has a
$C^2$-neighbourhood in $\SHS_{[\om]}(M)$ in which all SHS are stably 
homotopic to $(\om,\lambda)$. 
\end{thm}

This means that, for any closed 3-manifold $M$ and cohomology class
$\eta\in H^2(M;\R)$, the space $SHS_\eta(M)/\sim$ is discrete in the
topological sense, which justifies the term ``stable Hamiltonian
topology''. In particular, this implies that there are at most
countably many homotopy classes of SHS representing $\eta$
(Corollary~\ref{cor:countable}). 

\begin{question}
Does an analogue of Theorem~\ref{thm:discrete-intro} hold in higher 
dimensions? 
\end{question}

{\bf Stabilization. }
Much of this paper is concerned with properties of the natural map
$$
   \Pi: \SHS_\eta\to \HS_\eta,\qquad (\om,\lambda)\mapsto\om. 
$$
It has the following properties:

(a) Each nonempty fibre $\Pi^{-1}(\om)$ is convex and hence
contractible.

(b) $\Pi$ is in general not surjective. In fact, its image
$\im(\Pi)\subset\HS_\eta$ is in general neither closed (this is easy to
see) nor open (this is much harder and proved
in~\cite{CFP-notopen}. 

(c) There exist smooth paths in $\im(\Pi)$ which have no smooth lift
to $\SHS_\eta$. More precisely, we show

\begin{thm}[cf.~Theorem~\ref{thm4}]
On every closed oriented 3-manifold there exists a stable
Hamiltonian foliation $\LL$ with the following property. 
For any HS $\om_0$ defining $\LL$ there exists
a Baire set $\tilde B$ in the space $\tilde E$ of (exact) deformations
$\{\om_t\}_{t\in[0,1]}$ of $\om_0$ that cannot be stabilized, no
matter what stabilizing  $1$-form $\lambda_0$ we take for $\om_0$. 
\end{thm}

A similar result holds for lifts of convergent sequences, see
Theorem~\ref{thm:seq}. For non-exact deformations of HS, easy
obstructions to stabilizability arise from foliated cohomology
(Section~\ref{subsec:pert-contact}).  

{\bf Contact structures. }
Every positive contact form $\lambda$ induces a SHS
$(d\lambda,\lambda)$ and homotopies of contact forms induce stable
homotopies, so we have a natural map 
$$
   \Cont/\sim\to \SHS_0/\sim
$$
from homotopy classes of positive contact forms to homotopy classes of
exact SHS. Using Eliashberg's classification of contact structures on
$S^3$ we prove 
 
\begin{thm}[cf.~Theorem~\ref{thm:contact-SHS}]\label{thm:contact-SHS-intro}
On $S^3$ the map $\Cont/\sim\to \SHS_0/\sim$ is not
bijective. In fact, if transversality for holomorphic curves can be
achieved (see Remark~\ref{rem:SFT-intro} below), then the map is
injective but not surjective . 
\end{thm}

{\bf Open books. }
Consider an open book decomposition $(B,\pi)$ of a closed oriented
3-manifold $M$ with binding $B$ and fibration $\pi:M\setminus B\to S^1$
(see Section~\ref{sec:openbook} for the precise definition). 
We call a SHS $(\om,\lambda)$ {\em supported by $(B,\pi)$}
if $\om$ is positive on each page $\pi^{-1}(\theta)$. It follows that
$\lambda$ is  
nowhere vanishing on the binding. Thus for every binding component
$B_i$ we have a sign $\pm$ which is $+$ iff $\lambda$ induces the
orientation of $B_i$ as boundary of a page. 

\begin{thm}[cf.~Theorems~\ref{obd1} and~\ref{obd2}]
Let $(B,\pi)$ be an open book decomposition of a closed oriented
3-manifold $M$ and $\eta\in H^2(M;\R)$. Then there exists a SHS
$(\om,\lambda)$ supported by $(B,\pi)$ with $[\om]=\eta$ and with
given signs at the binding components.   
Moreover, any two SHS $(\om,\lambda)$ and $(\tilde\om,\tilde\lambda)$
supported by $(B,\pi)$ with $[\om]=[\tilde\om]\in
H^2(M;\R)$ and have the same signs at the binding components are connected
by a stable homotopy supported by $(B,\pi)$.    
\end{thm}

By a theorem of Giroux every contact structure is {\em positively}
(i.e.~with all signs $+$) 
supported by some open book. This is not the case for SHS: A SHS as in
Theorem~\ref{thm:contact-SHS-intro} which is not homotopic to a
contact one is not homotopic to one that is positively supported by an
open book. 

\begin{question}
Is every SHS in dimension 3 homotopic to one that is supported by an
open book (with some signs at the binding components)? 
\end{question}

{\bf Structure theorem in dimension 3. }
We have seen that not every SHS is homotopic to a contact
one. However, in dimension 3 we have the following structure theorem. 

\begin{theorem}[cf.~Corollary~\ref{cor:structure}] 
Every stable Hamiltonian structure on a closed 3-manifold $M$ is 
stably homotopic to a SHS $(\om,\lambda)$ for which 
there exists a (possibly disconnected and possibly with boundary)
compact $3$-dimensional submanifold $N=N^+\cup N^-\cup N^0$ of $M$,
invariant under the Reeb flow, and a (possibly empty) disjoint union
$U=U_1\cup\dots\cup U_k$ of compact regions with the
following properties:  
\begin{itemize}
\item $\inn U\cup \inn N=M$;
\item $d\lambda=\pm\om$ on $N^\pm$ and $d\lambda=0$ on $N^0$;
\item on each $U_i\cong [0,1]\times T^2$ the SHS $(\om,\lambda)$ is
  $T^2$-invariant. 
\end{itemize}
Moreover, we can always arrange that $N^+$ is nonempty, and if $\om$
is exact we can arrange that $N^0$ is empty. 
\end{theorem}

Roughly speaking, this says that $M$ is a union of regions on which
$(\om,\lambda)$ is positive contact, negative contact, or flat,
glued along 2-tori invariant under the Reeb flow. 

{\bf Morse-Bott approximations and SFT. }
In order to define invariants such as symplectic field theory (SFT) or
Rabinowitz Floer homology, one needs to consider SHS that are {\em
  Morse} (i.e.~all closed Reeb orbits are nondegenerate) or at least
{\em Morse-Bott} (see Section~\ref{subsec:Morse-Bott}). It is well
known that any contact form can be $C^\infty$-approximated by one
which is Morse. Surprisingly, this fails for more general SHS:
 
\begin{thm}[cf.~Theorems~\ref{thm1}, \ref{thm2} and
  \ref{thm3}]\label{thm:main} 
In every stable homotopy class in dimension 3 there exists a stable
Hamiltonian structure which cannot be
$C^2$-approximated by stable Hamiltonian structures which are Morse.  
\end{thm}

\begin{remark}
Using the techniques in the proof of Theorem~\ref{thm:Morse-Bott}, the
stable Hamiltonian structure in Theorem~\ref{thm:main} can actually be 
chosen to be Morse-Bott. 
\end{remark}

This is bad news for the foundations of SFT: Besides the
transversality problems for holomorphic curves, one faces the
additional difficulty of making a SHS Morse, or at least
Morse-Bott. We discuss this in more detail in
Section~\ref{subsec:sft} and show that this difficulty can be overcome
in dimension 3:

\begin{theorem}[cf.~Theorems~\ref{thm:Morse-Bott} and~\ref{thm:sft}] 
Any SHS on a closed oriented 3-manifold can be connected to a SHS
which is Morse-Bott by a $C^1$-small stable homotopy. This suffices to
define SFT as a homotopy invariant of SHS in dimension 3, provided
transversality for holomorphic curves can be achieved. 
\end{theorem}

\begin{cor}[cf.~Corollary~\ref{cor:sft}]\label{cor:sft-intro}
Suppose transversality for holomorphic curves can be achieved. 
Then the SHS $(d\alpha_\st,\alpha_\st)$ and $(d\alpha_\ot,\alpha_\ot)$ on
$S^3$ are not stably homotopic, where $\alpha_\st$ is the standard
contact form and $\alpha_\ot$ is an overtwisted contact form defining
the same orientation. 
\end{cor}

\begin{remark}\label{rem:SFT-intro}
The assumption on transversality for holomorphic curves will be
satisfied once the polyfold program of Hofer, Wysocki and
Zehnder~\cite{HWZ} is finished. In the present paper we use this
assumption only in Corollary~\ref{cor:sft} and the second part of 
Theorem~\ref{thm:contact-SHS}. 
\end{remark}

{\bf h-principle. }
Hamiltonian structures satisfy an h-principle (\cite{McD}, see
Section~\ref{subsec:h} for the precise statement). For stable
Hamiltonian structures in dimension 3, the 0-parametric h-principle
holds (Proposition~\ref{prop:ex}). By contrast,
Corollary~\ref{cor:sft-intro} shows that, assuming 
transversality for holomorphic curves can be achieved, the
1-parametric h-principle fails for stable Hamiltonian structures in
dimension 3.  

{\bf Cobordisms. }
Symplectic field theory satisfies TQFT type axioms with respect to
symplectic cobordisms. In order to turn it into a homotopy invariant
of SHS, we need to construct suitable symplectic cobordisms from stable
homotopies. The naive definition of a {\em (topologically trivial)
  symplectic cobordism} between HS 
$\om_a$ and $\om_b$ on $M$ is a symplectic manifold $([a,b]\times M,\Om)$
such that $\Om|_{\{i\}\times M}=\om_i$ for $i=a,b$. This definition
turns out to be too restrictive (a SHS is not even cobordant to itself
in this sense!). We introduce two more general notions: {\em strong
  cobordism} (replacing $\om_i$ by $\om_i+t_id\lambda_i$ for
sufficiently small $t_i$), and {\em weak cobordism} (replacing $\om_i$
by more general closed 2-forms with the same kernel), see
Section~\ref{ss:cob-def}. Moreover, we define a {\em length} of a
stable homotopy (Section~\ref{subsec:shorth}) and show that every
sufficiently short stable homotopy gives  rise to a strong cobordism
(Proposition~\ref{maincob}). This suffices for the homotopy invariance
of SFT, modulo transversality for holomorphic curves
(Theorem~\ref{thm:sft}).  

In Section~\ref{subsec:largeh} we construct
pairs of SHS which are homotopic as HS but not weakly cobordant, with
obstructions to weak cobordisms arising from helicity and fillability.
Our failure to construct such examples which are {\em stably} homotopic
motivates the following

\begin{question}
Given a stable homotopy $(\om_t,\lambda_t)_{t\in[0,1]}$, are $\om_0$ and
$\om_1$ weakly bicobordant?
\end{question} 

{\bf Embeddability and ambient homotopies. }
Given an abstract SHS $(M^{2n-1},\om,\lambda)$, we can ask whether it
admits an embedding $\iota:M\into W$ into a given symplectic manifold
$(W^{2n},\Om)$ such that $\iota^*\Om=\om$. 

\begin{thm}[cf.~Corollary~\ref{cor:exotic2}]
There exists a SHS $(\om,\lambda)$ on $S^3$ which is 
stably homotopic to $(d\alpha_\st,\alpha_\st)$ but cannot be embedded
in $(\R^4,\Om_\st)$. 
\end{thm}

This shows that abstract stable homotopies cannot in general be
realized by homotopies of hypersurfaces
in a fixed symplectic manifold. On the other hand, under an additional
``tameness'' assumption, many examples of smoothly but not tame stably
homotopic hypersurfaces are constructed in~\cite{CFP}. In
Section~\ref{subsec:tame} we exhibit a large class of SHS that are not
tame (Corollary~\ref{cor:tame}).  

{\bf Integrability in dimension 3. }
Although we set up the theory in arbitrary dimensions, most of our
actual results concern dimension 3. The reason is that for a SHS
$(\om,\lambda)$ on a 3-manifold $M$ we have $d\lambda=f\om$ for a
function $f:M\to\R$ which is invariant under the Reeb flow. By a
version of the Arnold-Liouville theorem first observed
in~\cite{Ar-hydro}, regions where $df\neq 0$ are foliated by invariant
2-tori on which the Reeb flow is linear
(Theorem~\ref{thm:integr}). The study of SHS on these {\em integrable
  regions} reduces to the study of $T^2$-invariant SHS on $[0,1]\times
T^2$ (Section~\ref{subsec:t2inv}). On the other hand, on regions where
$f$ is constant the SHS is either positive contact ($f>0$), negative
contact ($f<0$), or flat ($f=0$), so it can be studied using methods
from contact topology and foliations. The remaining difficulty is thus
to understand the boundaries between these regions. It is overcome by
our main technical result, Proposition~\ref{prop:thick}, which allows
to replace $\lambda$ by a new stabilizing 1-form $\tilde\lambda$ such
that the boundaries of the regions where the new function $\tilde
f=d\tilde\lambda/\om$ is constant are contained in integrable regions
for $f$. 
\bigskip

{\bf Acknowledgements. }
We thank the participants of the 2008 seminar ``Topics in Symplectic
Geometry'' at LMU for patiently listening to our half-baked ideas on
stable Hamiltonian structures and correcting our numerous
misconceptions. Moreover, we benefited from discussions with many
people of whom we only mention a few (and ask those omitted for
forgiveness): 
J.~Bowden, 
Y.~Chekanov, 
Y.~Eliashberg,
U.~Frauenfelder, 
G.~Paternain, 
T.~Vogel,
C.~Wendl.

The second author 
thanks F.~Bourgeois for his hospitality at ULB Bruxelles, 
K.~Mohnke for his hospitality at HU Berlin, and C.~Stamm
for generously sharing her unpublished work~\cite{Sta}.

\section{Background on stable Hamiltonian structures}\label{sec:SHS}

In this section we discuss the basic properties of stable Hamiltonian
structures and collect some examples. 

\subsection{Basic examples}\label{subsec:ex}

The three basic examples of stable Hamiltonian structures are the
following. 

{\em Contact manifolds: }$(M,\lambda)$ is a contact manifold, $R$ is
  the Reeb vector field, and $\om=\pm d\lambda$.  

{\em Mapping tori: }$M:=W_\phi:=\R\times
  W/(t,x)\sim\bigl(t+1,\phi(x)\bigr)$ is the mapping torus of a
  symplectomorphism $\phi$ of a symplectic manifold $(W,\bar\om)$, 
  $R=\frac{\p}{\p t}$, $\lambda=dt$, and $\om$ is the
  form on $M$ induced by $\bar\om$. Note that $d\lambda=0$, so
  $\ker(\lambda)$ defines a {\em foliation}. Note that 
  $W_\phi\cong[0,1]\times W/(0,x)\sim (1,\phi(x))$

{\em Circle bundles: }$\pi:M\to W$ is a principal circle bundle over a
  symplectic manifold $(W,\bar\om)$, $R$ is the vector field
  generating the circle action, $\lambda$ is a connection form, and
  $\om=\pi^*\bar\om$. Note that if $\om=d\lambda$ is the curvature
  of the connection $\lambda$ then we are actually in the contact
  case (Boothby-Wang construction).

\subsection{Stable hypersurfaces}\label{subsec:hyper}

Historically, the stability condition for Hamiltonian structures first
appeared as a dynamical stability condition for hypersurfaces in
symplectic manifolds~\cite{HZ}. To see this relation, let $(X,\Om)$ be
a symplectic manifold and let $M\subset X$ be a closed
hypersurface. Note that the restriction $\om:=\Om\bigl|_M$ is a HS. 

\begin{lemma}[\cite{CM}, Lemma 2.3]\label{lem:hyper}
For a closed hypersurface $M$ in a symplectic manifold $(X,\Om)$ the
following are equivalent: 
\begin{itemize}
\item [(a)] The hypersurface $M$ is stable in the sense of \cite{HZ}, i.e. 
there exists a tubular neighbourhood 
$(-\varepsilon,+\varepsilon)\times M$ of $\{0\}\times M$ such that the
$1$-dimensional kernel 
foliations of $\Om\bigl|_{\{t\}\times M}$ on $\{t\}\times M$ are all
conjugate via a family of diffeomorphisms depending smoothly on $t$.
\item [(b)] There exists a vector field $Y$ transverse to $M$ such
  that $\ker(L_Y\Om\bigl|_M)\subset\ker(\Om\bigl|_M)$.
\item [(c)] The Hamiltonian structure $(M,\Om\bigl|_M)$ is stabilizable.
\end{itemize}
\label{equiv}
\end{lemma}

Given a SHS $(\omega,\lambda)$ on an abstract manifold $M$, we can
always realize $M$ as a stable 
hypersurface in some symplectic manifold $(X,\Om)$, i.e. there exists
an embedding  
$i:M\hookrightarrow X$ in a symplectic manifold $(X,\Om)$ such that
$i^{\star}\Om=\om$ and $i(M)$ is a stable hypersurface in $X$. For
this take, for instance, the {\em symplectization}
\begin{equation}
   \Bigl(X:=(-\varepsilon,\varepsilon)\times M,\ 
   \Om:=\om+td\lambda+dt\wedge \lambda\Bigr)
\label{trivcob}
\end{equation}
for $\varepsilon>0$ small enough. In this example
a transverse vector field $Y$ in Lemma~\ref{lem:hyper} (b) can be
taken to be just $\p_t$ and a family  
$\{\phi_t\}_{t\in (-\varepsilon,+\varepsilon)}$
of diffeomorphisms $$\phi_t\colon\{0\}\times M\longrightarrow
\{t\}\times M$$ 
with $$\phi_{t*}\ker\Om\bigl|_{\{0\}\times
  M}=\ker\Om\bigl|_{\{t\}\times M}$$  
to be the flow of $\p_t$,
$$\phi_t(x,\tau):=(x,\tau+t).$$
The much more interesting question whether a given SHS can be embeded
into more specific symplectic manifolds, e.g.~closed ones or the
standard $\C^n$, will be discussed in Section~\ref{subsec:emb}.

\subsection{Stability and geodesibility}\label{subsec:geod}
The property of a Hamiltonian structure $\om$ to be stabilizable
depends only on its kernel distribution $\ker(\om)$. More generally,
we say that an oriented $1$-dimensional foliation $\LL$ is {\it
  stabilizable} if there exists a vector field $X$ generating $\LL$
such that 
\begin{equation}
\lambda(X)=1\qquad\text{and}\qquad i_Xd\lambda=0.
\label{eq:stab}
\end{equation}

\begin{definition} 
An orientable $1$-dimensional foliation $\LL$ is called {\it
  geodesible} if there exists a vector 
field $X$ generating $\LL$ and a Riemannian metric $g$ such that the
  (naturally parametrized) flow lines of $X$ are geodesics for $g$. 
\end{definition}

The following theorem gives a characterization of stabilizability which
is sometimes more convenient than Equation~\eqref{eq:def}. 

\begin{theorem}[Wadsley~\cite{Wa}]
An orientable $1$-dimensional foliation $\LL$ is stabilizable if and
only if it is geodesible. Given a vector field $X$
generating $\LL$ whose flow lines are geodesics for a metric $g$ and
such that $g(X,X)\equiv 1$, a stabilizing 1-form (called {\em Wadsley
  form}) is obtained by 
$$
   \lambda:=i_Xg.
$$
\label{thm:Wad}
\end{theorem}

\begin{proof}
Let us give the simple proof since we will need parts of it in the
sequel. Given a metric $g=\la\ ,\ \ra$ and a vector field $X$, we set
$\lambda:=i_Xg$ and compute for a second vector field $Y$:
\begin{align*}
   d\lambda(X,Y) 
   &= X\cdot\lambda(Y)-Y\cdot\lambda(X)-\lambda([X,Y]) \cr
   &= X\cdot\la X,Y\ra - Y\cdot|X|^2 - \la X,[X,Y]\ra \cr
   &= \la\nabla_XX,Y\ra + \la X,\nabla_XY-[X,Y]\ra - Y\cdot|X|^2 \cr 
   &= \la\nabla_XX,Y\ra + \la X,\nabla_YX\ra - Y\cdot|X|^2 \cr 
   &= \la\nabla_XX,Y\ra - \frac{1}{2}Y\cdot|X|^2.
\end{align*}
So we have shown the formula
\begin{equation}\label{eq:Wad}
   i_Xd(i_Xg) = i_{\nabla_XX}g-d\left(\frac{|X|^2}{2}\right). 
\end{equation}
Now let $X$ be a vector field generating $\LL$. If its 
flow lines are geodesics for a metric $g$ and $g(X,X)\equiv 1$, then
the right-hand side in Equation~\eqref{eq:Wad} vanishes, so the
1-form $\lambda=i_Xg$ satisfies Equation~\eqref{eq:stab}. Conversely,
if $\lambda$ is a 1-form satisfying Equation~\eqref{eq:stab} pick any
metric $g$ such that $i_Xg=\lambda$, i.e.~$X$ has length 1 and is
perpendicular to $\ker\lambda$ with respect to $g$. Then $g(X,X)=1$,
so in Equation~\eqref{eq:Wad} the left-hand side and the second term
on the right-hand side vanish and we obtain $\nabla_XX=0$,
i.e.~flow lines of $X$ are geodesics. 
\end{proof}

\begin{corollary}\label{cor:geod}
On an oriented 3-manifold, a vector field $X$ is the Reeb vector field
of a stable Hamiltonian structure if and only if there exist a metric
$g$ and a volume form $\mu$ such that 
\begin{equation}\label{eq:geod}
   \nabla_XX=0,\qquad g(X,X)\equiv 1,\qquad L_X\mu=0.  
\end{equation}
Given a metric $g$ and volume form $\mu$ satisfying~\eqref{eq:geod},
the stable Hamiltonian structure is given by
$$
   \lambda:= i_Xg,\qquad \om:= i_X\mu. 
$$
Moreover, given a stable Hamiltonian structure $(\om,\lambda)$ with
Reen vector field $X$, the metric and volume form
satisfying~\eqref{eq:geod} can be chosen such that
$\mu=\lambda\wedge\om$ is the volume form induced by $g$.  
\end{corollary}

\begin{proof}
Given a stable Hamiltonian structure $(\om,\lambda)$ with
Reeb vector field $X$, pick a metric $g$ such that $i_Xg=\lambda$,
i.e.~$X$ has length 1 and is perpendicular to $\ker\lambda$. Then
$g(X,X)=1$ and $\nabla_XX=0$ follows from $i_Xd\lambda=0$ and
Equation~\eqref{eq:Wad}. The volume form $\mu:=\lambda\wedge\om$
satisfies $L_X\mu=d\om=0$. If we pick $g$ on $\ker\lambda$ to be
$\om(\cdot,J\cdot)$ for a complex structure $J$ on $\ker\lambda$
compatible with $\om$, then $\mu$ is the volume form induced by $g$. 

The converse direction is an immediate consequence of
Equation~\eqref{eq:Wad}. 
\end{proof}

\subsection{Relation to hydrodynamics}

In dimension 3, Reeb vector fields of stable Hamiltonian structures
naturally arise as special solutions in hydrodynamics. 
This observation is due to Etnyre and Ghrist~\cite{EG} (in the contact
case, but the stable case is analogous). We refer to~\cite{AK} for
background on hydrodynamics. 

The velocity field $X$ of an ideal incompressible fluid on a closed
oriented Riemannian 3-manifold $(M,g)$ with volume form $\mu$ (not
necessarily the one induced by the metric) satisfies the {\em Euler
  equation} 
\begin{equation*}
   \frac{\p X}{\p t} + \nabla_XX = -\nabla p,\qquad L_X\mu=0.
\end{equation*}
Here the {\em pressure function} $p$ is uniquely determined up to a constant
by the equation $L_X\mu=0$. Stationary (i.e.~time-independent) solutions
thus satisfy 
\begin{equation}\label{eq:Euler1}
   \nabla_XX = -\nabla p,\qquad L_X\mu=0.
\end{equation}
In view of Equation~\eqref{eq:Wad}, this equation can be rewritten as 
\begin{equation}\label{eq:Euler2}
   i_Xd(i_Xg)=-d\alpha,\qquad L_X\mu=0,
\end{equation}
where $\alpha:= p+\frac{|X|^2}{2}$ is the
Bernoulli function. It follows from the first equation that $\alpha$
is a first integral for $X$ and it is shown in~\cite{Ar-hydro} that in
regions where $\alpha$ is nonconstant the flow of $X$ is completely
integrable (cf.~\cite{AK,EG} and the discussion in Section~\ref{subsec:integr}
below). The other extreme (and little understood) case are the {\em
  Beltrami fields}: solutions of Equation~\eqref{eq:Euler2} with
constant $\alpha$, i.e.~such that
\begin{equation}\label{eq:Beltrami}
   i_Xd(i_Xg)=0,\qquad L_X\mu=0
\end{equation}
for some metric $g$ and volume form $\mu$. The proof of the following
corollay is analogous to that of Corollary~\ref{cor:geod}. 

\begin{corollary}
On an oriented 3-manifold, a vector field $X$ is a Beltrami field if
and only if it generates the kernel foliation of a stable Hamiltonian
structure. 
Given a metric $g$ and volume form $\mu$
satisfying~\eqref{eq:Beltrami}, the stable Hamiltonian structure is
given by 
$$
   \lambda:= i_Xg,\qquad \om:= i_X\mu. 
$$
Moreover, given a stable Hamiltonian structure $(\om,\lambda)$ whose
kernel foliation is generated by $X$, the metric and volume form
satisfying~\eqref{eq:Beltrami} can be chosen such that
$\mu=\frac{1}{\lambda(X)}\lambda\wedge\om$ is the volume form induced
by $g$. 
\end{corollary}

\begin{remark}\label{rem:Beltrami}
(a) Note that any rescaling $fX$ of a Beltrami field by a positive
function $f:M\to\R$ is again a Beltrami field (just rescale $g$ and
$\mu$ accordingly), but the same is not true for the Reeb vector field
of a stable Hamiltonian structure. 

(b) Beltrami vector fields arise e.g.~as solutions of the stationary
Euler equation~\eqref{eq:Euler1} with constant pressure,
i.e.~$\nabla_XX=0$ and $L_X\mu=0$. Indeed, the rescaled vector field
$\bar X:=X/|X|$ satisfies the same equations and in addition $g(\bar
X,\bar X)\equiv 1$, so it satisfies Equation~\eqref{eq:Euler2} with
vanishing Bernoulli function, hence $\bar X$ as well as $X$ is
Beltrami. 

(c) The {\em curl} $\curl X$ of a Beltrami field $X$ (with respect to
$g,\mu$) is defined by the equation 
$$
   i_{\curl X}\mu = d(i_Xg). 
$$
The first equation in~\eqref{eq:Beltrami} implies $i_Xi_{\curl
X}\mu=0$, so $\curl X=fX$ for a function $f:M\to\R$. This function $f$
is a first integral of $X$ and will play a crucial role in
Section~\ref{sec:dim3}. 
\end{remark}

\subsection{Obstructions to stability}\label{subsec:obstr}

Not every orientable orientable $1$-dimensional foliation is
geodesible and, moreover, not every kernel foliation of a HS is. In
this paper, we will only use the second one of the following two easy
obstructions.  

{\em Obstruction}~1: An oriented $1$-foliation $\LL$ is non-geodesible
if it contains a {\em Reeb component}, i.e.~an oriented embedded
annulus $A$ consisting of leaves of $\LL$ such that the boundary
orientation of $\p A$ coincides with that given by $\LL$. Indeed, if 
the ambient manifold contains a Reeb component, then for any vector
field $X$ spanning $\LL$ and any $1$-form $\lambda$ satisfying
$i_Xd\lambda=0$ we have $\int_{A}d\lambda=0$; hence
by Stokes' theorem $\int_{\p A}\lambda=0$ and the equation
$\lambda(X)=1$ in~\eqref{eq:stab} cannot be satisfied.

{\em Obstruction}~2: A Hamiltonian structure $\om$ on a closed
manifold is not stabilizable
if $\om$ has a primitive $\alpha$ for which  
\begin{equation}
\alpha\wedge \om^{n-1}\equiv 0.
\label{foliat}
\end{equation}
Indeed, assume that there exits $\lambda$ stabilizing $\om$. Consider 
$$
   d(\alpha\wedge\lambda\wedge\om^{n-2}) = 
   \lambda\wedge\omega^{n-1} - \alpha\wedge d\lambda\wedge\om^{n-2}.
$$
The first term on the right-hand side is pointwise positive by the
first condition in~\eqref{eq:stab}, the second term is identically
zero by the second condition in~\eqref{eq:stab} and~\eqref{foliat}
(all factors vanish on the Reeb vector field), and the left-hand side
is exact. Integration over the manifold thus yields a contradiction. 

Obstructions 1 and 2 are special cases of the following general
characterization of geodesibility due to Sullivan. 

\begin{thm}[Sullivan~\cite{Su}]\label{thm:sullivan}
An oriented foliation $\LL$ is non-geodesible if and only if 
there exists a foliation cycle which can be arbitrarily well
approximated by boundaries of singular $2$-chains tangent to the
foliation. 
\end{thm}

\subsection{The h-principle for Hamiltonian
  structures}\label{subsec:h} 

Hamiltonian structures satisfy the h-principle. For this, let $M$ be
an odd-dimensional manifold. Denote by $\Om^2_{\rm nondeg}(M)$ the
space of (not necessarily closed) 2-forms on $M$ of maximal rank, and
by $\mathcal{HS}_a(M)$ the space of {\em closed} 2-forms of maximal
rank (i.e.~Hamiltonian structures) representing the
cohomology class $a\in H^2(M;\R)$. Both spaces are equipped with the
$C^\infty_\loc$-topology. The following h-principle was first proved
by McDuff~\cite{McD}, see also~\cite{EM}. In fact, it is a consequence
(via symplectization) of the h-principle for symplectic forms on open
manifolds. 

\begin{theorem}[McDuff~\cite{McD}]\label{thm:h}
The inclusion $\mathcal{HS}_a(M)\into\Om^2_{\rm nondeg}(M)$ is a
homotopy equivalence. In particular, if $M$ admits a 2-form of maximal
rank then every $a\in H^2(M)$ is represented by a Hamiltonian
structure, and two cohomologous Hamiltonian structures are 
homotopic iff they are homotopic through (not necessarily closed)
2-forms of maximal rank.
\end{theorem}

\begin{cor}
Let $M$ be an oriented 3-manifold $M$. Then every $a\in H^2(M)$ is
represented by a Hamiltonian structure, and two cohomologous
Hamiltonian structures are homotopic iff their kernel
distributions are homotopic as oriented line fields. In particular,
the Hamiltonian structures $d\alpha_0,d\alpha_1$ induced by contact
forms $\alpha_0,\alpha_1$ on $M$ are homotopic iff the contact
distributions are homotopic as oriented plane fields. 
\end{cor}

We will show in Proposition~\ref{prop:ex} below that the existence
part of this corollary also holds for stable Hamiltonian
structures. In Theorem~\ref{thm:contact-SHS} we will show that,
assuming transversality for holomorphic curves can be achieved, the
homotopy part fails for stable Hamiltonian structures.

\subsection{Foliated cohomology}\label{ss:folcoh}

Let $\LL$ be an oriented 1-dimensional foliation on a closed manifold
$M$. Fix any vector field $R$ generating $\LL$ and consider the
subcomplex  
$$
   \Om^k_\LL(M):=\{\alpha\in \Om^k(M)\mid i_R\alpha=0,i_Rd\alpha=0\}
$$
of the de Rham complex (note that this does not depend on the choice
of $R$). Its cohomology $H^k_\LL(M)$ is the {\em foliated
  cohomology} of the foliation $\LL$. 
The inclusion $\Om^k_\LL(M)\subset \Om^k(M)$ induces a canonical map
$$
   \kappa:H^k_\LL(M)\to H^k(M). 
$$
Note that any HS $\om$ with $\ker\om=\LL$ carries a cohomology class
$[\om]\in H^2_\LL(M)$, and any stabilizing 1-form $\lambda$ for $\LL$
defines a cohomology class $[d\lambda]\in \ker\bigl(\kappa:H^2_\LL(M)\to
H^2(M)\bigr)$. 

Foliated cohomology will appear throughout this paper. A first
illustration of its importance is the following ``foliated'' version of
Moser's stability theorem. 

\begin{lemma}\label{lem:Moser}
Let $(\om_t)_{t\in[0,1]}$ be a smooth family of Hamiltonian structures
with constant kernel foliation $\ker(\om_t)=\LL$ and
$[\om_t]=[\om_0]\in H^2_\LL(M)$ for all 
$t$. Then there exists a smooth family of diffeomorphisms
$(\phi_t)_{t\in[0,1]}$ with $\phi_0=\id$, $\phi_t^*\LL=\LL$ and
$\phi_t^*\om_t=\om_0$ for all $t$. 
\end{lemma}

\begin{proof}
Pick a vector field $R$ generating $\LL$ and a 1-form $\lambda$ with
$\lambda(R)=1$.  
Since $[\om_t]=[\om_0]\in H^2_\LL(M)$, there exist 1-forms $\mu_t$ with
$\om_0+d\mu_t=\om_t$ and $i_R\mu_t=0$. By an argument as for Theorem
3.17 in~\cite{MS},
the $\mu_t$ can be chosen to depend smoothly on $t$. 

Since $i_R\dot\mu_t=0$ and $\ker\om_t=\LL$, there exists a unique
$t$-dependent vector field $X_t$ satisfying $i_{X_t}\om_t=-\dot\mu_t$
and $\lambda(X_t)=0$. Its flow $\phi_t$ satisfies 
$$
   \frac{d}{dt}\phi_t^*\om_t = \phi_t^*(L_{X_t}\om_t+\dot\om_t) =
   \phi_t^*d(i_{X_t}\om_t+\dot\mu_t) = 0
$$
and hence $\phi_t^*\om_t=\om_0$, which in turn implies
$\phi_t^*\LL=\LL$. 
\end{proof}

\begin{example}\label{ex:circle}
Suppose $\LL$ are the orbits of a (locally free) circle action, so the
orbit space $B$ is an orbifold and we have a projection $\pi:M\to B$. 
Then $\Om^k_\LL(M)=\pi^*\Om^k(B)$ and $H^k_\LL(M)\cong H^k(B)$. The
map $\kappa:H^k_\LL(M)\to H^k(M)$ corresponds to the pullback
$\pi^*:H^k(B)\to H^k(M)$, which is in general neither surjective nor
injective. This example also shows that Lemma~\ref{lem:Moser}
fails under the weaker hypothesis $[\om_t]=[\om_0]\in H^2(M)$
(e.g.~take the Hopf fibration $S^3\to S^2$ and $\om_t$ the pullback of
an area form on $S^2$ of area $1+t$).  
\end{example}

\begin{remark}
In the preceding example $H^k_\LL(M)$ was finite dimensional. We will
see in Section~\ref{ss:folcoh3} that for stable Hamiltonian structures
in dimension three $H^2_\LL(M)$ is often infinite dimensional. 
\end{remark}

\subsection{Helicity}\label{ss:hel}
Let $M$ be a closed oriented 3-manifold. The {\em helicity pairing} of
two exact 2-forms $\mu=d\alpha$ and $\nu=d\beta$ is 
$$
   \Hel(\mu,\nu) := \int_M\alpha\wedge \nu = \int_M\beta\wedge \mu. 
$$
This is independent of the primitives $\alpha,\beta$ and defines a
symmetric bilinear form on the space of exact 2-forms. The associated quadratic
form 
$$
   \Hel(\mu) = \int_M\alpha\wedge\mu
$$
is called the {\em helicity} of $\mu$. 

Now consider a 1-dimensional oriented foliation $\LL$. An
exact 2-form $\om$ with $\LL\subset\ker(\om)$ carries a helicity
$\Hel(\om)$. If $\om=d\alpha$ with $\alpha|_\LL=0$ then $\alpha\wedge\om$
contracts to zero with $\LL$ and thus vanishes identically, so we have
$\Hel(\om)=0$. Thus helicity descends to a quadratic form on 
$$
   D_\LL := \ker\bigl(\kappa:H^2_\LL(M)\to H^2(M)\bigr). 
$$
Note that for each stabilizing 1-form $\lambda$ for $\LL$ we have
$d\lambda\in D_\LL$ and thus a helicity $\Hel(d\lambda)$.

\begin{example}
If $\om=d\alpha$ for a positive contact form $\alpha$ its helicity satisfies 
$$
   \Hel(\om) = \int_M\alpha\wedge d\alpha > 0. 
$$
\end{example}

\begin{example}
For $\LL$ corresponding to the fibres of a circle bundle $\pi:M\to B$
as in Example~\ref{ex:circle} 
we distinguish two cases according to its Euler number $e\in\Z$. If
$e=0$ the pullback $\pi^*:H^2(B)\to H^2(M)$ is injective, thus
$D_\LL=0$ and the helicity vanishes identically. If $e\neq 0$ the
pullback $\pi^*:H^2(B)\to H^*(M)$ is the zero map, thus $D_\LL\cong
H^2(B)$ and the helicity pairing is given by
$$
   \Hel:H^2(B)\times H^2(B)\to\R,\qquad
   (x,y)\mapsto\frac{\int_Bx\int_By}{e}.
$$
\end{example}

\subsection{Left-invariant Hamiltonian structures on
  $PSL(2,\R)$}\label{subsec:left}  

This example is a particular case of the contact one. We consider
left-invariant $1$-forms  
on the Lie group $PSL(2,\R)$. To be more concrete, we choose a basis 
$$H:=\left(\begin{array}{cc}
1 & 0 \\
0 & -1  \\
\end{array}\right),\qquad E_+:=\left(\begin{array}{cc}
0 & 1 \\
0 & 0  \\
\end{array}\right),\qquad E_-:=\left(\begin{array}{cc}
0 & 0 \\
1 & 0  \\
\end{array}\right).$$
of the Lie algebra $sl(2,\R)$. 
The structural equations are: 
$$
   [H,E_+]=2E_+,\qquad[H,E_-]=-2E_-,\qquad [E_+,E_-]=H.
$$ 
Let $(h,e^+,e^-)$ be the basis of $sl(2,\R)^{\star}$
dual to $(H,E_+,E_-)$. We extend $h$, $e^+$ and 
$e^-$ to left invariant $1$-forms on $PSL(2,\R)$
by left multiplication and denote the $1$-forms by the same
letters. The structural 
equations above can be rewritten in terms of the $1$-forms $h$, $e^+$ and
$e^-$ as follows: 
$$
   dh=e^-\wedge e^+,\qquad de^+=2e^+\wedge h,\qquad de^-=2h\wedge e^-.
$$ 
This shows that if $0\neq\alpha\in sl(2,\R)^*$, then $d\alpha$ is
nowhere zero and so for dimensional reasons defines a Hamiltonian
structure. For $$\alpha=ah+a_+e^++a_-e^-\ne 0$$ we
compute 
\begin{align*}
   \alpha\wedge d\alpha &= (ah+a_+e^++a_-e^-)\wedge 
   (adh+a_+de^++a_-de^-) \cr
   &= a^2h\wedge e^- \wedge e^++
   a_+a_-2e^+\wedge h\wedge e^-+a_-a_+2e^-\wedge e^+\wedge h \cr
   &= (a^2+4a_+a_-)h\wedge e^-\wedge e^+.
\end{align*} 
We fix an orientation of $PSL(2,\R)$
(and its compact quotients) by specifying the volume form 
$h\wedge e^-\wedge e^+$. It is clear from the computation above that
$\alpha$ is a positive contact 
form when $a^2+4a_+a_->0$, negative contact when $a^2+4a_+a_-<0$, and
defines a foliation when 
$a^2+4a_+a_-=0$. The cone 
$$
   C:=\{\alpha\in sl(2,\R)^*|a^2+4a_+a_-=0\}
$$
separates the space $sl(2,\R)^*$ into three connected components: one
consisting of positive contact forms, and the other two consisting   
of negative contact forms. For a contact form $\alpha\in sl(2,\R)^*$
its Reeb vector field (viewed as an element of $sl(2,\R)$) is given by 
$R=X/(a^2+4a_+a_-)$ with
$$
   X=aH+2a_-E_++2a_+E_-= \left(\begin{array}{cc}
a & 2a_- \\
2a_+ & -a  \\
\end{array}\right).
$$ 
Note that the spectrum of $X$ consists of two distinct real eigenvalues
if $\alpha$ is positive contact and of two conjugate imaginary
eigenvalues if $\alpha$ is negative contact.  
Let $\Gamma$ be a lattice in $PSL(2,\R)$ such that the quotient
$\Gamma\setminus PSL(2,\R)$ is smooth and compact. Then the
space of left-invariant forms on $PSL(2,\R)$ 
descends to a certain space $V$ of exterior forms on the quotient 
$\Gamma\setminus PSL(2,\R)$ 
and the natural identification is an isomorphism of graded commutative
algebras. Thus all our equalities involving left-invariant
forms or left-invariant vector fields on $PSL(2,\R)$ induce 
the corresponding equalities of their descendants on
$\Gamma\setminus PSL(2,\R)$. In particular, the cone $C$ descends to a
cone in $V$ (still denoted by $C$) which separates $V$ into three
connected components as before. Moreover, {\it Obstruction}~2 in
Section \ref{subsec:geod} yields: {\em For any nonzero $\alpha\in
C\subset V$ (i.e., $\alpha$ defines a foliation) the HS $d\alpha$ is
not stabilizable.}    

The dynamics on the quotient is well understood in view of
\begin{theorem}[\cite{Aus} Theorem 5.3]\label{thm:left}
(a) For a negative contact form $\alpha_-\in V$ the Reeb flow is
periodic. 

(b) For a positive contact form $\alpha_+\in V$ the Reeb flow is
ergodic.  
\end{theorem}

We elaborate a bit on the contact cases. 

{\em Case 1: }
$\alpha_-\in V$ is a negative contact form. By
Theorem~\ref{thm:left} (a), its Reeb vector field $R$ defines a
(locally free) circle action on $M:=\Gamma\setminus PSL(2,\R)$, so its
orbit space $\Sigma$ is a closed 2-dimensional oriented orbifold. 
Denote by $\pi\colon M\rightarrow \Sigma$ the projection and by $\LL$
the foliation defined by $R$. According to example~\ref{ex:circle},
the foliated cohomology is given by $H^2_\LL(M)\cong
H^2(\Sigma)\cong\R$. Since $d\alpha_-$ is nowhere vanishing on planes
transverse to $\LL$ it represents a nonzero class $[d\alpha_-]\in
H^2_\LL(M)$ and therefore generates $H^2_\LL(M)$.  
If $\theta$ is a closed $2$-form on $M$ with  
\begin{equation}
   i_R\theta=0,
\label{foliate}
\end{equation}
then $[\theta]=c[d\alpha_-]\in H^2_\LL(M)$ for a constant $c\in\R$,
hence there exists a $1$-form $\rho$ on $M$ such that
\begin{equation}
   \theta=cd\alpha_-+d\rho,\qquad i_R\rho=0. 
\label{psl2neg}
\end{equation} 

{\em Case 2: }$\alpha_+\in V$ is a positive contact form. Here the
corresponding discussion is even easier. Namely, let $R$ be the Reeb
vector field of $\alpha_+$ and let $\theta$ be a closed $2$-form on $M$
satisfying equation~\eqref{foliate}. Then $\theta$ is proportional to
$d\alpha$, i.e.~$\theta=fd\alpha$ for a function $f:M\to\R$. Now the
vector field $R$ preserves both $d\alpha$ and $\theta$, so it must
preserve the function $f$. By ergodicity of $R$
(Theorem~\ref{thm:left} (b)), the function $f$ is then constant, i.e.
\begin{equation} 
\theta=cd\alpha_+
\label{psl2pos}
\end{equation}
for some $c\in \R$.
In particular, we see that in both cases the foliated cohomology
$H^2_\LL(M)$ is 1-dimensional and generated by $[d\alpha_\pm]$, so
$\kappa:H^2_\LL(M)\to H^2(M)$ is the zero map.  
We will come back to this example in Section~\ref{sec:pert} and in
Section~\ref{subsec:largeh}.

\subsection{Magnetic flows}\label{subsec:mag}

A large and interesting class of SHS arises in magnetic flows. 
For this, consider the cotangent bundle $\tau:T^*Q\to Q$ of a closed
$n$-manifold $Q$. Given a closed 2-form $\sigma$ on $Q$ (the
``magnetic field''), 
$$
   \Om = dp\wedge dq + t\tau^*\sigma
$$
defines a symplectic form on $T^*Q$. Consider on $T^*Q$ a classical
Hamiltonian of the form
$$
    H(q,p) = \frac{1}{2}|p|^2
$$
for a metric $g=|\ |$ on $Q$.
Solutions of the Hamiltonian system defined by $\Om$ and $H$ describe
the motion of a charged particle in the magnetic field $\sigma$.
For each $t>0$ the level set $M_t:=H^{-1}(t)$ is canonically
diffeomorphic (by rescaling in $p$) to the unit cotangent bundle
$S^*Q$, so $\om_t:=\Om|_{M_t}$ defines a Hamiltonian structure on
$S^*Q$. We assume that $\sigma|_{\pi_2(Q)}=0$. Then the classical
theory of Hamiltonian systems associates to this situation a {\em
  Ma\~n\'e critical value} $c=c(g,\sigma)\in(0,\infty]$ (see
e.g.~\cite{CFP}). The following facts are proved resp.~illustrated by
examples in~\cite{CFP}. 

(1) In all known examples except one ($Q=T^2$), $M_t$ is stable for
    $t<c$ and non-stable for $t=c$. For $t>c$, $M_t$ is sometimes
    stable and sometimes not. 

(2) In most known examples, $M_t$ is stable except for finitely many
    $t$. However, there exist examples in which $M_t$ is non-stable
    for all $t$ in an open interval $(a,\infty)$. 

(3) There exist examples of level sets $M_s,M_t$ with $s<c<t$ that are
    tame stable but not tame stably homotopic, see
    Section~\ref{subsec:rfh}. 

Note that if $\tau^*\sigma$ is non-exact then $M_t$ cannot be of contact
type, so magnetic flows provide a large class of stable Hamiltonian
structures that are not contact. We conclude this subsection with an
explicit example (see~\cite{CFP}). 

\begin{example}\label{ex:horo}
Let $(Q,g)$ be a closed hyperbolic surface and $\sigma$ the area form
defined by the hyperbolic metric $g$. Then the Ma\~n\'e critical value
is $c=1/2$. For $t>1/2$, $M_t$ is of contact type with positive contact form
$p\,dq$, and the Reeb flow is ergodic with periodic orbits
representing all nontrivial free homotopy classes in $T^*Q$. 
For $t<1/2$, $M_t$ is of negative contact type, i.e.~$\om_t=d\alpha_t$
for a contact form defining the opposite orientation, and all Reeb
orbits are periodic and contractible in $T^*Q$. 
$M_{1/2}$ is non-stable and has no periodic
orbits at all (the flow on $M_{1/2}$ is the famous ``horocycle flow'' and
non-stability follows from Obstruction 2 in
Section~\ref{subsec:obstr}). Note that this example can also be
recovered from Theorem~\ref{thm:left} since $S^*Q=\Gamma\setminus
PSL(2,\R)$ for a lattive $\Gamma$. 
\end{example}

\subsection{An existence result}\label{subsec:exist}

The following existence result was proved in~\cite{NW} in dimension 3;
we thank C.~Wendl for suggesting its generalization to higher
dimensions. 

\begin{proposition}\label{prop:ex}
For any $(2n-1)$-dimensional closed contact manifold $(M,\xi)$ 
and any rational cohomology class $\eta\in H^2(M;\Q)$ there exists
a closed 2-form $\om$ with $[\om]=\eta$ and a contact form $\lambda$
defining $\xi$ such that $(\om,\lambda)$ is a stable Hamiltonian
structure. In the case $n=2$ the same holds for any real cohomology
class $\eta\in H^2(M;\R)$.  
\end{proposition}

\begin{proof}
We proceed by induction on $n$.
Let $n\geq 2$ and assume the statement has been established for all
contact manifolds of dimension $2(n-1)-1$ (for $n=2$ this hypothesis
is vacuous). Let $(M,\xi)$ be a contact manifold of dimension
$2n-1$ and $\eta\in H^2(M;\Q)$. After rescaling we may assume $\eta\in
H^2(M;\Z)$. By \cite{Presas} we can represent the Poincar\'e dual of the
cohomology class $\eta$ by an embedded closed contact submanifold
$N\subset M$. We apply the induction hypothesis
to $(N,\xi|_N,\eta|_N)$. This gives us a contact form $\lambda_N$ on
$N$ and a closed 2-form $\om_N$ on $N$ representing the class
$\eta|_N$ such that $(\om_N,\lambda_N)$ is a SHS (for $n=2$ we take
any $\lambda_N$ and $\om_N=0$).  

Pick some Riemannian metric on $N$. 
Let $\pi:D_N\rightarrow N$ denote the disc bundle associated to the
normal bundle of $N$ in $M$ and let $S_N:=\p D_N$ be the
corresponding circle bundle. Note that $\eta|_N$ is the Euler
class $e(S_N)\in H^2(N;\Z)$ of the circle bundle $S_N$.  
We proceed in two steps. First, we write down a normal form for a
contact form $\lambda$ on $D_N$ 
using a connection form on $S_N$. Then we use the same connection form
to write down a Thom form for $D_N$ and add it to $d\lambda$  to
obtain the desired 2-form $\om$. 

{\bf Constructing a contact form on $D_N$. }
We can always realize a closed $2$-form 
representing the Euler class of a principal $S^1$-bundle in terms of a
connection form. Thus, there exists a connection 
$1$-form $\delta$ on $S_N$ such that $d\delta=-\pi^*\om_N$ (here the
connection form $\delta$ is normalized to have integral $1$ over the
fiber). Let $r$ be the radial coordinate in a fiber of $D_N$
normalized such that $\p D_N=\{r=1\}$. Let $\dot D_N:=D_N\setminus N$
be the punctured disc bundle. We extend the  
$1$-form $\delta$ from $S_N$ to $\dot D_N$ in an $r$-invariant way.
Then the 1-form $r^2\delta$ on $\dot D_N$ extends as zero over $N$ 
to a smooth $1$-form on $D_N$, 

\begin{lemma}\label{lem:posit}
For $\eps>0$ sufficiently small the 1-form
$$ 
   \lambda_\eps:=\pi^*\lambda_N+\eps r^2\delta
$$
defines a contact form on $D_N$. Moreover, for each compactly
supported function $F:[0,1)\to\R$ constant near $0$ there exists a
constant $K_F>0$ such that for all $K\geq K_F$ we have 
$$
   \lambda_\eps\wedge
   \Bigl(d\lambda_\eps-d[K^{-1}F(r)\delta]\Bigr)^{n-1}>0.  
$$
\end{lemma}

\begin{proof}
Let us write out 
$$
   d\lambda_\eps = 
   d\pi^*\lambda_N+\eps r^2d\delta+
   2\eps rdr\wedge\delta 
   = [\pi^*d\lambda_N-\eps r^2\pi^*\om_N] + 2\eps rdr\wedge\delta
$$ 
and
$$
   d\lambda_\eps-d[K^{-1}F(r)\delta] =
   [\pi^*d\lambda_N+(K^{-1}F(r)-\eps r^2)\pi^*\om_N] + [2\eps
   r-K^{-1}F'(r)dr]\wedge\delta. 
$$
It follows that 
$$
   \lambda_\eps\wedge\Bigl(d\lambda_\eps-d[K^{-1}F(r)\delta]\Bigr)^{n-1} =
   (n-1)\alpha_1\wedge \alpha_2, 
$$
where 
$$
   \alpha_1 := \pi^*\lambda_N\wedge\Bigl(\pi^*d\lambda_N+[K^{-1}F(r)-\eps
   r^2]\pi^*\om_N\Bigr)^{n-2} 
$$
and
$$
   \alpha_2 := [2\eps r-K^{-1}F'(r)]dr\wedge \delta.
$$
Since $(\om_N,\lambda_N)$ is a SHS on $N$ and $\lambda_N$ is contact,
for small $\eps$ and large $K$ we get 
$$
 \ker\Bigl(\pi^*d\lambda_N+[K^{-1}F(r)-\eps
   r^2]\pi^*\om_N\Bigr)
   = \pi^*\ker d\lambda_N.
$$ 
Hence the kernel of $\alpha_1$ is exactly the tangent space to the
fiber. Since $\alpha_2$ restricts to the fiber as a positive area form
for $K$ large, the lemma follows.  
\end{proof}

Note that in the preceding proof we have shown (for $F\equiv 0$)  
$$
   \ker(\pi^*d\lambda_N-\eps r^2\pi^*\om_N)
   = \pi^*\ker d\lambda_N.
$$ 
The expression for $d\lambda_\eps$ now shows that
$$
   \pi^*\ker d\lambda_N\cap \ker dr\cap \ker\delta\subset\ker
   d\lambda_\eps, 
$$
and since $\lambda_\eps$ is contact we have in fact equality:
$$
  \ker d\lambda_\eps=\pi^*\ker d\lambda_N\cap \ker dr\cap \ker\delta. 
$$

{\bf Constructing a Thom form for the bundle $D_N$.} 
Let now $F$ be any nonnegative compactly supported
nonincreasing function on $[0,1)$ with $F\equiv 1$ near $r=0$. 
Then 
$$ 
   T_F:=-d(F(r)\delta)
$$
is a Thom form on $D_N$. Set 
$$
   \om_F:=K\,d\lambda_\eps+T_F
$$
on $D_N$, for some constant $K\geq K(F)>0$ as in
Lemma~\ref{lem:posit}. Then Lemma~\ref{lem:posit} yields
$\lambda_\eps\wedge \om_F^{n-1}>0$. On the other hand, we have 
$$
  \ker d\lambda_\eps=\pi^*\ker d\lambda_N\cap \ker dr\cap \ker\delta
  \subset \ker\Bigl(-F'(r)dr\wedge\delta + F(r)\pi^*\om_N\Bigr) = \ker
  T_F, 
$$
hence $\ker d\lambda_\eps\subset \ker\om_F$. Since $\lambda_\eps$ is
contact, this shows that $(\om_F,\lambda_\eps)$ is a SHS on $D_N$. 

Having defined the forms $\lambda_\eps$, $T_F$ and $\om_F$ on $D_N$, we
now finish the argument as follows. By the contact neighbourhood
theorem (see eg.~Theorem 2.5.15 in~\cite{Ge}) there exists a
positive $r_1<1$ and an embedding  
$$
  i:\{r\le r_1\}\rightarrow M
$$ 
such that $i_*\ker\lambda_\eps=\xi$. Thus there exists a positive
$r_0<r_1$ and a contact form $\lambda$ on $M$ defining 
$\xi$ that coincides with $i_*\lambda_\eps$ on 
$i(\{r\le r_0\})\subset M$. We choose the function 
$F$ in such a way that $supp(F)\subset \{r<r_0\}$ and
thus 
$$
   supp(T_F)\subset \{r<r_0\}.
$$ 
The pushforward $i_*T_F$ extends as zero to the complement of
$i_*(\{r<r_0\})$. Finally we set
$$
 \om := K\,d\lambda+i_*T_F
$$
on $M$ (with $K$ as above). The pair $(\om,\lambda)$ then constitutes
the desired SHS and Proposition~\ref{prop:ex} is proved for rational
cohomology classes. 
\medskip

{\bf The case $n=2$. }
In the case $n=2$ the argument can be adjusted to work for any
$\eta\in H^2(M;\R)$ as follows.  
Pick an integer basis $C_1,\dots,C_k$ of the free part of
$H^2(M;\Z)$ and write $\eta=\sum_{i=1}^ka_iC_i$ with coefficients
$a_i\in\R$. Represent the Poincar\'e duals of $C_1,\dots,C_k$ in
$H_1(M;\Z)$ by disjoint embedded loops $N_1,\dots,N_k$ that
are positively transverse to $\xi$, i.e.~$\lambda(\dot N_i)>0$. Pick
a contact form $\lambda$ defining $\xi$ which is standard near the
$N_i$ and Thom forms $T_i$ near $N_i$ as above. Then
$$
   \om := K\,d\lambda+\sum_{i=1}^k a_iT_i
$$
with a suficiently large constant $K$ yields the desired SHS
$(\om,\lambda)$. 
\end{proof}

In Section~\ref{sec:openbook} we will give an alternative proof of
Proposition~\ref{prop:ex} in the case $n=2$ using open books.

\section{Stable Hamiltonian structures in dimension three}\label{sec:dim3}

\subsection{Integrability}\label{subsec:integr}

When the underlying manifold $M$ is $3$-dimensional 
there are a number of simplifications. First, $\omega$ being maximally
nondegenerate means simply that it is nowhere zero. Second, the second
condition in~\eqref{eq:def} simplifies to the following 
\begin{equation}
   d\lambda=f\omega,
\label{stabker3}
\end{equation}
where $f$ is a smooth function on $M$. 
Since the Reeb vector field $R$ of the SHS $(\omega,\lambda)$
preserves both $\omega$ and $\lambda$, 
it must also preserve the proportionality coefficient $f$ between
$d\lambda$ and $\omega$. In other words, the function $f$ is an
integral of motion for the vector field $R$.

\begin{remark}
The function $f$ already appeared in Remark~\ref{rem:Beltrami} (c) as
proportionality coefficient between $R$ and its curl. To see that this
is really the same function, write $\curl R=fR$ and compute, using the
volume form $\mu=\lambda\wedge\om$: 
$$
   d\lambda = i_{\curl R}\mu = f\,i_R\mu = f\,\om. 
$$ 
\end{remark}

If $f$ is constant we have either $f\equiv 0$ or $f\equiv c\ne 0$.  
In the first case the closed 1-form $\lambda$ defines a taut foliation,
in the second case $\lambda$ defines a contact structure and
$\om=c^{-1}d\lambda$ is exact. Regions where $f$ is nonconstant
are foliated by the level sets of $f$. The level sets of $f$ are $2$-tori,
so any connected region where $df\ne 0$ is simply an interval times $T^2$.
This motivates the following

\begin{definition}
\label{def:integr}
Let $\LL$ be a stable Hamiltonian foliation on $M$.
Let $I\subset \R$ be an interval (open, closed or half open) and $U\subset
M$ be a region diffeomorphic to $I\times T^2$ 
such that $\LL$ is a subfoliation of the foliation by $T^2$'s. 
Then $U$ is called an {\em integrable region} for $\LL$ (and for any
SHS defining $\LL$).  
\end{definition}

A Hamiltian structure $\om$ on $I\times T^2$ is called {\em
  $T^2$-invariant} if $\om$ is invariant under the
obvious action of $T^2$. Similarly  
a SHS $(\om,\lambda)$ is called {\em $T^2$-invariant} if both forms 
$\om$ and $\lambda$ are invariant under the action of $T^2$. These are
discussed in detail in Section \ref{subsec:t2inv}. For now we state
the following consequence of the Arnold-Liouville Theorem~\cite{Ar-mechanics}.  

\begin{theorem}\label{thm:integr}
Let $(\om,\lambda)$ be a SHS with Reeb vector field $R$ on an integrable region 
$I\times T^2$. Then there exists an orientation preserving diffeomorphism $\Phi$
of $I\times T^2$ preserving the tori $\{r\}\times T^2$ such that 
$\Phi^*\om$, $\Phi^*R$ and the restrictions
$\Phi^*\lambda|_{\{r\}\times T^2}$ are $T^2$-invariant. If in addition
$f=d\lambda/\om$ is constant on the tori $\{r\}\times T^2$, then 
$\Phi^*\lambda$ is $T^2$-invariant.

Moreover, the assignment $(\om,\lambda)\mapsto \Phi$ is continuous
with respect to $C^1$ topologies on the space of stable Hamiltonian
structures and the group of diffeomorphisms. 
\end{theorem}

\begin{proof} We denote by $\mathcal{E}$ the space of SHS having
$I\times T^2$ as an integrable region, equipped with the
$C^1$-topology. Consider some $(\om,\lambda)\in \mathcal{E}$. 
We denote $U:=I\times T^2$ and view the flow of the Reeb vector field
$R$ as a Hamiltonian system 
as follows. Consider the product $(-\varepsilon,\varepsilon)\times U$
with the symplectic form 
\begin{equation}
   \Om=\om+td\lambda+dt\wedge \lambda.
\label{eq:symp}
\end{equation}
Then $R$ is the Hamiltonian vector field with respect to $\Om$ of the
Hamiltonian function given by the projection onto the first factor
$$
   H:(-\varepsilon,\varepsilon)\times U\longrightarrow
   \mathbb{R},\qquad H(t,x):=t.
$$ 
Let $p$ denote the projection to the first factor of $U:=I\times
T^2$. This gives us an integral of motion for this Hamiltonian vector
field
$$
   F:(-\varepsilon,\varepsilon)\times U\longrightarrow
   \mathbb{R},\qquad F(t,x):=p(x).
$$
Vanishing of the Poisson bracket $\{H,F\}=0$ is automatic as for any
integral of motion with the Hamiltonian, so the pair $(H,F)$ gives us
a complete system of integrals on $(-\varepsilon,\varepsilon)\times U$.
Now we argue as in the proof of the Arnold-Liouville Theorem
in~\cite{Ar-mechanics}.   
Let $X_H$ and $X_F$ be the Hamiltonian vector fields on $(-\eps,\eps)\times U$ 
of $H$ and $F$. 
Both vector fields $X_H$ and $X_F$
preserve $H$ and $F$, thus the vector fileds are tangent to the tori $\{t,r\}\times T^2$.
Vanishing of $\{H,F\}$ implies vanishing of the commutator $[X_H,X_F]$. Linear independence 
of $dH$ and $dF$ at every point yields that vector fields $X_H$ and $X_F$ are pointwise linear 
independent. Altogether, the flows of $X_H$ and $X_F$ define an $\R^2$ action on 
$(-\eps,\eps)\times I\times T^2$ whose orbits are the tori $\{(t,r)\}\times T^2$.
For each torus $\{(t,r)\}\times T^2$ the corresponding stabilizer group
$\Gamma_{(t,r)}^{(\om,\lambda)}$ that leaves one (and then any) point
of $\{(t,r)\}\times T^2$ fixed is a lattice in $\R^2$. Moreover, by
the implicit function theorem, these lattices vary continuously with
$((t,r),(\om,\lambda))$, and smoothly with $(t,r)$ for  fixed
$(\om,\lambda)$. Let
$$
  \Phi(X_H,X_F,\cdot):\R^2\to\Diff_+\bigl((-\eps,\eps)\times U\bigr)
$$ 
denote the
2-flow of the pair $(X_H,X_F)$ of vector fields. Let
$i:\R^2/\Z^2\to\R^2/\Gamma_{(t,r)}^{(\om,\lambda)}$ 
be a linear orientation preserving identification varying continuously
with $((t,r),(\om,\lambda))$. Given $\tau\in T^2=\R^2/\Z^2$ let 
$\tau_{(t,r)}^{(\om,\lambda)}$ denote the its image in
$\R^2/\Gamma_{(t,r)}^{(\om,\lambda)}$ under the identification
$i$. This gives us a free $T^2$-action parametrized by
$(\om,\lambda)\in \mathcal{E}$ 
$$
   \rho:T^2\times \mathcal{E}\longrightarrow
   \Diff_+\bigl((-\eps,\eps)\times U\bigr),\qquad 
  \bigl(\tau,(\om,\lambda)\bigr)\mapsto
  \Phi(X_H,X_F,\tau_{(t,r)}^{(\om,\lambda)}) 
$$
transitive on every torus. The parametrization by $(\om,\lambda)\in
\mathcal{E}$ is continuous because the triple
$(X_H,X_F,\tau_{(t,r)}^{(\om,\lambda)})$ depends continuously on
$((t,r),(\om,\lambda))$ and $\Phi$ is a continuous function of its
arguments. Here the space of vector fields and the space of  
diffeomorphisms are given the $C^1$-topologies.

Fix a smooth section $s$ (independent of $(\om,\lambda)$) of the trivial
fibration $(-\eps,\eps)\times I\times T^2\rightarrow T^2$. We use the
action $\rho$ and the section $s$ to define an orientation preserving
self-diffeomorphism of $(-\eps,\eps)\times U$ by the formula 
$$
   \Psi(t,r,\theta,\phi):= \rho(\theta,\phi)\bigl(t,r,s(t,r)\bigr).
$$
Note that the first two coordinates of $\Psi(t,r,\theta,\phi)$ are
$(t,r)$. Pulling back the action $\rho$ we get the action
$\Psi^{-1}\rho\Psi$ which is just the standard action of
$T^2$ on $(-\eps,\eps)\times U$ by shifts in $\theta$ and $\phi$. So
given a differential-geometric object on $(-\eps,\eps)\times U$
invariant under the action of $\rho$, its pullback under $\Psi$ is
invariant under the standard $T^2$ -action (below ``$T^2$-invariant''
always means invariance under the standard $T^2$ action). 

By construction of $\rho$ the vector fields $X_H$ and $X_F$ are
$\rho$-invariant, hence $\Psi^*X_H$ and $\Psi^*X_F$ are $T^2$
invariant. In other words, $\Psi^*X_H$ and $\Psi^*X_F$ are linear on
the tori $\{t,r\}\times T^2$. The symplectic form 
$\Om$ is invariant under the flows of $X_H$ and $X_F$ and thus under
$\rho$, hence the pullback $\Psi^*\Om$ is $T^2$ invariant. 
We view $\Phi:\Psi|_{\{0\}\times U}$ as a diffeomorpsim of $U$. 
Formula \eqref{eq:symp} shows that $\Psi^*\Om|_{\{0\}\times U}=\Phi^*\om$
and $X_H|_{\{t=0\}}=R$. This shows $T^2$-invariance of $\Phi^*\om$ and
of $\Phi^*R$. Set $X:=X_F|_{\{t=0\}}$ and note that $\Phi^*X$ is
$T^2$-invariant. We contract formula \eqref{eq:symp} with $X_F$
at $t=0$ to get 
$$
   \iota_{X}\Om=\iota_{X}\om-\lambda(X)dt=-dr.
$$ 
This shows that $\lambda(X)=0$ and $\iota_{X}\om=-dr$. We view the
last two equalities as living on $U$. Pulling back with $\Phi$ yields 
$$
   (\Phi^*\lambda)(\Phi^*X)=0,\qquad\iota_{\Phi^*X}\Phi^*\om=dr.
$$ 
On the other hand, we have $\Phi^*\lambda(\Phi^*R)=\lambda(R)=1$, so
the restriction of $\Phi^*\lambda$ to a torus $\{r\}\times T^2$ is the
unique 1-form which takes value $0$ on $\Phi^*X$ and $1$ on
$\Phi^*R$. Since the vector fields $\Phi^*X$ and $\Phi^*R$ are
$T^2$-invariant, this shows $T^2$-invariance of
$\Phi^*\lambda|_{\{r\}\times T^2}$. 

To analyze invariance of $\Phi^*\lambda$, we work in standard coordinates 
$(r,\theta,\phi)$ on $(-\eps,\eps)\times T^2$ and rename $\Phi^*\om,\,
\Phi^*R,\, \Phi^*X$ and $\Phi^*\lambda$ back to $\om,\, R,\, X$ and
$\lambda$ respectively. The invariance properties allow us to write
$\om$ and $\lambda$ uniquely as 
$$
   \om = dr\wedge\bigl(k_1(r)d\theta+k_2(r)d\phi\bigr),\qquad 
   \lambda=g_1(r)d\theta+g_2(r)d\phi+g_3(r,\theta,\phi)dr
$$ 
(cf.~the proof of Lemma~\ref{lem:t2inv} below). Here $k_1,k_2,g_1,g_2$
are functions of $r$ only, but $g_3$ will in general depend on
$(r,\theta,\phi)$, see Remark~\ref{rem:integr} below.  
However, if we in addition we assume that $f=d\lambda/\om$ is constant
on the tori $r\times T^2$, then we can conclude more. Indeed, the
relation $\lambda=f\om$ writes out as  
$$
   g_1'-\frac{\p g_3}{\p\theta} = fk_1,\qquad
   g_2'-\frac{\p g_3}{\p\phi} = fk_2. 
$$
Since the functions $g_1,g_2,k_1,k_2,f$ are $T^2$-invariant, so are
$\frac{\p g_3}{\p\theta}$ and $\frac{\p g_3}{\p\phi}$, which therefore
vanish by periodicity of $g_3$. So in this case $g_3$, and thus
$\lambda$, is $T^2$-invariant.
\end{proof}

\begin{remark}\label{rem:integr}
In general, we cannot achieve in Theorem~\ref{thm:integr} that
$\Phi^*\lambda$ is $T^2$-invariant. To see this, consider
$$\lambda=r^2d\theta+(1-r^2)d\phi$$ 
on $(0,1)\times T^2$ and $$\om=d\lambda=2rdr\wedge(d\theta-d\phi).$$ 
Let $\eta:=2\sin(\theta-\phi)rdr$. Note that
$$d\eta=2\cos(\theta-\phi)(d\theta-d\phi)r\wedge
dr=\cos(\theta-\phi)\om.$$ Thus for for small $\eps>0$ the form
$\lambda_\eps$ stabilizes $\om$ and
$$f_\eps=d\lambda_\eps/\om=1+\eps\cos(\theta-\phi).$$ This function is not 
constant on the tori $\{r\}\times T^2$. So for any diffeomorphism
$\Phi$ of $(0,1)\times T^2$ preserving the tori $\{r\}\times T^2$ we
have that $\Phi^*d\lambda_\eps/\Phi^*\om=\Phi^*f_\eps$ 
is not constant on the tori $\{r\}\times T^2$. Thus the form
$\Phi^*\lambda_\eps$ cannot be $T^2$-invariant (assuming that $\Phi^*\om$ is $T^2$-invariant). 
\end{remark}

\subsection{Slope functions}\label{subsec:slope}

In standard coordinates $(r,\theta,\phi)$ on an integrable region
$I\times T^2$, the pullback $\Phi^*R$ of the Reeb vector field in
Theorem~\ref{thm:integr} has the form  
$$
   \Phi^*R=w_1(r)\p_{\theta}+w_2(r)\p_{\phi}.
$$
In particular, it is $T^2$-invariant, linear on each invariant torus
$\{r\}\times T^2$, and 
it preserves the standard area form $d\theta\wedge d\phi$. In this
subsection we associate a ``slope function'' to such a vector field. 

Consider a nowhere zero vector field $X$ on a 2-torus $T$ preserving
an area form $\sigma$. By Cartan's formula, 
$$
   \beta := i_X\sigma
$$ 
is a closed $1$-form on $T$. 

\begin{lemma}\label{lem:area}
(a) There exists a closed 1-form $\alpha$ on $T$ with
$\alpha(X)>0$. 

(b) There exist coordinates $(\theta,\phi)$ on $T\cong T^2$ in which
$X$ defines a linear foliation. In particular, the flow lines of $X$
are either all dense or all closed. 

(c) Upon replacing $\sigma$ by a different $X$-invariant area form of
the same sign, the cohomology class $[i_X\sigma]\in H^1(T;\R)$ gets
multiplied by a positive constant. 

(d) For any vector field $\tilde X$ defining the same oriented $1$-foliation as $X$ and any $\tilde X$-invariant area form $\tilde \sigma$ of the same sign as $\sigma$, 
the cohomology class $[i_{\tilde X}\tilde \sigma]\in H^1(T;\R)$ is a positive multiple of 
$[i_X\sigma]\in H^1(T;\R)$.
  
\end{lemma}

\begin{proof}
(a) By looking at the universal cover $\R^2$, one easily sees that
$\beta$ is {\em transitive}, i.e.~any two points in $T$ can be
connected by a path $\gamma$ with $\beta(\gamma)>0$. By a theorem of
Calabi~\cite{Ca}, this implies the existence of a metric on $T$ for
which $\beta$ is harmonic. Then $\alpha:=-*\beta$, where $*$ is the
Hodge star operator of the metric, has the desired properties. 

(b) Let $\alpha$ be a 1-form as in (a). The vector field $\bar
X:=\alpha(X)^{-1}X$ satisfies $\beta(\bar X)=0$ and $\alpha(\bar
X)=1$. Define a vector field $Y$ on $T$ by $\alpha(Y)=0$ and 
$\beta(Y)=1$. Then $\bar X$ and $Y$ are commuting vector fields which
are linearly independent at every point, so they are linear in
suitable angular coordinates $(\theta,\phi)$ on $T\cong T^2$. 

(c) If $f\,\sigma$ is a different $X$-invariant area form of
the same sign, $f$ is a positive function on $T$ invariant under
$X$. If all flow lines of $X$ are dense this implies that $f$ is
constant and (c) follows. Otherwise we may choose coordinates
$(\theta,\phi)$ such that $X$ points in the direction $\p_\theta$; then
$f$ is a function of $\phi$ only and $\beta$ is a $\theta$-independent
multiple of $d\phi$, so $f\,\beta$ is cohomologous to a positive
multiple of $\beta$. 

(d) Let $f$ be any positive smooth function on $T^2$ and consider the
vector field $\tilde X:=fX$. Assume that $\tilde X$ preserves an area
form $\tilde\sigma$ of the same sign as $\sigma$.
Then $i_{\tilde X}\tilde\sigma=i_Xf\tilde\sigma$ is closed. In
particular, $f\tilde\sigma$ is an $X$-invariant area form of the same
sign as $\sigma$ and the result follows from part (c). 
\end{proof}

Let $X$ and $\sigma$ be as above. 
Note that $i_X\sigma$ is nowhere vanishing, so in
particular it is not exact (if $i_X\sigma=d\chi$ for a function
$\chi\in C^{\infty}(T)$ the maximum of $\chi$ would give a zero of
$i_X\sigma$). We say that $a$ and $b$ 
in $H^1(T;\mathbb{R})$ are equivalent and write $a\sim b$ if one is
the positive multiple of the other. The quotient
$$
   PH^1(T;\mathbb{R}):=H^1(T;\mathbb{R})\setminus\{0\}/\sim
$$ 
is called the projectivization. The descendant of $a\neq 0\in
H^1(T;\mathbb{R})$ 
in the projectivization will be denoted by $Pa$. Since the cohomology
class $[i_X\sigma]$ is not zero it descends to the projectivization to
give a quantity 
$$
   k_X:=-P[i_X\sigma]\in PH^1(T;\mathbb{R})
$$ 
called the {\em slope of} $X$. Note that by Lemma~\ref{lem:area} (d),
the slope $k_X$ depends only on the oriented foliation defined by $X$.
The integral cohomology $H^1(T,\mathbb{Z})$ sits
inside $H^1(T,\mathbb{R})$ as an integer lattice; we call images of   
nonzero integral cohomology classes in the projectivization {\em
rational} points. By Lemma~\ref{lem:area} (b) we have the following
dichotomy: Either $k_X$ is rational and all the orbits of $X$
are closed, or $k_X$ is irrational and all the orbits of $X$ are dense.  

There is a more concrete coordinate way to describe the first
cohomology of $T$ and its projectivization which will enable us to
write a formula for $k_X$ in terms of $X$.
Namely, let $T\cong\mathbb{R}^2/\mathbb{Z}^2$ be coordinatized by $(\theta,\phi)$
and oriented by $d\theta\wedge d\phi$.  
The form $i_X\sigma$ can be written as 
$$
   i_X\sigma=a_1(\theta,\phi)d\theta+a_2(\theta,\phi)d\phi, \qquad
   \frac{\p a_1}{\p\phi} = \frac{\p a_2}{\p\theta}. 
$$
In the identification $H^1(T;\R)\cong\R^2$ via the basis
$[d\theta],[d\phi]$ its cohomology class is given by
$$
   [i_X\sigma]=\Bigl(\int_{T^2}a_1d\theta d\phi, \int_{T^2} a_2d\theta
   d\phi\Bigr)
$$ 
and the slope by
\begin{equation}
k_X=-P[i_X\sigma]=-|i_X\sigma|^{-1}\Bigl(\int_{T^2}a_1d\theta d\phi,
\int_{T^2} a_2d\theta d\phi\Bigr),  
\label{slopesmooth}
\end{equation}
where 
$$
   |i_X\sigma|:=\Bigl((\int_{T^2}a_1d\theta d\phi)^2+(\int_{T^2} a_2d\theta
   d\phi)^2\Bigr)^{1/2}. 
$$
Consider now $I\times T$ for an interval $I\subset\R$, open or closed
or half open. Let $X$ be a vector field on $I\times T$ tangent to the
tori $\{r\}\times T$ and preserving a volume form $V$ on $I\times T$. 
The vector field $X$ can also be viewed as a family 
$\{X_r\}_{r\in I}$ of vector fields on $T$ parametrized by
$I$. Similarly, we can write the volume form as 
$$
   V=dr\wedge\sigma_r,
$$
where $r$ is the coordinate on $I$ and $\sigma_r$ is a smooth family
of area forms on $T$ preserved by $X_r$. This gives rise to a
smooth family of $1$-forms $\{i_{X_r}\sigma_r\}_{r\in I}$ as
above, and Formula \eqref{slopesmooth} applied to $X_r$ in place of
$X$ shows that $k_{X_r}$ depends smoothly on $r$ (since $a_1$ and $a_2$ depend smoothly on $r$). 

\begin{definition}
The smooth function
$$
   k:I\longrightarrow S^1,\qquad r\mapsto k_{X_r}
$$
is called the {\em slope function} of the oriented foliation $\LL$
defined by $X$. 
\label{slopef}
\end{definition}

We note one useful simplification of Formula \eqref{slopesmooth} for
future use. Indeed, by Lemma \ref{lem:area} (b) we may assume that
the vector field $X$ is linear in coordinates 
$(\theta,\phi)$, say $X=b_1\p_\theta+b_2\partial_\phi$ for constants
$b_1,b_2\in\R$. By Lemma \ref{lem:area} (d) we 
may choose $\sigma:=d\theta\wedge d\phi$, then $a_1=-b_2$ and $a_2=b_1$. So Formula 
\eqref{slopesmooth} simplifies to 
\begin{equation}\label{slopesimple}
   k_X=(b_2,-b_1)/\sqrt{b_1^2+b_2^2}.
\end{equation}

The situation considered in this subsection arises in integrable
4-dimensional Hamiltonian systems as follows. Consider a symplectic
4-manifold $(W,\om)$ and two Poisson commuting functions $H,F$ on $W$
whose differentials are linearly independent at every point. Then each
compact connected component of a level set of $(H,F)$ is a 2-torus
$T$. Define a 2-form $\sigma$ on $T$ by $\sigma(X_H,X_F):=1$. A short
computation using $[X_H,X_F]=0$ shows that $\sigma$ is preserved by
the flows of $X_H$ and $X_F$. So we can consider the slope function of
the foliation defined by $X_H$.

For a stable Hamiltonian structure $(\om,\lambda)$ on a 3-manifold $M$
this situation arises in regions where the proportionality factor
$f=d\lambda/\om$ is nonconstant. The unique vector field
$X_f\in\ker\lambda$ defined by $i_{X_f}\om+df=0$ commutes with the
Reeb vector field $R$ and is linearly independent of $R$ where $df\neq
0$. So  $\sigma(R,X_f):=1$ defines an invariant area form on level
sets of $f$ and we can consider the slope function of the foliation
defined by the Reeb vector field $R$. Of course, this description is
related to the 4-dimensional picture by symplectization as described
in Section~\ref{subsec:integr}.

\subsection{Persistence of invariant tori}\label{subsec:persist}

\begin{theorem}
Let $(\om_0,\lambda_0)$ be a SHS on $M$ and assume that it has  
has an integrable region $K_0\cong [a,b]\times T^2\subset M$ on which
the proportionality coefficient $f_0:=d\lambda_0/\om_0:K_0\to[a,b]$ is
the projection onto the first factor and $(\om_0,\lambda_0)$ is $T^2$ invariant on $K_0$. 
Denote by $\mu$ the Lebesgue measure on $[a,b]\times T^2$. Then: 

(a) There exist constants $C,\delta_0>0$ 
such that for any $\delta<\delta_0$ and any SHS $(\om,\lambda)$ which
is $\delta$-close to $(\om_0,\lambda_0)$ in the $C^2$-metric, 
has an integrable region $K$ with $\mu(K)\ge \mu(K_0)-C\delta$. 
Moreover there exists a diffeomorpism
$\Psi$ of $M$, $C^1$-close (and thus isotopic) to the identity, such
that $\Psi([a,b]\times T^2)=K$ and the pullback SHS
$(\Psi^*\om,\Psi^*\lambda)$ is $T^2$-invariant on $[a,b]\times T^2$. 
 
(b) If in addition the slope function $k_0$ of $\LL_0$ on $K_0$ is
not constant, then $\delta_0>0$ can be chosen such that for any SHS
$(\om,\lambda)$ as in (a), the kernel foliation $\LL$ contains
rational as well as irrational invariant tori.  
\label{thm1}
\end{theorem}

\begin{proof}
The hypothesis implies that the function $f:=d\lambda/\om:K_0\to\R$ is
$C^1$-close to $f_0$, more precisely, $||f-f_0||_{C^1}< C\delta$ for
some constant $C$ independent of $\delta$. To simplify notation, we
will replace $C\delta$ by $\delta$ and drop $C$ in the following. Set
$I:=[a,b]$ and  
$$
   I_{\delta}:=[a+\delta,b-\delta],\,\,\,I_{2\delta}:=[a+2\delta,b-2\delta],
   \,\,\,I_{3\delta}:=[a+3\delta,b-3\delta].
$$  
For the first step we will only need that $||f-f_0||_{C^0}<\delta$.
This implies  
$$
   f^{-1}(I_\delta)\subset \inn(I\times T^2) 
$$
and 
$$
   I_{3\delta}\subset f(I_{2\delta}\times T^2)\subset I_{\delta}.
$$
Combining these inclusions, we obtain
$$
   f^{-1}(f(I_{2\delta}\times T^2))\subset I\times T^2,
$$ 
i.e.~if a level set of $f$ meets $I_{2\delta}\times T^2$, then it is
contained in the interior of $I\times T^2$. If $\delta$ is small
enough, then $C^1$-closeness of $f$ and $f_0$ implies that all points
in $I\times T^2$ are regular for $f$. Combining this with the above
inclusion gives us an integrable region
$$
   K:=f^{-1}(f(I_{2\delta}\times T^2))\supset I_{2\delta}\times T^2 
$$
for $\LL$ containing $I_{2\delta}\times T^2$.
Set $J:=f(I_{2\delta}\times T^2)$ and define 
$$
   \Phi: K\to J\times T^2,\qquad  
   (r,x)\mapsto\bigl(f(r,x),x\bigr).
$$
This is a diffeomorphism $C^1$-close to the identity because the function
$f$ is $C^1$-close to $f_0$, the projection onto the first factor. We
extend $f|_K$ to a submersion $\tilde f:[a,b]\times T^2\to [a,b]$
which is $C^1$-close to $f_0$ and coincides with $f_0$ near
$\p[a,b]$. Using the last displayed formula with $\tilde f$ in place
of $f$ we extend $\Phi$ to a self-diffeomorphism $\tilde\Phi$ of
$[a,b]\times T^2$ which equals the identity near the boundary. We
extend $\tilde \Phi$ to a self-diffeomorphism (still called
$\tilde\Phi$) of $M$ in the obvious way. The pushforward SHS
$\Phi_*(\om,\lambda)$ is $C^1$-close to $(\om_0,\lambda_0)$ and
$J\times T^2$ is an integrable region for $\Phi_*(\om,\lambda)$. Thus 
by Theorem~\ref{thm:integr} there exists a diffeomorphism $\Theta$ of
$J\times T^2$, $C^1$-close to the indentity, such that
$\Theta_*\Phi_*(\om,\lambda)$ is $T^2$-invariant. The diffeomorphism 
$\Theta$ extends to a diffeomorphism of $[a,b]\times T^2$ supported
away from the boundary and the latter extends to a diffeomorphism of
$M$ (called $\tilde\Theta$) in the obvious way. 
Let $\Gamma:J\times T^2\to [a,b]\times T^2$ be the mapping which is
the identity on the second factor and a linear surjection on the first
factor. Since the integrable region $[a,b]\times T^2$ can be slightly
extended, the diffeomorphism $\Gamma$ can be extended to a
self-diffeomorphism $\tilde\Gamma$ of $M$ $C^1$-close to identity. Now
the diffeomorphism
$\Psi:=(\tilde\Gamma\circ\tilde\Theta\circ\tilde\Phi)^{-1}$ has the
required properties and part (a) follows.  

For part (b) assume that the slope function $k_0$ of $\LL_0$ is not
constant on $I$. By taking $\delta_0$ small enough, we can ensure that
$k_0$ is not constant on $I_{3\delta}$. Note that $I_{3\delta}\subset
J=f(I_{2\delta}\times T^2)$ by one of the inclusions
above. In particular, $k_0$ is not constant on $J$. 
The pushforward foliation $\Phi_*\LL$ on $J\times T^2$ is the kernel
foliation of $\Phi_*(\om,\lambda)$ and thus  
tangent to the tori $\{r\}\times T^2$. Let $k\colon J\rightarrow S^1$
be the slope function of $\Phi_*\LL$ on $J\times T^2$. 
Since the
diffeomorphism $\Phi$ is $C^1$-close to the identity and the foliation
$\LL$ on $K$ is $C^0$-close to $\LL_0$, we get that the foliation
$\Phi_*\LL$ is $C^0$-close  
to $\LL_0$, so the slope function $k$ is $C^0$-close to the restriction of the slope
function $k_0$ to $J$, which is nonconstant. Thus $k$ is not constant
and so attains rational and irrational values, therefore the foliation
$\Phi_{\star}\LL$ and thus $\LL$ has rational as well as irrational
invariant tori.  
\end{proof}

As a first application of Theorem~\ref{thm1}, we illustrate the difference
between a convergent sequence of stabilizable Hamiltonian structures
and a convergent sequence of stable Hamiltonian structures. 

\begin{theorem}\label{thm:seq}
There exists a stable Hamiltonian structure $(\om_0,\lambda_0)$ on $\R
P^3$ and a sequence of stabilizable Hamiltonian structures $\om_n$ which
$C^2$-converges to $\om_0$, but such that there is no sequence of
1-forms $\lambda_n$ stabilizing $\om_n$ which $C^2$-converges to
$\lambda_0$. 
\end{theorem}  

\begin{proof}
We view $\R P^3$ as the unit cotangent bundle for the standard round
metric $g_0$ on $S^2$. The Liouville form $\alpha_0$ defines a contact
form on $\R P^3$ whose Reeb flow is the geodesic flow for $g_0$. This
is a periodic flow, so it gives $\R P^3$ the structure of a principal
$S^1$-bundle over a closed surface $\Sigma\cong S^2$ (the space of
oriented great circles on $S^2$). Let $\pi\colon R P^3\rightarrow
\Sigma$ denote the corresponding bundle projection. Note that
$\alpha_0$ defines a connection 1-form on this circle bundle and
$d\alpha_0=\pi^*\sigma$ for an area form $\sigma$ on $\Sigma$.  
We choose a $1$-form $\rho$ on $\Sigma$ such that $d\rho$ is somewhere
zero and somewhere nonzero. Then $\lambda_0:=\alpha_0+\pi^*\rho$
defines another connection form on the circle bundle, so $\lambda_0$
stabilizes the HS $\om_0:=d\alpha_0$. Since
$d\lambda_0=\pi^*(\sigma+d\rho)$, the proportionality
coefficient $f_0:=d\lambda_0/\om_0$ is given by
$(\sigma+d\rho)/\sigma$, which is nonconstant because $d\rho$ is
somewhere zero and somewhere nonzero. This produces an integrable
region $K_0$ for the kernel foliation $\LL_0$ of $\om_0$ meeting the
conditions of of Theorem \ref{thm1} (a). It follows that for any SHS
$(\om,\lambda)$ sufficiently $C^2$-close to $(\om_0,\lambda_0)$ there
persist invariant tori, and moreover, they fill a region of 
measure approximately $\mu(K_0)$. 

Now we invoke the following theorem of Katok~\cite{Ka}: For each $k\ge
2$ and $\eps,\delta>0$ there exists a Finsler metric $g$ on $S^2$,
$\delta$-close to $g_0$ in the $C^k$-norm, such that the geodesic 
flow of $g$ is ergodic on an open invariant region $U$ of the unit
cotangent bundle $S_g^*S^2$ whose complement $S_g^*S^2\setminus U$ has
measure $<\eps$. This implies that the invariant tori for the geodesic
flow of $g$ constitute a set of measure at most $\eps$. 

Now pick a sequence $g_n$ of such Finsler metrics which
$C^3$-converges to $g_0$ and such that the set of invariant tori for
the geodesic flow of $g_n$ has measure at most $\mu(K_0)/2$ for all
$n$. Let $\alpha_n$ be the contact form on $\R P^3$ obtained by
restricting the Liouville form to the unit cotangent bundle
$S_{g_n}^*S^2$. Then $\om_n:=d\alpha_n$ is a sequence of stabilizable (by
$\alpha_n$) Hamiltonian structures which $C^2$-converges to
$\om_0=d\alpha_0$. This sequence cannot be stabilized by a family of
1-forms $\lambda_n$ which $C^2$-converges to $\lambda_0$ because, by
the discussion above, this would imply that for large $n$ the set of
invariant tori for $\om_n$ would have measure bigger than
$\mu(K_0)/2$, contradicting the choice of the $g_n$. 
\end{proof}

\subsection{$T^2$-invariant Hamiltonian structures on $I\times
  T^2$}\label{subsec:t2inv}  

Let $I$ be an interval in $\R$ (open or closed or half open). 
Consider $I\times T^2$ with coordinates $(r,\theta,\phi)$ and the
$T^2$-action by shift in $(\theta,\phi)$. We orient $I\times T^2$ by
the volume form $dr\wedge d\theta\wedge d\phi$. 

{\bf Hamiltonian structures. }
For a path $h=(h_1,h_2):I\to\R^2$ consider the $T^2$-invariant 1-form 
$$
   \alpha_h := h_1(r)d\theta+h_2(r)d\phi
$$
and the $2$-form 
$$
   \om_h := d\alpha_h = h_1'(r)dr\wedge d\theta+h_2'(r)dr\wedge d\phi. 
$$
This defines a HS iff $h'(r)\neq 0$ for all $r\in I$, and in that case
its oriented kernel foliation is
\begin{equation}\label{kerfol}  
   \LL_h=Span_{\R}\{-h_2'(r)\p_\theta+h_1'(r)\p_\phi\}.
\end{equation}

Note that $\alpha_h$ is a positive contact form if and only if $h$
always turns clockwise. To see this, we viewing $\R^2$ as $\C$.
Then the contact condition $\alpha_h\wedge d\alpha_h>0$ reads as 
$\la h,ih'\ra > 0$, where $\la\cdot,\cdot\ra$ denotes the standard
scalar product on $\R^2$. Writing $h=\rho(r)e^{i\sigma(r)}$ we find
$\la h,ih'\ra=-|h|^2\sigma'(r)$, so the contact condition is 
equivalent to $\sigma'<0$. 

{\bf Slope functions. }
We define the slope function of the HS $\om_h$ on $I\times T^2$ 
as the slope function of its oriented kernel foliation
$\ker\om_h=\LL_h$ given by formula~\eqref{kerfol}.  
In terms of the function $h$, formula~\eqref{slopesimple} (with
$b_1=-h_2'$ and $b_2=h_1'$) translates to
$$
   k(r) = h'(r)/|h'(r)|. 
$$
Note that the slope function and all its properties are intrinsic to
the foliation $\LL_h$ and do not depend on the defining HS.

{\bf Winding numbers. }
We define the {\em winding number} of $h$ by
$w(h):=\sigma(b)-\sigma(a)\in\R$, where $h'(r)/|h'(r)|=e^{i\sigma(r)}$
for a function $\sigma:I=[a,b]\to\R$. Note that for immersions
$h_0,h_1$ which agree near $\p I$ we have $w(h_0)-w(h_1)\in\Z$, and
$h_0,h_1$ are regularly homotopic rel $\p I$ iff $w(h_0)=w(h_1)$. 
For a $T^2$-invariant HS $\om$ defining the foliation $\LL_h$ we
define its winding number by $w(\om):=w(h)$.  

{\bf Stabilizing 1-forms. }
For another $T^2$-invariant 1-form 
$$
   \lambda_g := g_1(r)d\theta+g_2(r)d\phi
$$
we have
$$
   d\lambda_g = g_1'(r)dr\wedge d\theta+g_2'(r)dr\wedge d\phi,\qquad
   \lambda_g\wedge\om_h = (h_1'g_2-h_2'g_1)dr\wedge d\theta\wedge d\phi. 
$$
So $\lambda_g$ stabilizes $\om_h$ iff
$$
   g_1'h_2'-g_2'h_1'=0,\qquad h_1'g_2-h_2'g_1>0. 
$$
Viewing $\R^2$ as $\C$, this can also be written as
\begin{equation}\label{eq:g}
   \la g',ih'\ra=0,\qquad \la g,ih'\ra>0. 
\end{equation}

\begin{lemma}\label{lem:t2inv}
Fix an interval $I$ and a relatively compact subinterval
$J\subset\subset I$. 

(a) Any $T^2$-invariant closed 2-form $\om$ on $I\times T^2$ can be
written as $\om=\om_h$ for a function $h:I\to\C$ which is unique up to
adding a constant. 

(b) Any 1-form $\alpha$ with $d\alpha=\om$ $T^2$-invariant
can be modified rel boundary to $\tilde\alpha$ with
$d\tilde\alpha=\om$ satisfying $\tilde\alpha|_J=\alpha_h$ for a
function $h:J\to \R$. Moreover, if $\alpha$ is $C^k$-small we can
choose $\tilde\alpha$ $C^k$-small as well. 

(c) Any $T^2$-invariant $1$-form $\lambda$ stabilizing $\om_h$ is
homotopic rel boundary through $T^2$-invariant stabilizing 1-forms to
$\tilde\lambda$ satisfying $\tilde\lambda|_J=\lambda_g$ for a
function $g:J\to \R$. 
\end{lemma}

\begin{proof}
(a) Any $T^2$-invariant $2$-form $\om$ can be written as 
$$
   \om=k_1(r)dr\wedge d\theta+k_2(r)dr\wedge d\phi+k_3(r)d\phi\wedge
   d\theta
$$
for smooth functions $k_1,k_2,k_3:I\rightarrow \R$. 
If $\om$ is closed, then $k_3=0$ and
integrating $k_1$ and $k_2$ we see that there exists a function
$h:I\to\R^2$, unique up to adding a constant, such that
$\om=d\alpha_h$. 

(b) Consider a 1-form $\alpha$ with $d\alpha=\om$ $T^2$-invariant. 
Denote by $\bar\alpha$ the $T^2$-invariant 1-form on $I\times
T^2$ obtained by averaging $\alpha$. Then $d\bar\alpha=\om=d\alpha$,
so $\alpha-\bar\alpha$ is closed. Since $\om$ vanishes on each torus
$\{r\}\times T^2$, the restriction $\alpha|_{\{r\}\times T^2}$ is
closed. Its de Rham cohomology class is preserved under averaging, so
we have $[\bar\alpha|_{\{r\}\times T^2}]=[\alpha|_{\{r\}\times T^2}]\in
H^1({\{r\}\times T^2})$. This shows that $[\bar\alpha-\alpha]=0\in
H^1(I\times T^2)$, so we can write $\bar\alpha-\alpha=df$ for
a function $f:I\times T^2\to\R$. Pick a cutoff function
$\rho:I\to[0,1]$ which equals $0$ near $\p I$ and $1$ on a
neighbourhood $K$ of $J\subset\subset I$. Then $\hat\alpha:=\alpha +
d\bigl(\rho(r)f\bigr)$ satisfies $d\hat\alpha=\om$, agrees with
$\alpha$ near $\p I\times T^2$ and is $T^2$-invariant on $K\times
T^2$. Thus $\hat\alpha|_{K\times T^2}$ can be written in the form  
$$
   h_1(r)d\theta+h_2(r)d\phi+h_3(r)dr = \alpha_h + h_3(r)dr
$$ 
for smooth functions $h_1,h_2,h_3:K\rightarrow \R$.
Pick a function $\sigma:I\to[0,1]$ which equals $0$ outside $K$ and $1$
on $J$. Then $\tilde\alpha:=\hat\alpha-\sigma(r)h_3(r)dr$ is the desired
1-form. The construction shows that $\tilde\alpha$ will be $C^k$-small
if $\alpha$ is. 

(c) Any $T^2$-invariant $1$-form $\lambda$ on $I\times T^2$ 
$$
   \lambda= g_1(r)d\theta+g_2(r)d\phi+g_3(r)dr = \lambda_g + g_3(r)dr
$$ 
for smooth functions $g_1,g_2,g_3:I\rightarrow \R$.
It stabilizes $\om_h$ if and only if $g=(g_1,g_2)$ satisfies~\eqref{eq:g}. 
Pick a function $\rho:I\to[0,1]$ which equals $1$ near $\p I$ and $0$
on $J$. Then $\tilde\lambda:=\lambda_g+\rho(r)g_3(r)dr$ is the desired
1-form and $(1-t)\lambda+t\tilde\lambda$ the desired homotopy. 
\end{proof}

In the remainder of this section we investigate the question of
stabilizability: Given an immersion $h:I\to\C$, does there exist a
function $g:I\to\C$ satisfying~\eqref{eq:g} with prescribed values near $\p I$? 

It turns out that the answer depends on the slope function
$k=h'/|h'|:I\to S^1$. 

{\bf Stabilization for constant slope. }
Let us first consider the case of {\em constant slope}, i.e.~with
constant slope function $k=h'/|h'|\in S^1$.
In this case~\eqref{eq:g} is equivalent to 
\begin{equation}\label{eq:g-lin}
   h'/|h'|\equiv k,\qquad \la g,ik\ra\equiv c>0
\end{equation}
for a constant $c>0$. Thus $h$ moves along a straight line in direction
$k$, and the function $g$ can move freely on the straight line $\la
g,ik\ra\equiv c$. Note that convex combinations of pairs
satisfying~\eqref{eq:g-lin} with the same slope $k$ (but possibly
different constants $c$) again satisfy~\eqref{eq:g-lin}. So we obtain
the following answer to the stabilizability question: 

Let $h:[0,1]\to\C$ be an immersion with constant slope $k$ and $g_0,g_1$
be functions near $0,1$ satisfying~\eqref{eq:g-lin} with constants
$c_0,c_1$. Then there exist a function $g:I\to\C$
satisfying~\eqref{eq:g} which agrees with $g_0$ near $0$ and $g_1$
near $1$ if and only if $c_0=c_1$.  

More generally, we can show a stabilization result for constant
slopes near the boundary. Fix $0<\delta<\eps$ and $\bar h,\bar
g:[0,\eps]\cup[1-\eps,1]\to\C$ satisfying
\begin{gather*}
   \bar h'/|\bar h'|\equiv k_- \text{ on }[0,\eps],\qquad 
   \bar h'/|\bar h'|\equiv k_+ \text{ on }[1-\eps,1],\cr
   \la \bar g,ik_-\ra\equiv c_- \text{ on }[0,\eps],\qquad 
   \la \bar g,ik_+\ra\equiv c_+ \text{ on }[1-\eps,1]
\end{gather*}
for unit vectors $k_\pm\in\C$ and constants $c_\pm>0$. The preceding
discussion shows that we cannot hope for a general stabilization
result unless $c_-=c_+$, which turns out to be also sufficient:

\begin{lemma}\label{lem:stablin}
Suppose $c_-=c_+$. Then there exists a continuous map that assigns to
every immersion $h:[0,1]\to\C$ with $h=\bar h$ on
$[0,\eps]\cup[1-\eps,1]$ a function $g:[0,1]\to\C$
satisfying~\eqref{eq:g} and $g=\bar g$ on
$[0,\delta]\cup[1-\delta,1]$. 
\end{lemma}

\begin{proof}
Note that a special solution to conditions~\eqref{eq:g} is given by
the formula 
\begin{equation}\label{eq:g-special}
   g(r):=ich'(r)/|h'(r)| 
\end{equation}
for some constant $c>0$. In other words, each immersion $h$ can be
stabilized by $g:I\to\C$ via~\eqref{eq:g-special}. To apply this, 
pick a function $\rho:[0,1]\to[0,1]$ which equals $0$ on
$[0,\delta]\cup[1-\delta,1]$ and $1$ on $[\eps,1-\eps]$. With 
$c:=c_-=c_+$, we define
$g$ by $(1-\rho)\bar g+\rho c i\bar h'/|\bar h'|$ on
$[0,\eps]\cup[1-\eps,1]$ and by~\eqref{eq:g-special} on
$[\eps,1-\eps]$. 
\end{proof}

\begin{remark}
The 1-form $\lambda_g$ corresponding to $g$ defined by
formula~\eqref{eq:g-special} is the Wadsley form associated to the
(suitably scaled) flat metric on $I\times T^2$.  
\end{remark}

The following corollary spells out the special case $\bar g=\bar h$. 

\begin{corollary}\label{cor:stablin}
Fix $0<\delta<\eps$, $c>0$ and unit vectors $k_\pm\in \C$. Suppose
that $\bar h:[0,\eps]\cup[1-\eps,1]\to\C$ satisfies $\bar h'/|\bar
h'|\equiv k_-$ and $\la\bar h,ik_-\ra\equiv c$ on $[0,\eps]$, and $\bar
h'/|\bar  h'|\equiv k_+$ and $\la\bar  h,ik_+\ra\equiv c$ on
$[1-\eps,1]$. 
Then there exists a smooth map that assigns to
every immersion $h:[0,1]\to\C$ with $h=\bar h$ on
$[0,\eps]\cup[1-\eps,1]$ a function $g:[0,1]\to\C$
satisfying~\eqref{eq:g} and $g=\bar h$ on
$[0,\delta]\cup[1-\delta,1]$. 
\end{corollary}

One useful example to which the corollary applies is $\bar h(r)=\pm(r^2,1-r^2)$. 

\begin{remark}\label{rem:nonconst}
Moving the function $g$ on the line we can achieve that $g\equiv
const$ on some subinterval $J\subset I$. 
Then we can perturb $h$ slightly on $J$, so that $(h,g)$ still
satisfies \eqref{eq:g} and the slope $h'/|h'|$ is nonconstant on $J$.
\end{remark}  

{\bf Stabilization for nonconstant slope. }
We have seen that for constant slope there is an obstruction to
stabilizability. On the other hand, we will show that stabilization is
always possible if the slope function is nonconstant:

\begin{prop}\label{prop:t2inv-stab}
Let $h:[0,1]\to\C$ be an immersion such that $h'/|h'|$ is
not constant on $[\eps,1-\eps]$. Let $\bar
g:[0,\eps]\cup[1-\eps]\to\C$ be given such that $(h,\bar g)$
satisfies~\eqref{eq:g} on $[0,\eps]\cup[1-\eps,1]$. Then there exists
a function $g:[0,1]\to\C$ which agrees with $\bar g$ on
$[0,\eps]\cup[1-\eps,1]$ such that $(h,g)$ satisfies~\eqref{eq:g}. 
\end{prop}

The proof is based on two lemmata. 
The first lemma gives a criterion when a $T^2$-invariant
Hamiltonian perturbation of a $T^2$-invariant SHS is again
stabilizable. It will play a crucial role in
Section~\ref{sec:structure}. 

\begin{lemma}\label{lem:g-pert}
(a) Let $\bar h,\bar g:[0,1]\to\C$ satisfy~\eqref{eq:g} and suppose that
$\bar h'/|\bar h'|$ is not constant on $[\eps,1-\eps]$. Then for every
$h:[0,1]\to\C$ sufficiently $C^1$-close to $\bar h$ and
$g_0:[0,\eps]\cup[1-\eps,1]\to\C$ sufficiently $C^1$-close to
$\bar g|_{[0,\eps]\cup[1-\eps,1]}$ such that $(h,g_0)$
satisfies~\eqref{eq:g} on $[0,\eps]\cup[1-\eps,1]$ 
there exists $g:[0,1]\to\C$ which is $C^1$-close to $\bar g$ and
agrees with $g_0$ on $[0,\eps]\cup[1-\eps,1]$ such that $(h,g)$
satisfies~\eqref{eq:g}. 
Moreover, for fixed $\bar g$ the assignment $(h,g_0)\mapsto g$ works
smoothly in families. 

(b) Let $C^1$-small $\xi:[0,1]\to \C$ be given. Set $g_0=\bar g$ and
consider the family $h_t=\bar h+t\xi$, $t\in[0,1]$. Then for the
corresponding functions $g_t$ from (a) we have an estimate $|\dot
g_t|_{C^1}\leq C|\xi|_{C^1}$, for some constant $C>0$ and all
$t\in[0,1]$.  
\end{lemma}

\begin{proof}
(a) By the first
condition in~\eqref{eq:g} we have $\bar g'(r)=\bar\rho(r)\bar h'(r)$ for a
function $\bar\rho:[0,1]\to\R$, and $g_0'(r)=\rho_0(r)h'(r)$ for a
function $\rho_0:[0,\eps]\cup[1-\eps,1]\to\R$ which is $C^0$-close to
$\bar\rho|_{[0,\eps]\cup[1-\eps,1]}$. Let us extend $\rho_0$ to a
function $\rho_0:[0,1]\to\R$ which is $C^0$-close to
$\bar\rho$. We look for $g$ satisfying
$g'(r)=\rho(r)h'(r)$ for a function $\rho:[0,1]\to\R$ which agrees
with $\rho_0$ on $[0,\eps]\cup[1-\eps,1]$, so that the
first condition in~\eqref{eq:g} holds. Given such $\rho$ we set
$$
   g(r) := g_0(0) + \int_0^r\rho(s)h'(s)ds. 
$$
This agrees with $g_0$ on $[0,\eps]$, and it agrees with $g_0$ on
$[1-\eps,\eps]$ iff
$$
   \int_0^1\rho(s)h'(s)ds = g_0(1)-g_0(0). 
$$
Denote by $V$ the space of smooth functions $\sigma:[0,1]\to\R$ with
support in $[\eps,1-\eps]$, equipped with the $C^0$ norm. Writing
$\rho=\rho_0+\sigma$, the preceding condition is equivalent to 
$$
   \delta:=g_0(1)-g_0(0)-\int_0^1h'(s)\rho_0(s)ds
$$
being in the image of the linear map 
$$
   L_{h'}:V\to\C,\qquad \sigma\mapsto\int_0^1\sigma(s)h'(s)ds. 
$$
Now a unit vector $v\in\C$ is orthogonal to the image of $L_{h'}$ iff $\la
k'(s),v\ra=0$ for all $s\in[\eps,1-\eps]$, which can only happen if
the slope $h'/|h'|$ is
constant on $[\eps,1-\eps]$. But this is not the case since by
assumption the slope $\bar h'/|\bar h'|$ is not constant on
$[\eps,1-\eps]$ and $h'$ is $C^0$-close to $\bar h'$. Thus $L_{h'}$ is
surjective. Closer inspection shows
that we find a right inverse for $L_{h'}$  whose norm is uniformly bounded
for $h'$ in a $C^0$-neighbourhood of $\bar h'$. Indeed, let $K\subset
V$ be the kernel of the operator $L_{\bar h'}$ and $T$ be a
$2$-dimensional algebraic complement, so $L_{\bar h'}|_T:T\to\C$ is an
isomorphism. Since the map $h'\mapsto L_{h'}$ is  continuous with respect
to the operator norm, $L_{h'}|_T$ is invertible with uniformly bounded
inverses for $h'$ in a $C^0$-neighbourhood of $\bar h'$. Thus
$(L_{h'}|_T)^{-1}:\C\to T\into V$ are the required right inverses of
$L_{h'}$.  

Note that, since the pair $(g_0,\rho_0)$ is $C^0$-close to the pair $(\bar g,\bar\rho)$,
the complex number $\delta$ is small and hence we find $g$ which is
$C^1$-close to $\bar g$ such that the first condition in~\eqref{eq:g}
holds. The second condition in~\eqref{eq:g} now follows from
$C^1$-closeness. 

For (b) note that $g_t$ is defined by 
$$
   g_t(r) := \bar g(0) + \int_0^r\rho_t(s)h_t'(s)ds, \qquad
   h_t(s)=\bar h(s)+t\xi(s),\qquad 
   \rho_t(s)=\rho_0(s)+\sigma_t(s). 
$$
Thus
$$
   \dot g_t(r) :=
   \int_0^r[\dot\sigma_t(s)h_t'(s)+\rho_t(s)\xi'(s)]ds. 
$$
Since $h_t'$ is $C^0$-bounded and $\xi'$ is $C^0$-small, for
$C^1$-smallness of $\dot g_t$ it remains to show $C^0$-smallness of
$\dot\sigma_t$. 

Note that $\sigma_t$ is defined by 
$$
   L_{h_t'}(\sigma_t) = L_{\bar h'}(\sigma_t)+tL_{\xi'}(\sigma_t) =
   \delta_t
$$
with 
$$
   \delta_t = \bar g(1)-\bar g(0)-\int_0^1(\bar
   h+t\xi)'(s)\rho_0(s)ds, \qquad
   L_{\xi'}(\sigma_t) = \int_0^1\sigma_t(s)\xi'(s)ds.  
$$
It follows that
$$
   L_{h_t'}(\dot\sigma_t) + L_{\xi'}(\sigma_t) = \dot\delta_t =
   -\int_0^1\xi'(s)\rho_0(s)ds.
$$
Since $\xi'$ is $C^0$-small, we see that $\dot\delta_t$ is small (and
independent of $t$). Since the operators $L_{h_t'}$ have uniformly
bounded right inverses, it follows that 
$$
   \dot\sigma_t = (L_{h_t'})^{-1}\Bigl(\dot\delta_t -
   L_{\xi'}(\sigma_t)\Bigr)  
$$
is uniformly $C^0$-small for all $t\in[0,1]$. This finishes the proof
of Lemma~\ref{lem:g-pert}. 
\end{proof}

\begin{lemma}\label{lem:stab-interpol}
Let $h,g_0,g_1:[0,1]\to\C$ be given such that the pairs $(h,g_0)$ and
$(h,g_1)$ both satisfy~\eqref{eq:g}. Suppose that
$h'/|h'|$ is not constant. Then there exists a function
$g:[0,1]\to\C$ which agrees with $g_0$ near $0$ and with $g_1$ near
$1$ such that $(h,g)$ satisfies~\eqref{eq:g}. 
\end{lemma}

\begin{proof}
By assumption, there exists a point $p\in(0,1)$ at which the slope
function $k=h'/|h'|$ has nonzero derivative. We pick a small interval
$I\subset(0,1)$ on which $k'\neq 0$, so the unit vectors $k_0=k(0)$
and $k_1=k(1)$ are linearly independent and close to each other, and
rescale $I$ back to $[0,1]$. We set $c_j=\la g_j(j),ik_j\ra$ for $j=0,1$.
Linear independence of $k_0$ and $k_1$ implies that
there exists a unique pair $(s_0,s_1)\in \R^2$ such that 
\begin{equation}\label{eq:intersect}
  g_0(0)-g_1(1)=s_1k_1-s_0k_0
\end{equation}
We pick a small $\eps>0$ and define the following piece-wise constant function $\sigma$ 
on $[0,1]$: on $[0,\eps)$ we set $\sigma=\eps^{-1}s_0$, on $[\eps,1-\eps]$ we set $\sigma=0$ 
and finally on $(1-\eps,1]$ we set  $\sigma=-\eps^{-1}s_1$. We set 
$$
  g(r):=g_0(0)+\int_0^r\sigma(s)k(s)ds.
$$  
We claim the following:
\begin{itemize}
\item [(i)] $g(1)-g_1(1)$ is small and 
\item [(ii)] the pair $(h,g)$ satisfies \eqref{eq:g}.
\end{itemize}
For (i) we introduce the averages
$A_0:=\eps^{-1}\int_0^{\eps}k(s)ds,\,
A_1:=\eps^{-1}\int_{1-\eps}^1k(s)ds$ 
and note that $\lim_{\eps\to 0}A_0=k_0$ and $\lim_{\eps\to 
  0}A_1=k_1$. Using the explicit formula for $g$ we write  
$$
  g(1)-g_1(1)=g_0(0)-g_1(1)+s_0A_0-s_1A_1.
$$
This together with the above convergence properties of $A_0$ and $A_1$
and equation~\eqref{eq:intersect} implies (i) for small enough
$\eps$. 

For (ii) we first consider $r\in [0,\eps]$ and write out the estimate 
\begin{align*}
   |\la g(r)-g_0(0),ik(r)\ra|
   &\le |s_0|\eps^{-1}\int_0^\eps |\la k(s),ik(r)\ra| ds \cr
   \le |s_0|max_{r_1,r_2\in [0,\eps]}|\la k(r_1),ik(r_2)\ra|.
\end{align*}
Note that the latter maximum tends to zero as $\eps\to 0$. This shows
that if $\eps$ is small enough, then $\la g(r),ik(r)\ra$ is close to
$\la g_0(0),ik(r)\ra$ and the latter is close to  
$\la g_0(0),ik_0\ra=c_0>0$. Altogether, $\la g(r),ik(r)\ra$ is close
to $c_0$ and thus in particular  
positive. Similarly for $r\in [1-\eps,1]$, we write
$g(r)-g(1)=\eps^{-1}s_1\int_r^1k(s)ds$, 
use that $g(1)$ is close to $g_1(1)$ and that $max_{r_1,r_2\in
  [1-\eps,1]}|\la k(r_1),ik(r_2)\ra|$ 
tends to zero as $\eps\to 0$ to get that $\la g(r),ik(r)\ra$ is close
to $c_1$ and thus in particular 
positive. For the intermediate region $r\in [\eps,1-\eps]$ recall that
$k'\neq 0$ and thus 
any $k(r)$ ``sits between'' $k(\eps)$ and $k(1-\eps)$. Since $g$ is
constant on $[\eps,1-\eps]$, the number $\la g(r),ik(r)\ra$ belongs to
the interval bounded by 
$\la g(\eps),ik(\eps)\ra$ and $\la g(1-\eps),ik(1-\eps)\ra$,
which implies the required positivity.

Now approximate $\sigma$ by an $L^1$-close smooth function (still
denoted by $\sigma$) and define $g$ as above. The new function $g$ is
smooth and still satisfies conditions (i) and (ii). Recall that $k'$
is nonconstant. Now condition (i) and the equality $g(0)=g_0(0)$ allow
us to use Lemma~\ref{lem:g-pert} to adjust $\sigma$ such that $g=g_0$
near $0$ and $g=g_1$ near $1$. 
\end{proof}

\begin{proof}[Proof of Proposition~\ref{prop:t2inv-stab}]
We extend $\bar g$ from $[0,\eps]$ to a larger interval $[0,r_0]$ such
that~\eqref{eq:g} still holds and the derivative of the slope function
$k=h'/|h'|$ is nonzero at $r_0$. (If $k$ is not constant near
$\eps$ this can be done for $r_0$ slightly larger than $\eps$, and
if $k$ is constant near $\eps$ we extend $\bar g$ satisfying
equation~\eqref{eq:g-lin} until $k$ becomes nonconstant). Similarly,  
we extend $\bar g$ from $[1-\eps,1]$ to a larger interval $[r_1,1]$ such
that~\eqref{eq:g} still holds and $k'(r_1)\neq 0$. Now we use
Lemma~\ref{lem:stab-interpol} to find $g:[0,r_0]\cup[r_1,1]\to\C$
satisfying~\eqref{eq:g} which agrees with $\bar g$ on
$[0,\eps]\cup[1-\eps,1]$ and with $ik$ near $r_0$ and $r_1$, so we can
extend it as $ik$ over $[r_0,r_1]$ to the desired function
$g:[0,1]\to\C$.  
\end{proof}

Proposition~\ref{prop:t2inv-stab} answers the stabilization question
for a single function $h:[0,1]\to\C$. The  following proposition
answers the question for homotopies. 

\begin{proposition}\label{prop:stabhom}
Let $h_t:[0,1]\to \C$, $t\in [a,b]$ be a 
homotopy of immersions such that for each $t$ the slope $h_t'/|h_t'|$
restricted to $[\eps,1-\eps]$ is nonconstant. Let $\bar
g_t:[0,\eps]\cup [1-\eps,1]$, $t\in [a,b]$ be a homotopy such that
$(h_t,\bar g_t)$ satisfies~\eqref{eq:g} on $[0,\eps]\cup [1-\eps,1]$
for all $t\in [a,b]$. Then there exists a homotopy $g_t$ which agrees
with $\bar g_t$ on $[0,\eps]\cup [1-\eps,1]$ such that $(h_t,g_t)$
satisfies~\eqref{eq:g} on $[0,1]$ for all $t\in [a,b]$.
\end{proposition}

\begin{proof} 
We apply Proposition~\ref{prop:t2inv-stab} for each time $s\in [a,b]$
to get a (possibly discontinuous) family $\{\tilde g_s\}_{s\in [a,b]}$
stabilizing $h_s$ on $[0,1]$ (i.e.~$(h_s,\tilde g_s)$
satisfies~\eqref{eq:g}) and restricting to $[0,\eps]\cup [1-\eps,1]$
as $\bar g_s$. Lemma~\ref{lem:g-pert} applied with $\bar g:=\tilde
g_s$ shows that each $s\in [a,b]$ has an open neighbourhood $U\subset [a,b]$ and a
smooth family $g^U_t$, $t\in U$ such that $(h_t,g^U_t)$
satisfies~\eqref{eq:g} and $g^U_t|_{[0,\eps]\cup [1-\eps,1]}=\bar g_t$
for all $t\in U$. Finitely many such neighbourhoods $U_i$ cover
$[a,b]$. Let $\{\rho_i\}$ be a finite partition of unity subordinate
this covering and set $g_t:=\sum_i\rho_ig^{U_i}_t$. This family is
smooth as a function of $t\in [a,b]$. Moreover, for each $t$ it is a
finite sum of functions stabilizing $h_t$, thus the pair $(h_t,g_t)$
satisfies~\eqref{eq:g}. On $[0,\eps]\cup [1-\eps,1]$ we have
$g_t=\sum_i\rho_i\bar g_t=\bar g_t$. This completes the proof.
\end{proof}

\begin{corollary}\label{cor:class}
Let $h_0,\, h_1:[0,1]\to \C$ be two immersions with with $h_0=h_1$ on
$[0,\eps]\cup [1-\eps,1]$ and the same winding number. Let
$g_0:[0,\eps]\cup [1-\eps,1]\to\C$ be such that $(h_0,g_0)$ satisfies
\eqref{eq:g}. Then there exists a homotopy $(h_t,g_t)$, $t\in [0,1]$,
satisfying ~\eqref{eq:g} and fixed on $[0,\eps]\cup [1-\eps,1]$, such
that $h_t$ connects $h_0$ and $h_1$.
\end{corollary}

\begin{proof}
Since the immersions $h_0,h_1$ have the same winding number they are
homotopic through immersions $h_t$, $i\in[0,1]$, fixed on
$[0,\eps]\cap[1-\eps,1]$. If $h_0$ has constant slope on
$[\eps,1-\eps]$ we use Remark~\ref{rem:nonconst} to make it
nonconstant, and similarly for $h_1$. Since the condition to have
constant slope on $[\eps,1-\eps]$ is of infinite codimension, we can
choose the homotopy of immersions $h_t$ to have nonconstant slope on
$[\eps,1-\eps]$ for all $i\in[0,1]$. Now the stabilizing family $g_t$
is provided by Proposition~\ref{prop:stabhom}. 
\end{proof}

\begin{remark}\label{rem:contractible}
Clearly, Proposition~\ref{prop:stabhom} remains true with the interval
$[a,b]$ replaced by any compact manifold with boundary. Hence the
proof of Corollary~\ref{cor:class} shows that the space of pairs
$h,g:[0,1]\to\C$ satisfying~\eqref{eq:g}, fixed on
$[0,\eps]\cup[1-\eps,1]$ and with fixed winding number, is weakly
contractible.   
\end{remark}

\subsection{$T^2$-invariant Hamiltonian structures on $T^3$ and
  $S^3$}\label{subsec:t3} 

To illustrate the techniques developed in Section~\ref{subsec:t2inv},
we now classify $T^2$-invariant Hamiltonian structures on $T^3$ and
$S^3$ up to $T^2$-invariant stable homotopy.

We begin with the 3-torus $T^3=\R^3/\Z^3$ with coordinates
$(r,\theta,\phi)$. We let $T^2$ 
act on $T^3$ via shift along $\theta$ and $\phi$ and consider
$T^2$-invariant Hamiltonian structures on  
$T^3$. Any $T^2$-invariant Hamiltonian structure $\om$ on $T^3$ can be
written as $\om=\om_h$, where $h:\R\to\C$ is an immersion with
periodic $h'$. Now $T^2$-invariant HS on $T^3$ are classified up to
$T^2$-invariant homotopy by the winding number $w(h)\in\Z$, i.e.~the degree
of the map $h'/|h'|:S^1\to S^1$. Since every homotopy $\om_{h_t}$ of
$T^2$-invariant HS on $T^3$ can be stabilized by $\lambda_{g_t}$ with $g_t$
obtained from $h_t$ by formula~\eqref{eq:g-special}, we have shown

\begin{cor}
Two $T^2$-invariant SHS on $T^3$ are connected by a $T^2$-invariant
stable homotopy if and only if they have the same winding number. 
\end{cor}

Next we consider the 3-sphere 
$$
   S^3=\{(x_1,y_1,x_2,y_2)\in \mathbb{R}^4|x_1^2+y_1^2+x_2^2+y_2^2=1\}.
$$
We introduce the radial coordinate $r\in[0,1]$ and two angular coordinates
$\phi$ and $\theta$ on $S^3$ as follows:  
$$
   r:=(x_1^2+y_1^2)^{1/2},\,\,\,  \tan\theta=\frac{y_1}{x_1},\,\,\,
   \tan\phi=\frac{y_2}{x_2}.
$$  
Let $T^2$ act on $S^3$ by rotations in $(x_1,y_1)$ and $(x_2,y_2)$
planes, i.e.~by shifts along $\theta$ and $\phi$.  
The goal of this subsection is to classify $T^2$-invariant SHS up to
$T^2$-invariant stable homotopy.  

Consider first a $T^2$-invariant HS $\om$ on $S^3$. Its Reeb vector
field $R$ is tangent to the level sets $r=\const$, in particular $\{r=0\}$
and $\{r=1\}$ are closed Reeb orbits. Define two signs $s_0,s_1$ by
$s_0=+$ iff the orientations on $\{r=0\}$ induced by $R$ and $\p_\phi$
coincide, and 
$s_1=+$ iff the orientations on $\{r=1\}$ induced by $R$ and $\p_\theta$
coincide. Clearly, these signs remain constant during a homotopy of
$T^2$-invariant HS. 

Recall from Lemma~\ref{lem:t2inv} that $\om$ can be written on
$\{0<r<1\}$ as $\om=\om_h=h_1'(r)dr\wedge d\theta+h_2'(r)dr\wedge
d\phi$ for an immersion $h:(0,1)\to\C$ which is unique up to adding a
constant. Smoothness and nondegeneracy at $r=0$ imply that
$h_1'(r)=ar+O(r^2)$ near $r=0$ for some $a\neq 0$. Note that the sign
of $a$ is $s_0$. 
Thus we can $T^2$-invariantly homotope
$\om=\om_h$ until $h'(r)=(s_0r,0)$ near $r=0$, i.e.~$\om$ agrees with $s_0r\,dr\wedge d\theta$ near $r=0$. Note that the orientation preserving 
diffeomorphism of $S^3$ mapping $x_1\mapsto x_2,\, x_2\mapsto x_1,\, y_1\mapsto y_2$ and $y_2\mapsto y_1$ exchanges $\theta$ with $\phi$ and $r^2$ with $1-r^2$. In particular 
$rdr$ pulls back to $-rdr$. Thus, arguing symmetrically we can further $T^2$-invariantly homotope $\om=\om_h$ until it agrees with  
$-s_1r\,dr\wedge d\phi$ near $r=1$. After this
standardization, we define the winding number 
as in Section~\ref{subsec:t2inv}. Then two $T^2$-invariant HS are
$T^2$-invariantly homotopic iff they have the same signs and winding
number.  

Next consider a $T^2$-invariant SHS $(\om,\lambda)$ on $S^3$. 
After a $T^2$-invariant stable homotopy as in 
Remark~\ref{rem:nonconst}, we may assume that the slope function
$h'/|h'|$ is nonconstant. By definition of the signs $s_0,s_1$, the HS
$\om$ is stabilized by the 1-form $s_0d\phi$ near $r=0$ and by
$s_1d\theta$ near $r=1$. By Proposition~\ref{prop:t2inv-stab}, there
exists a $T^2$-invariant stabilizing 1-form $\tilde\lambda$ for $\om$
which agrees with $s_0d\phi$ near $r=0$ and with $s_1d\theta$ near
$r=1$. Fixing the stabilizing 1-form $\tilde\lambda$, we can now
homotope $\om$ near $r=0,1$ to $\tilde\om$ which agrees with
$s_0r\,dr\wedge d\theta$ near $r=0$ and with $-s_1r\,dr\wedge d\phi$
near $r=1$.

Finally, consider two $T^2$-invariant SHS $(\om_i,\lambda_i)$, $i=0,1$
on $S^3$ with the same signs $s_0,s_1$ and the same winding
number. After applying the stable homotopies in the previous
paragraph, we may assume that both $(\om_i,\lambda_i)$ agree with 
$(s_0r\,dr\wedge d\theta,s_0 d\phi)$ near $r=0$ and with
$(-s_1r\,dr\wedge d\phi,s_1d\theta)$ near $r=1$. By assumption, $\om_0$
and $\om_1$ have the same winding number. Hence, by
Corollary~\ref{cor:class}, $(\om_0,\lambda_0)$ and $(\om_1,\lambda_1)$
are connected by a $T^2$-invariant stable homotopy fixed near
$r=0,1$. So we have shown

\begin{cor}
Two $T^2$-invariant SHS on $S^3$ are connected by a $T^2$-invariant
stable homotopy if and only if they have the same signs and winding
number.  
\end{cor}

\begin{remark}
Similar arguments (cf.~Remark~\ref{rem:contractible}) show
that the space of $T^2$-invariant SHS on $T^3$ (resp.~$S^3$) with
fixed winding number (resp.~signs and winding number) is weakly
contractible. 
\end{remark}

\subsection{Thickening a level set}

In this subsection we prove the following technical result which will
play a crucial role in the remainder of this paper.

\begin{prop}\label{prop:thick}
Let $(\om,\lambda)$ be a SHS on a closed $3$-manifold $M$ and set
$f:=d\lambda/\om$. Let $Z\subset\R$ be any set of Lebesgue measure
zero containing a value $a\in Z\cap\im(f)$. 
Then there exists a stabilizing form $\tilde\lambda$ for $\om$ such
that $\tilde f:=d\tilde\lambda/\om$ can be written as 
$\tilde f=\sigma\circ f$ for a function $\sigma:\R\to\R$ 
which is locally constant on a open
neighbourhood of $Z$ (and thus $\tilde f$ is locally constant on an
open neighbourhood of $f^{-1}(Z)$) and $\tilde f\equiv a$ on an
open neighbourhood of $f^{-1}(a)$.  
Moreover, for every $s\in[0,1)$ we can achieve that $\tilde\lambda$ is
$C^{1+s}$-close to $\lambda$. 
\end{prop}

We will use two special cases of this result. The first one is
$Z=\{a\}$ for some value $a$ of $f$:
 
\begin{cor}\label{cor:thick-a}
Let $(\om,\lambda)$ be a SHS on a closed $3$-manifold $M$ and set
$f:=d\lambda/\om$. Let $a\in\im(f)$ be any (singular or regular)
value of $f$.  
Then there exists a stabilizing form $\tilde\lambda$ for $\om$ such
that $\tilde f:=d\tilde\lambda/\om$ satisfies $\tilde f\equiv a$ on an
open neighbourhood of $f^{-1}(a)$.
\end{cor}

The second special case arises for $Z$ the set of critical values:

\begin{corollary}\label{cor:thick-Z}
Let $(\om,\lambda)$ be a SHS on a closed $3$-manifold $M$ and set
$f:=d\lambda/\om$. 
Then there exists a (possibly disconnected and possibly with boundary)
compact $3$-dimensional submanifold $N$ of $M$, invariant under the Reeb flow,
a finite family $\{U_i\}_{i=1,...,k}$ of disjoint open integrable regions
and a stabilizing 1-form $\tilde\lambda$ for $\om$ with the following
properties: 
\begin{itemize}
\item $\cup_iU_i\cup N=M$;
\item $\tilde\lambda$ is $C^1$-close to $\lambda$;
\item the proportionality coefficient 
$\tilde f:=d\tilde\lambda/\om$ is constant on each connected component of $N$; 
\item on each $U_i\cong (a_i,b_i)\times T^2$ the function $f$ is given
  by the projection onto 
the first factor and for $r$ sufficiently close to $a_i$ or $b_i$ we
have $\{r\}\times T^2\subset N$.  
\end{itemize}
Moreover, $\tilde f=\sigma\circ f$ for a function $\sigma:\R\to\R$
which is $C^0$-close to the identity. 
\end{corollary}

\begin{proof}
Let the 1-form $\tilde\lambda$ and the proportionality coefficient 
$\tilde f=d\tilde\lambda/\om=\sigma\circ f$ be obtained from
Proposition~\ref{prop:thick} with $Z$ the set of critical values of
$f$ (and any value $a\in Z$). Since $Z$ is compact, it is covered by
finitely many open intervals $(c_j,d_j)$, $j=1,...,n$ on which
$\sigma$ is constant. By shrinking these intervals slightly if
necessary we may assume that all $c_j,d_j$ are regular values of
$f$. Set $J:=\cup_{j=1}^n[c_j,d_j]$ and 
$N:=f^{-1}(J)$. Let $d$ be the minimal  
distance between any two of the intervals $[c_j,d_j]$. Let
$I=\cup_{i\in\N}(a_i,b_i)$ be the set of regular values of $f$,
written as a countable union of disjoint intervals. As the $a_i,b_i$
are singular values, those intervals $(a_i,b_i)$ with length
$b_i-a_i$ strictly smaller than $d$ must be contained in one of the
intervals $[c_j,d_j]$. After renumbering we may asume that $(a_i,b_i)$
for $i=1,...,m$ are all the intervals from $I$ with finite length
greater or equal than $d$. Set $I_d:=\cup_{i=1}^m(a_i,b_i)$. It is
clear that the image of $f$ is contained in $I_d\cup J$. 
Thus setting $U:=f^{-1}(I_d)$ gives us that $N\cup U=M$. The connected
components $U_i$ of $U$ are diffeomorphic to $(a_i,b_i)\times T^2$
with $f$ being the obvious projection. 
Finally, note that each $a_i$ and $b_i$ is itself a {\em singular}
value of $f$ and thus must be contained in $(c_j,d_j)$ for some $j$. 
\end{proof}

The proof of Proposition~\ref{prop:thick} consists in analysis of
suitable linear function spaces. We 
begin with a lemma about functions on the real line. 

\begin{lemma}\label{lm:appro}
Let $\sigma$ be a smooth real valued function on an interval $[c,d]$. 
Let $Z\subset[c,d]$ be a set of Lebesgue measure zero, $a\in Z$, and
$s\in[0,1)$. Then there exists a sequence of smooth functions
$\sigma_n$ on $[c,d]$ converging to $\sigma$ in the $C^s$-norm such
that each $\sigma_n$ is locally constant on an open neighbourhood of
$Z$ and $\sigma_n(a)=\sigma(a)$ for all $n$. 
\end{lemma}

\begin{proof}
Since $Z$ has Lebesgue measure $0$ it has an open neighbourhood
$U$ of arbitrarily small Lebesgue measure. We can represent $U$ as a
countable union of intervals $U=\cup_{i\in\N}U_i$. We can 
assume that all $U_i$ are disjoint by replacing any two of them which
intersect by their union. 
Moreover, we can assume that the distance between any two of them is
positive by replacing  
any pair $(x_1,x_2)$, $(x_2,x_3)$ of intervals by the interval
$(x_1,x_3)$. 
Then we can choose a set of intervals $V_i$ with the following
properties. For each $i$ we have $\bar U_i\subset V_i$, all the $V_i$
have positive distance from each other, and the measure of
$V:=\cup_{i\in\N} 
V_i$ does not exceed twice the measure of $U$. For each $i$ 
we choose a compactly suported cutoff function $\chi_i:V_i\rightarrow
[0,1]$ which equals $1$ on $U_i$. These cutoff
functions patch together to a cutoff function $\chi$ which is
compactly supported in $V$ and equals $1$ on $U$.
 
We set $h_U:=(1-\chi)\sigma'$. By construction $h_U$ is a smooth
function vanishing on $U$. 
Now $\sigma(x)=\sigma(a)+\int_a^x\sigma'(y)dy$.
Set $\sigma_U(x):=\sigma(a)+\int_a^xh_U(y)dy$. This defines a smooth
function on $[c,d]$ with $\sigma_U(a)=\sigma(a)$. Note that
$\sigma'$ and $\sigma_U'=(1-\chi)\sigma'$ are both bounded and differ
only on the support of $\chi$, which is the set of small measure. Thus  
$\sigma'$ and $\sigma_U'$ are $L^p$-close to each other for any
$0<p<\infty$. Next set $C:=max_{y\in [c,d]}|\sigma'(y)|$  
and compute 
$$
   |\sigma(x)-\sigma_U(x)|
   = |\int_a^x(\sigma'(y)-h_U(y))dy|
   \le \int_{[c,d]}|\chi(y)||\sigma'(y)|dy
   \le C|V|.
$$
Thus $\sigma$ and $\sigma_U$ are $C^0$-close. Combined with
$L_p$-closeness of their derivatives this also gives
$W^{1,p}$-closeness for any $0<p<\infty$. The Sobolev embedding
theorem yields a continuous  embedding $W^{1,p}([c,d])\into
C^s([c,d])$ for $s\leq 1-1/p$, so by choosing 
$p$ large we get $C^s$-closeness for any given $s\in[0,1)$.   
\end{proof}

\begin{remark}
If $Z$ is compact it is covered by finitely many of the intervals
$U_\alpha$ in the proof, so in this case we can achieve that for each
$n$ the neighbourhood of $Z$ on which $\sigma_n$ is constant is a {\em
  finite} union of intervals. 
\end{remark}

\begin{proof}[Proof of Proposition~\ref{prop:thick}]
Now we fix some interval $[c,d]$ containing $\im(f)$ and let $Z$, $a$,
$s$ be as in the proposition. Set

$$
   \CC := \{\sigma\in C^{\infty}([c,d])\mid \sigma \text{ is
   constant on some open set containing } Z\},
$$
Consider the space $C^{\infty}(M)$ of smooth functions on $M$. Let 
$$
   f^*:C^{\infty}([c,d])\longrightarrow C^\infty(M)
$$ 
denote the linear operator given by composing with the function $f$ on
the right and introduce the linear subspaces 
$$
   \DD:=f^*(\CC)\subset \EE:=f^*(C^{\infty}([c,d]))\subset
   C^\infty(M).
$$
We introduce the following evaluation linear functional 
$$
   a^*:\EE\longrightarrow \R,\qquad h=\sigma\circ f\mapsto \sigma(a). 
$$
In other words, $a^*(h)$ is the value the function $h$ takes on the
level set $f^{-1}(a)$. For any $b\in \R$ and any subset $S\subset \EE$
we denote $S^b:=S\cap (a^*)^{-1}(b)$. 

We equip $\EE$ with the $C^s$-topology (note that $\EE$ is not
complete with this  topology) and denote by $\bar S\subset\EE$
the closure of a subset $S\subset \EE$ in $\EE$ with respect to this
topology. We claim that for any $b\in \R$, 
\begin{equation}\label{closure.b}
   \bar \DD^b=\EE^b.
\end{equation}
Indeed, the estimate $\|\sigma\circ
f\|_{C^s(M)}\leq\|\sigma\|_{C^s([c,d])}\|f\|_{C^1(M)}^s$ shows
continuity of $f^*:C^\infty([c,d])\to \EE$ with respect to the
$C^s$-norms, so the claim follows from Lemma \ref{lm:appro}.
In partucular, we have
\begin{equation}\label{closure}
   \bar \DD=\EE
\end{equation}
The key observation is that for any $h=\sigma\circ f\in \EE$ the
$2$-form $h\om$ is closed:  
$$
   d(h\om)=\sigma'df\wedge \om=\sigma'd(f\om)=\sigma'd(d\lambda)=0.
$$ 
This allows us to define the operator
$$
   H:\EE\longrightarrow F\subset H^2(M),
$$
where $H(h):=[h\om]$ is the de Rham cohomology class and $F:=H(\EE)$. 
At this point we fix some reference Riemannian metric.
In terms of the Hodge decomposition, the operator of taking cohomology
is just the $L^2$-projection from the space of closed forms to the
space of harmonic forms and thus is continuous with respect to the
$L^2$ topology on the space of closed forms. Since the topology on
$\EE$ is stronger than $L^2$, we deduce that the operator $H$ is
continuous on $\EE$ and from~\eqref{closure} we get
$$
   H(\DD) = H(\EE) = F.
$$ 
The real vector space $F$ is finite dimensional as a subspace of the
finite dimensional space $H^2(M)$ and we set $k:=\dim F$. Set 
$$
   K:=\ker H,\qquad K_\DD:=\ker (H|_{\DD}) = K\cap\DD.
$$ 
We claim that the codimension of $K_D$ in $\DD$ is $k$. Indeed, assume
there were $(k+1)$ linearly independent vectors in $\DD$ spanning a
subspace intersecting $K_D$ trivially; then the restriction of $H$ to
this subspace would give us an injective map from this space to the
space $F$ of dimension one less. On the other hand the codimension  
of $K_D$ in $\DD$ cannot be smaller than $k$, for otherwise the
image of $H$ would have dimension smaller than $k$, contradicting
$H(\DD)=F$.   
Thus there exists a (nonunique) $k$-dimensional subspace
$T$ of $\DD$ such that we have the following (algebraic, not
topological) direct sum decomposition:
\begin{equation}\label{split1}
\DD=K_\DD\oplus T,
\end{equation}
with $H$ restricting as an isomorphism to $T$. Continuity of $H$
implies that 
\begin{itemize}
\item the kernel $K$ of $H$ is closed in $\EE$, and
\item the projection from $\EE$ onto $T$ along $K$ (understood as the
  composition of $H$ and the finite dimensional inverse of
  $H\bigl|_T$) is continuous. 
\end{itemize}
Now we use the freedom in the choice of $T$ to see that we can without
loss of generality assume 
that either $K_\DD\subset \DD^0$ or $T\subset \DD^0$. Indeed, if for
all $h\in K_\DD$ we have $a^*(h)=0$, then the first case is
realized. Otherwise, there exists $h\in K_\DD$ with $a^*(h)=c\ne 0$.
Let $t_1,...,t_k\in T$ be a basis of $T$. For any $j=1,...,k$ set
$c_j:=a^*(t_j)$ and $\hat
t_j:=t_j-\frac{c_j}{c}h$. Now $a^*(\hat t_j)=c_j-\frac{c_j}{c}c=0$, 
so $\hat t_j\in \DD^0$. Moreover, $\{\hat t_j\}_{j=1,...k}$ is the
basis of a linear subspace space $\hat T$ 
which complements $K_\DD$ in $\DD$ because 
$H(\hat t_j)=H(t_j)$ and $\{H(t_j)\}_{j=1,...,k}$ is a basis of $F$.

Now we come to the crucial assertion. We claim that for each $b\in \R$,  
\begin{equation}\label{closurefinal}
   \bar K_\DD^b=K^b 
\end{equation}
To see this let $h\in K^b$ be arbitrary.
According to \eqref{closure.b} we find 
a sequence $\{h_n\}_{n\in \N}\subset \DD^b$ converging to $h$.   
According to \eqref{split1} each $h_n$ can be uniquely
decomposed as $h_n=h_n^K+h_n^T$, with $h_n^K\in K_\DD$ and $h_n^T\in
T$. Continuity of the projection from $\EE$ to $T$ along $K$ and
$h\in K$ implies that the sequence $\{h_n^T\}_{n\in \N}\subset T$
converges to $0\in T$. Since $K$ is closed, we conclude 
$$
   h_n^K = h_n-h_n^T\,\rightarrow\, h\in K.
$$ 
Next, recall that one of the spaces in the direct sum $K_\DD\oplus T$
is a subspace of $\DD^0$.
Now if $K_\DD\subset \DD^0$, then $a^*(h_n^K)=0$ for all $n$ and thus $a^*(h)=0$.
If $T\subset \DD^0$, then $a^*(h_n^K)=a^*(h_n-h_n^T)=a^*(h_n)=b=a^*(h)$. In any case,
$$a^*(h_n^K)=a^*(h)$$ for all $n$. This shows \eqref{closurefinal}.

Note that $f\in K^a$. Therefore, by~\eqref{closurefinal} the function $f$ can
be arbitrarily well $C^s$-approximated by some $\tilde f\in
K_\DD^a$. Then the difference $\beta:=\tilde f\om-f\om$ is $C^s$-small and
exact. By Lemma~\ref{lem:ellipt-small} below, $\beta$ has a primitive 1-form
$\alpha$ that is $C^{1+s}$-small. 
The desired  1-form is now $\tilde\lambda:=\lambda+\alpha$. It is 
$C^{1+s}$-close to $\lambda$, so in particular it evaluates positively on the
Reeb vector field of $(\om,\lambda)$. As 
$d\tilde\lambda=\tilde f \om$ by construction, we see that  
$\tilde \lambda$ stabilizes $\om$. By
definition of $\DD^a$, the function $\tilde f=\sigma\circ f$ is
locally constant on a neighbourhood of 
$f^{-1}(Z)$ and takes value $a$ on $f^{-1}(a)$. 
This proves Proposition~\ref{prop:thick}.
\end{proof}

\begin{lemma}\label{lem:ellipt-small}
Let $\beta$ be a $C^s$-small exact 2-form on a closed manifold $M$,
for some $s\in (0,1)$. Then $\beta$ has a primitive 1-form $\alpha$
that is $C^{1+s}$-small. 
\end{lemma}

\begin{proof}[Proof of Lemma~\ref{lem:ellipt-small}]
A primitive 1-form $\alpha$ is obtained from $\beta$ by first applying
the Green operator 
$G$ for the Laplace-Beltrami operator $\Delta$ and then taking the
co-differential $d^*$. Since $\Delta$ is elliptic (see
e.g.~\cite{War}), it satisfies for each $0<s<1$ and all smooth 2-forms
$\beta$ an estimate of H\"older norms
$$
   \|\beta\|_{2+s}\leq C(\|\Delta\beta\|_s+\|\beta\|_s). 
$$
Indeed, such an estimate is proved in~\cite{GT} for domains in
$\R^n$. 
From this the estimate on a compact manifold $M$ follows using a
finite partition of unity $\phi_i$ via
\begin{align*}
   \|\beta\|_{2+s} 
   &\leq \sum_i\|\phi_i\beta\|_{2+s}
   \leq C\sum_i(\|\Delta(\phi_i\beta)\|_s+\|\phi_i\beta\|_s) \cr
   &\leq C'(\sum_i\|\phi_i\Delta\beta\|_s+\|\beta\|_{s+1})
   \leq C''(\|\Delta\beta\|_s+\|\beta\|_{s+1}) 
\end{align*} 
and applying the same estimate to $\|\beta\|_{s+1}$. Now standard
Hodge theory arguments (see~\cite{War}) imply that the Green operator
satisfies an estimate 
$$
   \|G\beta\|_{C^{2+s}(M)}\leq C\|\beta\|_{C^s(M)}
$$
for all exact 2-forms $\beta$, from which the claim follows. 
\end{proof}

\subsection{Taut foliations}\label{subsec:taut}

In this subsection we investigate the special case of a SHS
$(\om,\lambda)$ with $d\lambda=0$, so $\ker\lambda$ defines a taut
foliation. After a small deformation of the closed 1-form we may
assume that $\lambda$ represents a rational cohomology class, and
after rescaling we may assume $[\lambda]\in H^1(M;\Z)$. Then
integration of $\lambda$ over paths from a fixed base point yields a
fibration $\pi:M\to S^1=\R/\Z$ over the circle. It follows that the
restriction $\bar\om$ of $\om$ to a fibre $W$ is
symplectic. Thus $M=W_\psi$ is the mapping torus of a
symplectomorphism $\psi$ of $(W,\bar\om)$ with $\om$ induced by
$\bar\om$ and $\lambda=\pi^*d\phi$, where $\phi$ is the coordinate on
$S^1$. Now every isotopy of symplectomorphisms $\psi_t$ induces a stable
homotopy on $M$. 
(For this, we always identify $W_{\psi_0}$ with
$W_{\psi_t}$ by the diffeomorphism
$(\phi,x)\mapsto\bigl(\phi,\psi_{t\rho(\phi)}\circ\psi_0^{-1}(x)\bigr)$,
for some fixed function $\rho:[0,1]\to[0,1]$   which equals $0$ near
$0$ and $1$ near $1$).  
Using Moser's theorem, this can be used to classify
SHS with $d\lambda=0$ in dimension 3. 
Here we content ourselves with the following observation that will be
needed later. 

\begin{lemma}\label{lem:taut}
For any symplectomorphism $\psi$ of a closed symplectic manifold
$(W,\om)$ there exists a symplectic isotopy $\psi_t$ such that
$\psi_0=\phi$ and $\psi_1=\id$ on some open subset $U\subset W$.  
\end{lemma}

\begin{proof}
After a Hamiltonian isotopy we may assume that $\psi$ has a fixed point
$p$. 
Now consider the graph ${\rm gr}(\psi)$ as a Lagrangian
submanifold of $(W\times W,\om\oplus-\om)$. A neighbourhood of the
diagonal $\Delta\subset W\times W$ is symplectomorphic to a
neighbourhood of the zero section in the cotangent bundle
$T^*\Delta$. Since ${\rm gr}(\psi)$ 
intersects the zero section at
$(p,p)$ it can be written nearby as the graph of an exact 1-form $dH$
on $\Delta$ with $H(p,p)=0$ and $dH(p,p)=0$. Replacing $H$ by a
function which vanishes identically near $(p,p)$ thus yields a
symplectic isotopy from $\psi$ to a symplectomorphism $\psi_1$ which
equals the identity near $p$. 
\end{proof}

\begin{cor}\label{cor:taut}
Any SHS $(\om,\lambda)$ with $d\lambda=0$ on a closed 3-manifold $M$
is stably homotopic to $(\om_1,\lambda_1)$ such that   
there exists an embedded solid torus $S^1\times D$ in
$M$ on which $(\om_1,\lambda_1)=(d\alpha_\st,\alpha_\st)$ with 
$$
   \alpha_\st = r^2d\theta + (1-r^2)d\phi.
$$
\end{cor}

\begin{proof}
By the discussion preceding Lemma~\ref{lem:taut}, after a
stable homotopy we may assume that $M=W_\psi$ is the mapping torus of a
symplectomorphism $\psi$ of $(W,\bar\om)$ with $\om$ induced by
$\bar\om$ and $\lambda=\pi^*d\phi$. By Lemma~\ref{lem:taut}, we may
further assume that $\psi=\id$ on some open region $U\subset W$. Pick
a disk $D=\{r\leq r_0\}\subset U$ on which $\bar\om=d(r^2d\theta)$ in
polar coordinates, so on $D_\psi\cong S^1\times D$ we have
$\om=d(r^2d\theta)$ and $\lambda=d\phi$. Define a new trivialization
of $D_\psi\cong S^1\times D$ by composing the previous one with the
diffeomorphism $(r,\theta,\phi)\mapsto(r,\theta-\phi,\phi)$ of
$S^1\times D$. In this trivialization we then have  
$$
   \om = d\bigl(r^2d\theta + (1-r^2)d\phi\bigr) = d\alpha_\st,\qquad
   \lambda=d\phi. 
$$
Since $\om$ has constant slope $(1-i)/\sqrt{2}$ on $S^1\times D$, we can homotope stabilizing form $\lambda$
to make it restrict as $(r^2d\theta + (1-r^2)d\phi)$ on a smaller solid torus $S^1\times D'$. 
\end{proof}

\subsection{Contact regions}\label{subsec:contact-region}

\begin{prop}\label{prop:contact-region}
Any SHS $(\om,\lambda)$ on a closed 3-manifold $M$ is stably homotopic
to $(\om_1,\lambda_1)$ such that 
there exists an embedded solid torus $S^1\times D$ in
$M$ on which $(\om_1,\lambda_1)=(d\alpha_\st,\alpha_\st)$ with 
$$
   \alpha_\st = r^2d\theta + (1-r^2)d\phi.
$$
\end{prop}

\begin{proof}
Corollary~\ref{cor:taut} allows us to assume
that the function $f=d\lambda/\om$ is not identically zero. After
applying Corollary~\ref{cor:thick-a} to a value $a\neq 0$ of $f$ we
may assume that $d\lambda=a\,\om$ on some open region $U\subset M$. 
Suppose first that $a>0$, so after
rescaling we may assume $d\lambda=\om$ on $U$. Pick any contractible transverse
knot $\gamma$ in $U$. Pick a neighbourhood $S^1\times D$ of
$\gamma$ on which the contact structure is given by
$\ker\lambda=\ker\alpha_\st$. Then there exists a contact homotopy 
$\lambda_t$ rel $\p U$ (with fixed kernel on $S^1\times D$) from
$\lambda$ to a contact form $\lambda_1$ which equals 
$\alpha_\st$ on a neighbourhood $S^1\times D$ of $\gamma$, so we are
done in this case. 

If $a<0$ we rescale such that $d\lambda=-\om$ on $U$. Let $\bar U$ be the neighbourhood 
$U\subset M$ with the orienation reversed (opposite to that of $M$). Then $\lambda$ is a 
{\em positive} contact form on $\bar U$. Thus, there exists a contact homotopy  
$\lambda_t$ rel $\p U$ from $\lambda$ to a contact form $\lambda_1$ which equals
$r^2d\theta+(1-r^2)d\phi$ on a neighbourhood $S^1\times D^2$ of a contractible transverse 
knot $\gamma$ with 
$d\phi\wedge dr\wedge d\theta$ defining the opposite orientation on $M$. Now consider 
another embedding of $S^1\times D^2$ in $M$ by composing the old one with the flip map
$\theta\mapsto -\theta$ on the right. In the new coordinates $(\phi,r,\theta)$ the form
$\lambda_1$ equals $-r^2d\theta+(1-r^2)d\phi$ near $\gamma$ and the volume form 
$d\phi\wedge dr\wedge d\theta$ defines the positive orientation on $M$. Define $\om_t:=-d\lambda_t$. 
Note that $\om_1=-d\lambda_1=2rdr\wedge(d\theta+d\phi)$. 
Now we homotope rel $\p (S^1\times D^2)$ the stabilizing form 
$\lambda_1$ to $\lambda_2$ (still stabilizing $\om_1$) that restricts as $d\phi$
to a neighbourhood of $\gamma$. (This uses constancy of the slope of $\om_1$.) Then we homotope the HS $\om_1$ to $\om_2$ (supported in
the neighbourhood $\{\lambda_2=d\phi\}$ of $\gamma$) restricting as
$2rdr\wedge(d\theta-d\phi)$ to an even smaller  
neighbourhood of $\gamma$ (by changing the $d\phi$ summand only). The last step is to 
homotope (supported in  $\{\om_2=2rdr\wedge(d\theta-d\phi)\}$) the
stabilizing form $\lambda_2$ to $\lambda_3$ that restricts as
$r^2d\theta+(1-r^2)d\phi$ to a neighbourhood of $\gamma$.  
\end{proof}

\subsection{Genuine stable Hamiltonian structures} \label{subsec:genuine}

Let $h:I\to\C^2$ be an immersion. 
In the notation introduced at the beginning of Section~\ref{subsec:t2inv} set $\om_h:=d\alpha_h$.
Set $\LL_h:=\ker\om_h$

\begin{definition} 
We say that the slope function $k$ of $\om_h$ 
\begin{itemize}
\item {\em twists} if $k(r)\in S^1$
  makes one full turn clockwise and one full turn counterclockwise as $r$ runs
  through $I$;  
\item {\em is nowhere constant} if for any interval $(a,b)\subset I$
  the restriction  
$k\bigl|_{(a,b)}$ is not constant.
\end{itemize}
\label{twistnowconst}
\end{definition}  

\begin{lemma}\label{lem:tw}
If $k$ twists, then the form $\alpha_h$ is not
contact.
\end{lemma}
\begin{proof}
Assume by contradiction that 
$\alpha_h$ is contact. Recall that this means that $\la h,ih'\ra\ne 0$ or equivalently that $h$
always turns clockwise (counterclockwise). Monotonicity of the angle of $h$ and twisting of $h'/|h'|$ 
imply that for some $r\in I$ vectors $h(r)$ and $h'(r)$ are real multiples of each other, i.e. $\la h(r),ih'(r)\ra\ne 0$  which is a contradiction.
\end{proof}

\begin{lemma}
Let $(\om,\lambda)$ be any SHS defining the foliation $\LL_h$ above. 
If $k$ twists, then the proportionality coefficient  $f:=d\lambda/\om$ is not globally constant.
\label{B}
\end{lemma}
\begin{proof}
Assume by contradiction that $f=d\lambda/\om$ is constant. Then by Theorem~\ref{thm:integr}
we can choose coordinates in which both the Hamiltonian structure $\om$
and the stabilizing form $\lambda$ are $T^2$ invariant. So we can write
$$
  \lambda=g_1d\theta+g_2d\phi+g_3dr=\lambda_g+g_3dr
$$ 
for some $g=(g_1,g_2):I\to \C$
and $g_3:I\to \R$. If $f\ne 0$ then $\lambda$ is contact. Now $\lambda\wedge d\lambda=(\lambda_g+g_3dr)\wedge d\lambda_g=
\lambda_g\wedge d\lambda_g$, so $\lambda_g$ is contact and $\ker d\lambda_g=\LL_h$. The contradiction follows from Lemma \ref{lem:tw} with $g$ in place of $h$. If $f=0$,
then the functions $g_1$ and $g_2$ above are constant. The Reeb vector field $R$ spans $\LL_h$
and thus is proportional to $ik$. So $\lambda(R)$ is a multiple of $\la g, ik\ra$. Constancy of $g$
and twisting of $k$ force the last expression to vanish for some $r\in I$ contradicting 
the condition $\lambda(R)=1$.
\end{proof} 

\begin{definition}
Let $\LL$ be a stable Hamiltonian foliation on a $3$-manifold
$M$. 
We say that $\LL$ is {\em genuine} if for {\em every} SHS
$(\om,\lambda)$ defining $\LL$ 
the proportionality coefficient $f:=d\lambda/\om$ is not globally
constant on $M$.  
An SHS defining a genuine $1$-foliation is called {\em genuine}. 
\label{genuine}
\end{definition}

\begin{theorem}
If a stable Hamiltonian foliation $\LL$ on 
$M$ has an integrable region $K$ in which the slope function $k$ 
twists, then for any SHS $(\om,\lambda)$ defining $\LL$ the proportionality coefficient
$f:=d\lambda/\om$ is not constant 
on $K$. In particular, $\LL$ is genuine. If in addition $k$ is nowhere constant, then
any SHS $(\om,\lambda)$ defining $\LL$ satisfies the assumptions
of Theorem \ref{thm1}. 
\label{thm2}
\end{theorem}

\begin{proof} The first assertion follows from Lemma \ref{B}. Assume in addition that 
$S$ is nowhere constant and write $K\cong I\times T^2$. Then for any SHS $(\om,\lambda)$ the proportionality coefficient $f=d\lambda/\om$ must be constant on irrational and thus on all the tori $\{r\}\times T^2$, so by shrinking $I$, if necessary, we can achieve that $f$ 
is a submersion. This matches the assumptions of Theorem \ref{thm1}.
\end{proof}

The following is the main result of this subsection. 

\begin{theorem}\label{thm3} 
Any SHS on a closed 3-manifold can be homotoped to a SHS $(\om,\lambda)$
with the following properties. The $1$-dimensional foliation $\LL$
associated with $(\om,\lambda)$ satisfies the assumptions of
Theorem \ref{thm2} (has an integrable region $K$ in which the slope
function is nowhere constant and twists). 

Moreover, we can achieve that $K=\gamma\times(D^2\setminus\{0\})$ for
an embedded solid torus $\gamma\times D^2$ around a contractible
periodic orbit $\gamma$ of $(\om,\lambda)$. 
\end{theorem}

\begin{proof} 
In view in Proposition~\ref{prop:contact-region}, we can achieve that
$(\om,\lambda)$ restricts to some embedded solid torus $S^1\times
D^2\cong W\subset M$ (where $D^2=\{r\le r_0\}\subset \R^2$ for some
$r_0<1$) as $(d\alpha_\st,\alpha_\st)$, with $\alpha_{st}$ as in
Proposition~\ref{prop:contact-region}. Moreover, we can arrange that
the periodic orbit $\gamma:=S^1\times (0,0)$ is contractible in
$M$. We write $\alpha_{st}=\alpha_{h^0}$ for
$h^0=(h^0_1,h^0_2)=(r^2,1-r^2)$. Note that $W\setminus \gamma$ is 
diffeomorphic to $(0,r_0]\times T^2$ with 
coordinates $(r,\theta,\phi)$ and we fall into the setup of
Section \ref{subsec:t2inv}. We consider a homotopy of immersions
$h^s=(h_1^s,h_2^s):(0,r_0]\to\R^2$ as shown in Figure~\ref{fig:contr}).
\begin{figure}[htb!]\label{fig:contr}
\centering
\includegraphics{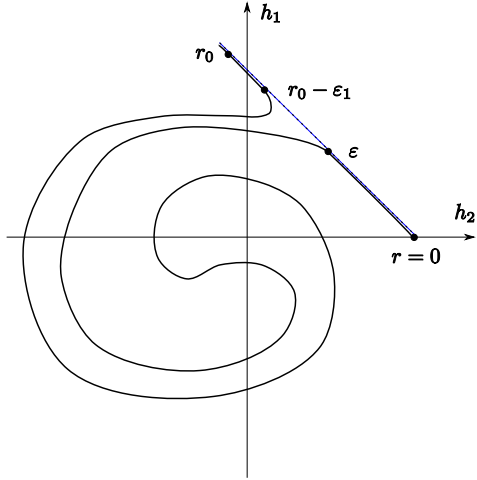}
\caption{The slope function twists}
\end{figure}
(The curve $h^0$ is dashed, the curve $h^1$ is bold, and the points $r=0$,
$\varepsilon$, $r_0-\varepsilon_1$, $r_0$ follow the direction of
increase of the parameter $r$ on the curve $h^1$.  
For Section~\ref{subsec:emb} note that we could have chosen $h^s$ to
have the image in the positive quadrant, though our picture shows a
different situation).  
This homotopy is fixed and equal to $(r^2,1-r^2)$ near $r=0$ and
$r=r_0$, so we can stabilize  it by Corollary~\ref{cor:stablin}. Since
the resulting homotopy of SHS is supported in the interior of
$W\setminus\gamma$, it extends to the whole of $M$. By construction  
the kernel foliation $\LL$ of $\om$ satisfies the assumptions of
Theorem \ref{thm2} on the integrable region $W\setminus\gamma$. 
\end{proof}

\subsection{Non-tame stable Hamiltonian
  structures}\label{subsec:tame} 

As a biproduct of the techniques developed in Section
\ref{subsec:t2inv}, we obtain examples of SHS which are not weakly
tame in the sense of~\cite{CFP}.

Consider a closed $(2n-1)$-manifold $M$ with a Hamiltonian structure
$\om$. Assume that $[\om]$ vanishes on $\pi_2(M)$. 
Denote by $\RR$ the space of contractible closed Reeb orbits.
For $\gamma\in\RR$ we define the {\em $\om$-energy} by 
$$
   E_\om(\gamma):=\int_D\bar\gamma^*\om,
$$
where $\bar\gamma\colon D \rightarrow M$ is a smooth map from the unit
disk 
with $\bar\gamma|_{\p D}=\gamma$. 
Note that this definition is unambiguous because of the assumption on
$[\om]$. For a stabilizing 1-form $\lambda$ we define the {\em
  $\lambda$-energy} of $\gamma$ by  
$$  
   E_\lambda(\gamma):=\int_\gamma\lambda.
$$ 

\begin{definition}\label{def:tame}
Following~\cite{CFP}, we say that a HS $\om$ as above is 
\begin{itemize}
\item {\em weakly tame} if for any compact 
interval $[a,b]\subset \R$ the set of $\gamma\in\RR$ with
$E_\om(\gamma)\in [a,b]$ is compact (with respect to the
$C^\infty$-topology);
\item {\em tame} if for some (and hence every) stabilizing 1-form
$\lambda$ there exists a constant $c_\lambda>0$ such that
$E_\lambda(\gamma)\leq c_\lambda|E_\om(\gamma)|$ for all
$\gamma\in\RR$.
\end{itemize}
\end{definition} 

\begin{remark}
Phrased differently, weak tameness means that the function
$E_\om:\RR\to\R$ is proper. 
It follows from Ascoli-Arcela Theorem that for any SHS the function
$(E_\om,E_\lambda):\RR\to\R^2$ is proper.
Thus a SHS is weakly tame
whenever $E_\lambda$ is bounded in terms of $E_\om$ on $\RR$, e.g.~if
$(\om,\lambda)$ is tame. Note that 
a SHS induced by a contact structure is tame because
$E_\lambda=E_\om$ in that case.   
\end{remark}

\begin{cor}\label{cor:tame}
Let $(\om,\lambda)$ be a stable Hamiltonian structure as in
Theorem~\ref{thm3}, constructed from a contractible knot $\gamma$. 
The $(\om,\lambda)$ is never tame, and it can be arranged not to be
weakly tame.
\end{cor}

\begin{proof}
We use the notation from the proof of Theorem~\ref{thm3}. 
By construction, $\om=d\beta$ for a 1-form $\beta$ on $M$ which 
is of the form $h_1(r)d\phi+h_2(r)d\theta$ on the region
$(0,r_0]\times T^2$ around the contractible knot $\gamma$. 
Moreover, we can arrange that $\gamma$ is contractible and
$E_\om(\gamma)=\int_\gamma\beta$. 

By construction, there exists an interval $(a,b)\subset(0,r_0]$ such
that on $(a,b)\times T^2$ the slope function of $\beta$ is nowhere
constant and twists. By Lemma~\ref{lem:tw}, twisting implies that
$\beta$ is not contact on $(a,b)\times T^2$, so there
exists $r^*\in(a,b)$ such that $\beta\wedge d\beta=0$ along
$\{r^*\}\times T^2$. In other words, $\beta(R)|_{\{r^*\}\times
T^2}\equiv 0$ for the Reeb vector field $R$ of $(\om,\lambda)$. 
Since the slope function $S$ is not constant near $r^*$, there exists a
sequence $\{r_p\}_{p\in \N}\subset S^{-1}(\Q)$ with $r_p\to r^*$ as
$p\to \infty$. Then the foliation $\LL$ is rational on the tori
$\{r_p\}\times T^2$ and $\eps_p:=\beta(R)|_{\{r_p\}\times T^2}\to 0$
as $p\to\infty$. Let $\gamma_p$ be a periodic orbit of $R$ on
$\{r_p\}\times T^2$ of period $T_p$. Then its $\lambda$- and
$\om$-energy satisfy 
$$
   E_\lambda(\gamma_p) = \int_0^{T_p}\lambda(R)dt = T_p, \qquad 
   E_\om(\gamma_p) = \int_0^{T_p}\beta(R)dt = \eps_pT_p, 
$$
which shows that $(\om,\lambda)$ is never tame. 

For the last statement, we choose the function $h(r)$ in the proof of
Theorem~\ref{thm3} such that for some $r^*$ the following holds: 
$h'(r^*)$ is rational and $h(r^*)$ is proportional to $h'(r^*)$
(e.g.~$h(r^*)=0$). Then each leaf $\gamma$ of $\LL$ on the torus
$\{r^*\}\times T^2$ is closed and $\int_\gamma\beta=0$, so $\om$ is
not weakly tame. 
\end{proof}


\subsection{Embeddability}\label{subsec:emb}

\begin{lemma}\label{lem:emb}
Let $h=(h_1,h_2)\colon(0,1)\longrightarrow \R_+\times\R_+$ be an {\em
embedding} such that
$(h_1,h_2)\bigl|_{(0,\varepsilon)\cup(1,1-\varepsilon)}=(r^2,1-r^2)$ 
for some $\varepsilon>0$. Then the corresponding $T^2$-invariant Hamiltonian structure 
$d\bigl(h_1(r)d\theta+h_2(r)d\phi\bigr)$ on 
$(0,1)\times T^2$ can be realized as a hypersurface in the
symplectization $\bigl(\R_+\times(0,1)\times
T^2,d(s\alpha_\st)\bigr)$.  
\end{lemma}

\begin{proof}
Consider the map $f:\R_+\times\R_+\to\R_+\times(0,1)$, 
$$
   f(x_1,x_2) := \left(x_1+x_2,\sqrt{\frac{x_1}{x_1+x_2}}\right). 
$$
This map is a diffeomorphism with inverse given by
$$
   f^{-1}(s,r) = \bigl(sr^2,s(1-r^2)\Bigr). 
$$
Thus it induces a diffeomorphism 
$$
   F:\R_+\times\R_+\times T^2\to\R_+\times(0,1)\times T^2,\qquad
   (x_1,x_2,\theta,\phi)\mapsto \bigl(f(x_1,x_2),\theta,\phi)
$$
which satisfies
$$
   F^*(s\alpha_\st) = F^*\Bigl(sr^2d\theta+s(1-r^2)d\phi\Bigr) =
   x_1d\theta+x_2d\phi. 
$$
The embedding $h=(h_1,h_2):(0,1)\to\R_+\times\R_+$ induces an
embedding
$$
   H:(0,1)\times T^2\to\R_+\times\R_+\times T^2,\qquad
   (r,\theta,\phi)\mapsto \bigr(h_1(r),h_2(r),\theta,\phi\bigr). 
$$
Hence the composition 
$$
   F\circ H:(0,1)\times T^2\to\R_+\times(0,1)\times T^2
$$
is an embedding satisfying
$$
   (F\circ H)^*(s\alpha_\st) = H^*(x_1d\theta+x_2d\phi) =
   h_1(r)d\theta+h_2(r)d\phi.
$$
Taking the exterior derivative on both sides the lemma follows. 
\end{proof}

As an application of this lemma, let us apply the construction of
Theorem~\ref{thm3} (with the function $h^1$ taking values only in the first quadrant) 
in Section~\ref{subsec:genuine} to the SHS
$(\om=d\alpha,\lambda=\alpha)$ corresponding to a Stein fillable
contact manifold $(M,\alpha)$. Since the symplectization of
$(M,\alpha)$ then embeds into the completion of the Stein filling, 
Lemma~\ref{lem:emb} implies 

\begin{cor}
Let $(\om=d\alpha,\lambda=\alpha)$ be the SHS corresponding to a Stein
fillable contact manifold $(M,\alpha)$. Then the resulting stable Hamiltonian
structure $(\om_1,\lambda_1)$ in Theorem~\ref{thm3} embeds as a
hypersurface in (the completion of) the same Stein manifold. 
For example, applying this to the standard contact structure on $S^3$,
we obtain a hypersurface bounding a ball in standard $\R^4$ whose
characteristic foliation has the properties in Theorem~\ref{thm2}. 
\end{cor}

\begin{remark}
Lemma~\ref{lem:emb} fails without the assumption that $(h_1,h_2)$
takes values only in the positive quadrant, see
Remark~\ref{rem:exotic} 
\end{remark}

\subsection{Foliated cohomology of integrable
  regions}\label{ss:folcoh3} 

\begin{prop}\label{prop:folcoh3}
Let $\om$ be a Hamiltonian structure on a closed 3-manifold $M$ with
$\ker(\om)=\LL$. If $\LL$ possesses an integrable region on which the
slope function is non-constant, then the foliated cohomology
$H^2_\LL(M)$ is infinite dimensional. 
\end{prop}

\begin{remark}
It follows that under the hypotheses of
Proposition~\ref{prop:folcoh3}, $\ker\bigl(\kappa:H^2_\LL(M)\to
H^2(M)\bigr)$ is infinite dimensional. 
\end{remark}

\begin{proof}
Let $I\times T^2$ be an open integrable region on which $\om=d\alpha$
with  
$$
   \alpha = h_1(r)d\theta+h_2(r)d\phi
$$
and non-constant slope function $S(r)=h'(r)/|h'(r)|$. After shrinking
$I$ we may assume that $S(r)$ is nowhere constant in $I$. Note that
the kernel foliation $\LL$ is positively generated  by 
the vector field 
$$
   R = h_1'(r)\p_\phi-h_2'(r)\p_\theta.
$$
spans the oriented foliation $\LL$.
Consider first a 2-form $\beta\in\Om^2_\LL(I\times T^2)$. Then $\beta
= f\,d\alpha$ for an $R$-invariant function $f$. As $f$ must be
constant on each irrational torus and irrational tori are dense
(because the slope of $R$ is nowhere constant), $f=f(r)$ depends only
on $r$. This shows that 
$$
   \Om^2_\LL(I\times T^2) = \ker(d:\Om^2_\LL\to\Om^3_\LL) =
   \{f(r)\,d\alpha\mid f:I\to\R\}. 
$$
Next we compute the image of $d:\Om^1_\LL\to\Om^2_\LL$. Thus consider
$\mu\in\Om^1_{\LL}$. The condition $i_R\mu=0$ implies 
$$
   \mu = a\,dr + b\bigl(h_1'(r)d\theta + h_2'(r)d\phi\bigr) 
$$
for functions $a,b:I\times T^2\to\R$. By the discussion above,
$i_Rd\mu=0$ implies that $d\mu = f(r)d\alpha$ for a function
$f:I\to\R$, which is equivalent to the conditions
\begin{align*}
   \frac{\p}{\p r}(bh_1') - \frac{\p a}{\p\theta} &= f(r)h_1'(r), \cr  
   \frac{\p}{\p r}(bh_2') - \frac{\p a}{\p\phi} &= f(r)h_2'(r), \cr  
   h_1'(r)\frac{\p b}{\p\phi} - h_2'(r)\frac{\p b}{\p\theta} &= 0.  
\end{align*}
The last equation shows that $b$ is invariant under $R$ and thus a
function $b=b(r)$ only of $r$. The first two equations then show that
$\frac{\p a}{\p\theta}$ and $\frac{\p a}{\p\phi}$ depend only on $r$,
so by periodicity $a=a(r)$ depends only on $r$. The first two
equations now combine to
$$
   b(r)h''(r) + b'(r)h'(r) = \frac{d}{dr}\Bigl(b(r)h'(r)\Bigr) =
   f(r)h'(r). 
$$
Since $h''(r)$ is linearly independent of $h'(r)$ for almost all $r$,
this implies $b\equiv 0$, and therefore $\mu=a(r)dr$ and $d\mu=0$. So
we have shown  
$$
   \im(d:\Om^1_\LL\to\Om^2_\LL) = 0
$$
and
$$
   H^2_\LL(I\times T^2) = \{f(r)\,d\alpha\mid f:I\to\R\}. 
$$
Now if $f\in C^\infty_0(I)$ has compact support the 2-form
$f(r)d\alpha$ extends by zero to a closed form in $\Om^2_\LL(M)$ which
by the preceding argument is exact iff $f\equiv 0$, so $H^2_\LL(M)$
contains the infinite dimensional subspace
$$
   \{f(r)\,d\alpha\mid f\in C^\infty_0(I)\}. 
$$
\end{proof}

\subsection{Overtwisted disks}\label{subsec:otd}

Legendrian and transverse knots play a central role in contact
topology. The basic reason for this is that, 
as a consequence of Gray's stability theorem, the homotopy types of the
spaces of Legendrian and transverse knots do not change during
contact homotopies (see below for the precise formulation). 
In this subsection we show that any SHS in dimension 3 is homotopic to
one containing an overtwisted disk, and as a consequence, the homotopy
types of the spaces of Legendrian and transverse knots may change
during stable homotopies. 

Recall that an {\em overtwisted disk} in a contact 3-manifold
$(M,\xi=\ker\lambda)$ is an embedded disk $D_\ot\subset M$ such that
$TD_\ot=\xi$ along $\p D_\ot$.   

\begin{prop}\label{prop:otd}
Any SHS $(\om_0,\lambda_0)$ on a closed oriented 3-manifold $M$ is
stably homotopic to a SHS $(\om_1,\lambda_1)$ such that $\om_1=
d\lambda_1$ on an open region $U\subset M$ and there exists an overtwisted
disk $D_\ot\subset U$. 
\end{prop}

\begin{proof}
By Proposition~\ref{prop:contact-region}, after a stable homotopy we
may assume that there exists an embedded solid torus $S^1\times D$ in
$M$ on which $(\om_0,\lambda_0)=(d\alpha_\st,\alpha_\st)$ with 
$$
   \alpha_\st = r^2d\theta + (1-r^2)d\phi.
$$
Here $(\phi,r,\theta)$ are polar coordinates on $S^1\times D$ and
$D=\{r\leq r_2\}$ for some $r_2>0$. 

Let $h_t=(h_{1t},h_{2t}):[0,r_2]\to\R^2$ be a family of immersions and $0<r_*<r_0<r_1<r_2$
be points with the following properties
\begin{figure}[htb!]\label{fig:otd}
\centering
\includegraphics{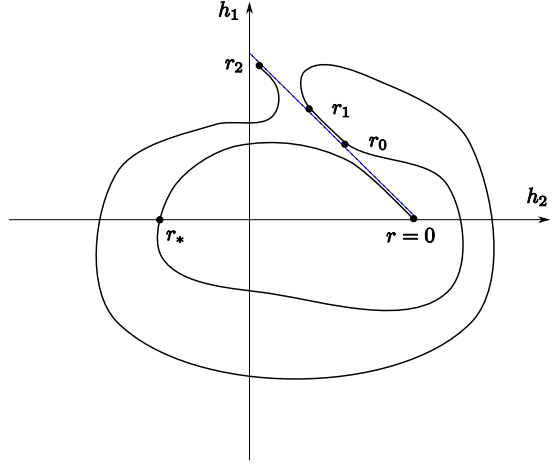}
\caption{Overtwisted disc}
\end{figure}
(see Figure \ref{fig:otd}); the curve $h_0$ is dashed, the curve $h_1$
is bold, and the points $r=0$, $r_*$, $r_0$, $r_1$, $r_2$ follow the
direction of increase of the parameter $r$ on the curve $h_1$; note
that the $h_1$-axis is vertical and the $h_2$-axis horizontal).  
\begin{itemize}
\item $h_0(r)=(r^2,1-r^2)$;
\item $h_t(r)=(r^2,1-r^2)$ near $r=0$ and $r=r_2$ for all $t$;
\item $h_1(r)=(r^2,1-r^2)$ for $r\in [r_0,r_1]$
\item $h_{11}'h_{21}-h_{21}'h_{11}>0$ on the interval $(0,r_1]$;
\item $h_{11}(r_*)=0$ and $h_{21}(r_*)<0$. 
\end{itemize}
This induces a homotopy of $T^2$-invariant HS $\om_t:=d\alpha_t$ with 
$$
   \alpha_t := h_{1t}(r)d\theta + h_{2t}(r)d\phi. 
$$
Note that, $\alpha_t=\alpha_\st$ near $r=0$ and $r=r_2$ for all $t$, and
$\alpha_1=\alpha_\st$ for $r\in [r_0,r_1]$. Moreover, $\alpha_1$ is a
positive contact form on $\{r\le r_1\}$. We apply Corollary~\ref{cor:stablin}  
twice. First, we apply it to the homotopy $\{\om_t\}_{t\in [0,1]}$
to get a family of $T^2$-invariant 1-forms $\lambda_t$ stabilizing
$\om_t$. Second, we apply Corollary~\ref{cor:stablin} to the constant homotopy
$\om_1\bigl|_{\{r\in [r_0,r_2]\}}$. This gives us a stabilizing form
$\hat\lambda$ for $\om_1\bigl|_{\{r\in [r_0,r_2]\}}$ which  
coincides with $\alpha_1$ on $\{r\in [r_0,r_0+\varepsilon]\}$ (both
forms equal $\alpha_{st}$ there). Since $\alpha_1$ is a contact form
on $\{r\le r_1\}$, the form $\hat\lambda$ can be extended
from $\{r\in [r_0,r_2]\}$ to $\{r\le r_1\}=S^1\times D$ as $\alpha_1$.
Thus $\hat\lambda$ stabilizes $\om_1$ and 
$\om_1=d\hat\lambda$ on $U:=\{r<r_0+\varepsilon\}$. 
We join $\lambda_1$ and $\hat\lambda$ by a linear homotopy. This
shows that $(\om_0,\lambda_0)$ is stably homotopic to the SHS
$(\om_1,\hat\lambda)$. We claim that for each angle $\phi_*$,  
$D_\ot:=\{r\leq r_*,\phi=\phi_*\}\subset U$ is an overtwisted disk for
the contact form $\hat\lambda|_U=\alpha_1|_U$. Indeed, the conditions 
$h_{11}(r_*)=0$ and $h_{21}(r_*)<0$ imply that along $\p D_\ot$ the
contact planes $\xi_1=\ker\alpha_1$ and the tangent spaces $TD_\ot$ are
both spanned by the vectors $\p_r$ and $\p_\theta$. Renaming
$\lambda_1:=\hat\lambda$, thus concludes the proof.
\end{proof}

Now consider a nowhere vanishing 1-form $\lambda$ on an
oriented 3-manifold $M$ with kernel distribution $\xi=\ker\lambda$. 
Following the terminology from contact topology (see e.g.~\cite{Et}),
we call an oriented knot $\gamma:S^1\to M$ {\em Legendrian} if
$\lambda(\dot\gamma)\equiv 0$ and {\em (positively) transverse} if
$\lambda(\dot\gamma)>0$. Denote by $\Om^1_\nv$ the space of nowhere
vanishing 1-forms on $M$ and by $\Lambda$ the space of embeddings
$S^1\into M$, both equipped with the $C^\infty$-topology, and consider
the projections 
\begin{align*}
   \pi^\Leg:\{(\lambda,\gamma)\in\Om^1_\nv\times\Lambda\mid
   \lambda(\dot\gamma)\equiv 0\}\to\Om^1_\nv, \cr
   \pi^\trans:\{(\lambda,\gamma)\in\Om^1_\nv\times\Lambda\mid
   \lambda(\dot\gamma)> 0\}\to\Om^1_\nv.
\end{align*}
Thus the fibres $(\pi^\Leg)^{-1}(\lambda)$
resp.~$(\pi^\trans)^{-1}(\lambda)$ are the spaces of Legendrian
resp.~transverse knots for $\lambda$. As a consequence of Gray's
stability theorem, the restrictions of $\pi^\Leg$ and $\pi^\trans$ to
the preimages of the space of contact forms are locally trivial
fibrations. By contrast, denote by $\SHS$ the space of SHS and
consider the projections 
\begin{align*}
   \Pi^\Leg:\{(\om,\lambda,\gamma)\in\SHS\times\Lambda\mid
   \lambda(\dot\gamma)\equiv 0\}\to\SHS, \cr
   \Pi^\trans:\{(\om,\lambda,\gamma)\in\SHS\times\Lambda\mid
   \lambda(\dot\gamma)> 0\}\to\SHS.
\end{align*}

\begin{cor}
The projections $\Pi^\Leg$ and $\Pi^\trans$ are not Serre fibrations. 
\end{cor}

\begin{proof}
By Proposition~\ref{prop:otd} there exists a homotopy
$(\om_t,\lambda_t)$ of SHS on 
$M$ with $(\om_0,\lambda_0)=(d\lambda_\st,\lambda_\st)$ such that $\om_1=
d\lambda_1$ on an open region $U\subset M$ and there exists an overtwisted
disk $D_\ot\subset U$. We claim
that this path in $\SHS$ cannot be lifted to a continuous path
$\gamma_t$ in the total space of $\Pi^\Leg$ with $\gamma_1$
parametrizing $\p D_\ot$. 

Indeed, suppose such a lift $\gamma_t$ exists. For each $t$ let
$\tb(\gamma_t)$ be the {\em Thurston-Bennequin invariant}, defined as
the linking number of $\gamma_t$ with its push-off in the Reeb
direction. Then $t\mapsto\tb(\gamma_t)$ is continuous and integer
valued, hence constant, and therefore
$\tb(\gamma_0)=\tb(\gamma_1)=0$. But this contradicts Bennequin's
inequality $\tb(\gamma_0)\leq -1$ for every
Legendrian knot $\gamma_0$ in $(S^3,\lambda_\st)$.  

This proves that $\Pi^\Leg$ is not a Serre fibration. An analogous
argument shows that the path $(\om_t,\lambda_t)$ cannot be lifted to a
continuous path $\gamma_t$ in the total space of $\Pi^\trans$ with
$\gamma_1$ parametrizing a positive transverse push-off of the
negatively oriented boundary $\p D_\ot$. 

Indeed, suppose such a lift $\gamma_t$ exists. For each $t$ let
${\rm sl}(\gamma_t)$ be the {\em self-linking number}, defined as
the linking number of $\gamma_t$ with its push-off in the direction of
a nowhere vanishing section of $\ker\lambda$ over a spanning disk for
$\gamma_t$. Then $t\mapsto{\rm sl}(\gamma_t)$ is continuous and integer
valued, hence constant, and therefore
${\rm sl}(\gamma_0)={\rm sl}(\gamma_1)=1$ (see~\cite{BCV}). But this
contradicts Bennequin's inequality ${\rm sl}(\gamma_0)\leq -1$ for every
transverse knot $\gamma_0$ in $(S^3,\lambda_\st)$.  
\end{proof}

\subsection{Contact structures as stable Hamiltonian
  structures}\label{subsec:contact-SHS}

Every positive contact form $\lambda$ induces a SHS
$(d\lambda,\lambda)$ and homotopies of contact forms induce stable
homotopies, so we have a natural map 
$$
   \Cont/\sim\to \SHS_0/\sim
$$
from homotopy classes of positive contact forms to homotopy classes of
exact SHS.

\begin{thm}\label{thm:contact-SHS}
On $S^3$ the map $\Cont/\sim\to \SHS_0/\sim$ is not
bijective. In fact, if rational SFT exists as an invariant of SHS,
then the map is injective but not surjective. 
\end{thm}


For the proof, let us fix some conventions. All contact structures
will be oriented, so that the underlying formal structures are 
oriented plane fields. For a contact form $\alpha$ the orientation  
on the contact structure $\ker\alpha$ is given by $d\alpha$. 
More generally, for any plane field $\xi$ a $2$-form $\gamma$
restricting as an area form on $\xi$ induces an orientation on
$\xi$. The notation $(\xi,\gamma)$ will stand for the corresponding
oriented plane field. The corresponding homotopy class of  
oriented plane fields will be denoted by $[(\xi,\gamma)]$.

Consider 
$$
   S^3:=\{(x_1,y_1,x_2,y_2)\in \R^4|x_1^2+y_1^2+x_2^2+y_2^2=1\}
$$ 
with its standard orientation. Consider the orientation reversing
involution 
$$
   \Psi:S^3\to S^3,\qquad
   (x_1,y_1,x_2,y_2)\mapsto (-x_1,y_1,x_2,y_2).
$$ 
For a 1-form $\alpha$ on $S^3$ we set $\bar\alpha:=\Psi^*\alpha$. 
So the standard contact form and its pullback are given by
\begin{gather*}
   \lambda_\st=x_1dy_1-y_1dx_1+x_2dy_2-y_2dx_2,\cr
   \bar\lambda_\st=-x_1dy_1+y_1dx_1+x_2dy_2-y_2dx_2.
\end{gather*} 

\begin{lemma}\label{lem:Hopf}
The standard contact form an its pullback under $\Psi$ represent
different homotopy classes of oriented plane fields
$$
   \CC := [(\ker\lambda_\st,d\lambda_\st)]\ne \bar\CC :=
   [(\ker\bar\lambda_\st,d\bar\lambda_\st)].
$$ 
\end{lemma}

\begin{proof}
The vector fields 
\begin{gather*}
   R_+:=-y_1\p_{x_1}+x_1\p_{y_1}-y_2\p_{x_2}+x_2\p_{y_2},\cr 
   R_-:=-y_1\p_{x_1}+x_1\p_{y_1}+y_2\p_{x_2}-x_2\p_{y_2}
\end{gather*}
are transverse to $\ker\lambda_\st$ and $\ker\bar\lambda_\st$,
respectively. Moreover, the transverse orientations defined by these
vector fields together with the corresponding orientations 
of the plane fields define the standard orientation on $S^3$. We show
that (in some trivialization) the Hopf invariants of these vector
fields are different. For this we choose a framing 
of $TS^3$ as follows: 
\begin{align*}
   e_1 &:= R_+, \cr
   e_2 &:= (-x_2\p_{x_1}+y_2\p_{y_1}+x_1\p_{x_2}-y_1\p_{y_2}), \cr
   e_3 &:= (-y_2\p_{x_1}-x_2\p_{y_1}+y_1\p_{x_2}+x_1\p_{y_2}).
\end{align*} 
In this framing $R_+$ is constant and thus 
has Hopf invariant zero. It remains to compute $R_-$. For this we use
the Riemannian metric $\la\ ,\ \ra$ on $S^3$ induced by the standard
Euclidean metric on $\R^4$. This gives us: 
\begin{align*}
   \la R_-,e_1\ra &= (x_1^2+y_1^2)-(x_2^2+y_2^2),\cr 
   \la R_-,e_2\ra &= 2(y_1x_2+x_1y_2),\cr
   \la R_-,e_3\ra &= 2(y_1y_2-x_1x_2).
\end{align*}
So in complex coordinates $z_1=x_1+iy_1$, $z_2=x_2+iy_2$ the map
$S^3\rightarrow S^2$ induced by $R_-$ is given by 
$$
   (z_1,z_2)\mapsto (|z_1|^2-|z_2|^2,-2iz_1z_2)
$$ 
if we identify $S^2$ with the unit sphere in $\R\times \C$. Up to a
factor of $-i$ this is just the Hopf fibration and thus has Hopf
invariant $1$.
\end{proof}

\begin{proof}[Proof of Theorem~\ref{thm:contact-SHS}]
Let $\lambda_\ot$ be an overtwisted contact form on $S^3$ with the
same underlying class of oriented plane fields as $\lambda_\st$. We
distinguish two cases.

{\bf Case 1: }$(d\lambda_\st,\lambda_\st)$ and
$(d\lambda_\ot,\lambda_\ot)$ are stably homotopic. Then the map
$\Cont/\sim\to \SHS_0/\sim$ is not injective. 

{\bf Case 2: }$(d\lambda_\st,\lambda_\st)$ and
$(d\lambda_\ot,\lambda_\ot)$ are not stably homotopic. 

In this case, observe that $[(\ker\bar\lambda_\ot,d\bar\lambda_\ot)] = 
[(\ker\bar\lambda_\st,d\bar\lambda_\st)]=\bar\CC$.  
Consider the positive SHS $(d\bar\lambda_\st,-\bar\lambda_\st)$ and
$(d\bar\lambda_\ot,-\bar\lambda_\ot)$ on $S^3$ and note that their
underlying classes of oriented plane fields are both equal to
$\bar\CC$. Assume by contradiction that the map $\Cont/\sim\to
\SHS_0/\sim$ is surjective. Then there exist positive contact  
forms $\alpha_1$ and $\alpha_2$ on $S^3$ such that 
the SHS $(d\bar\lambda_\st,-\bar\lambda_\st)$ is homotopic to
$(d\alpha_1,\alpha_1)$ and the SHS
$(d\bar\lambda_\ot,-\bar\lambda_\ot)$ is homotopic to
$(d\alpha_2,\alpha_2)$. Note that the underlying oriented plane field
class for both $\alpha_1$ and $\alpha_2$ is $\bar\CC$ and the class of
$\lambda_\st$ is $\CC\ne \bar\CC$ by Lemma~\ref{lem:Hopf}.  
Thus, by uniqueness of the positive tight contact structure on $S^3$
(see~\cite{Elia-tight}), both $\alpha_1$ and $\alpha_2$ must be
overtwisted. The classification of overtwisted contact
structures~\cite{Elia-overt} implies that $\alpha_1$ and $\alpha_2$
are homotopic through contact forms. By transitivity, this implies
that the SHS $(d\bar\lambda_\st,-\bar\lambda_\st)$ and
$(d\bar\lambda_\ot,-\bar\lambda_\ot)$ on $S^3$ are homotopic, so the
SHS $(d\lambda_\st,\lambda_\st)$ and $(d\lambda_\ot,\lambda_\ot)$
are homotopic, contradicting the assumption of Case 2.

If rational SFT exists as an invariant of SHS, a standard computation
shows that it vanishes for $(d\lambda_\ot,\lambda_\ot)$ but not for
$(d\lambda_\st,\lambda_\st)$ (see Corollary~\ref{cor:sft}), so Case 1
does not occur. The preceding discussion of Case 2 shows that the map
$\Cont/\sim\to \SHS_0/\sim$ is not surjective, and injectivity follows
from the classification of contact structures on $S^3$
in~\cite{Elia-tight,Elia-overt}. 
\end{proof}

\section{Open books and stable Hamiltonian structures}\label{sec:openbook}

Open books play a central role in contact geometry in dimension three: 
Thurston and Winkelnkemper~\cite{TW} constructed contact strutures
supported by open books on every oriented closed 3-manifold, and 
Giroux~\cite{Gir} proved that every contact structure arises this
way. In this section we generalize the Thurston-Winkelnkemper result
to stable Hamiltonian structures in the following form: Every 
cohomology class in $H^2(M;\R)$ is
represented by a SHS supported by a given open book. Moreover, we prove
that all SHS supported by the same open book and representing the same
cohomology class are stably homotopic. 

To state the results more precisely we need some notation. 
Throughout this section, $M$ is a closed oriented $3$-manifold. 
When dealing with manifolds, cohomology always means de Rham
cohomology; on general topological spaces it means singular cohomology
with coefficients in $\R$. 

We start by recalling the definition of open book decomposition. 
\begin{definition}\label{def:obd}
An {\em open book decomposition} of $M$ is a pair $(B,\pi)$, where 
\begin{itemize}
\item [(B)] $B$ is an oriented link in $M$ (called the {\em binding}),
  and  
\item [(P)] $\pi:M\setminus B\rightarrow S^1$ is a fibration with
  fibre a punctured oriented surface $\Sigma$ (called the {\em page})
  satisfying the following condition near $B$. Each connected
  component $B_l$ of $B$ has a tubular neighbourhood 
  $S^1\times D^2\hookrightarrow M$ with orienting coordinates 
  $(\phi_l,r_l,\theta_l)$, where $\phi_l$ is an orienting coordinate
  along $S^1\cong B_l$ and $(r_l,\theta_l)$ are polar coordinates on
  $D^2$, such that on $(S^1\times D^2)\setminus B_l$ we have
  $\pi(\phi_l,r_l,\theta_l)=\theta_l$.
\end{itemize}
\end{definition}
We can (and will) chose a global coordinate $\theta$ 
on the target $S^1$ of the projection $\pi$ and particular polar
coordinates on the tubular neighbourhood of $B$ such that
$\theta=\theta_l=\pi(\phi_l,r_l,\theta_l)$ near each $B_l$. 

\begin{definition}\label{def:openbook}
Let $(\omega,\lambda)$ be a stable Hamiltonian structure on $M$ and
$R$ be its Reeb vector field. 
We say that the stable Hamiltonian structure $(\omega,\lambda)$ is
{\em positively supported by the open book $(B,\pi)$} if the following
conditions hold:
\begin{itemize}
\item[(OB1)] $\lambda|_B>0$, and
\item[(OB2)] $\omega$ restricts as a positive area form to each page
  of the open book. 
\end{itemize}
Or equivalently:
\begin{itemize}
\item[(OB3)] $R$ is positively tangent to $B$, and
\item[(OB4)] $R$ is positively transverse to the pages.
\end{itemize}
We say that the stable Hamiltonian structure $(\omega,\lambda)$ is
{\em supported by the open book $(B,\pi)$} if (OB2), or equivalently (OB4), holds.
Note that if a SHS $(\om,\lambda)$ is suported by $(B,\pi)$, then it is automatic
that the Reeb vector field $R$ is tangent to $B$. In that case to each component $B_l$
of the binding we associate a sign $s_l(\om,\lambda)\in \{+,-\}$ 
according to whether $R$ is positively or negatively tangent. 
\end{definition}
This definition generalizes the contact case: A cooriented contact
structure $\xi$ is supported an open book decomposition $(B,\pi)$ in
the sense of~\cite{Gir} iff the exists a contact form $\alpha$ defining
$\xi$ such that the SHS $(d\alpha,\alpha)$ is positively supported by $(B,\pi)$.

The goal of this section is to prove the following two results. 

\begin{theorem}\label{obd1}
Let $(B,\pi)$ be an open book decomposition of a closed oriented
3-manifold $M$, $s:\pi_0(B)\rightarrow \{+,-\}$ any function, and 
$\eta\in H^2(M;\R)$. Then there exists a SHS $(\om,\lambda)$ supported
by the open book $(B,\pi)$ with 
$s_l(\om,\lambda)=s(B_l)$ for each component $B_l$ of $B$ and
$[\om]=\eta\in H^2(M;\R)$.  
\end{theorem}

\begin{theorem}\label{obd2}
Let $(B,\pi)$ be an open book decomposition of a closed oriented
3-manifold $M$. Then any two SHS
$(\om,\lambda)$ and $(\tilde\om,\tilde\lambda)$ supported by the open book
$(B,\pi)$ with $s_l(\om,\lambda)=s_l(\tilde\om,\tilde\lambda)$ for
each connected component  
$B_l$ of $B$ and $[\om]=[\tilde\om]\in H^2(M;\R)$ are connected by a
stable homotopy supported by $(B,\pi)$.
\end{theorem}

Note that the direct analog of being ``supported'' for a contact form is being ``positively 
supported'' for a SHS. So why to weaken ``positively supported'' to ``supported''? The reason 
is Theorem~\ref{thm:contact-SHS}. Assuming that SFT exists as an
invariant of SHS, this theorem together with Theorem~\ref{obd2} 
implies that there exists an exact SHS $(\om,\lambda)$ on $S^3$ which is not homotopic
to a SHS positively supported by an open book decomposition. So in order to have any chance
to get a structure theorem like ``each SHS is homotopic to a one
supported by an open book decomposition'' we have to say ``supported''
not ``positively supported''. 

\begin{remark}
(1) Theorem~\ref{obd1} yields an alternative proof of the existence
result in Proposition~\ref{prop:ex}. Indeed, let $\LL$ be an oriented
line field on $M$ and $\AA$ a transverse oriented plane field. By the
classical results of Lutz and Martinet (see e.g.~\cite{Ge}), $\AA$ is
homotopic to a 
contact structure $\xi$. By Giroux' theorem~\cite{Gir} the contact
structure $\xi$ is supported by an open book $(B,\pi)$. 
Now Proposition~\ref{prop:ex} follows from Theorem~\ref{obd1} applied
to the open book $(B,\pi)$ and the easy
observation that any two SHS's supported by the same open book are
homotopic as oriented plane fields. 

(2) Any positive stabilization of an open book $(B,\pi)$ supports a
contact structure in the same isotopy class (cf.~\cite{Gir}). By
taking successive positive stabilizations we obtain countably many
open books $(B_i,\pi_i)$ supporting contact structures $\xi_i$, all 
the $\xi_i$ being isotopic (in particular homotopic as plane fields)
to each other.  
Applying Theorem~\ref{obd1} (with positively
supported!) to each $(B_i,\pi_i)$ we obtain countably
many SHS's, all of them homotopic as oriented plane fields and
defining the same cohomology class $\eta$. 
The following question arises naturally: Are these SHS's all homotopic
as SHS's? 
\end{remark}

The proofs of Theorems~\ref{obd1} and~\ref{obd2} occupy the remainder
of this section. The actual proofs are given in
Subsections~\ref{ss:obd1} and~\ref{ss:obd2}; the other subsections
provide technical results needed in the proofs, some of which are of
independent interest (e.g.~Subsection~\ref{ss:binding}). 
 
\subsection{Cohomology of mapping tori and open books}
\label{subsec:top} 

For a continuous map $f:X\to X$ of a topological space denote by
$X_f:=[0,1]\times X/(0,x)\sim(1,f(x))$ its mapping torus. Denote by
$i:X\to X_f$ the inclusion $x\mapsto[0,x]$ and by $\pi:X_f\to S^1$ the
projection onto $S^1=\R/\Z$. There exists a long exact sequence in singular homology 
 (see~\cite{Ha}, Section 2.2)
$$
   \dots\longrightarrow H_k(X)\overset{\id-f_*}{\longrightarrow}
   H_k(X)\overset{i_*}{\longrightarrow}
   H_k(X_f)\overset{\p}{\longrightarrow} H_{k-1}(X)\to\dots 
$$
Moreover, the construction of this sequence shows that the boundary
map $\p$ is given by intersection with the fibre $X$ over $0$. Now
suppose that $f$ maps a subset $A\subset X$ to itself. Then the
corresponding sequence in relative cohomology is 
$$
   \dots\longrightarrow H^k(X_f,A_f)\overset{i^*}{\longrightarrow}
   H^k(X,A)\overset{\id-f^*}{\longrightarrow}
   H^k(X,A)\overset{d}{\longrightarrow} H^{k+1}(X_f,A_f)\to\dots 
$$
To describe the coboundary map $d$, pick a function
$\rho:(-1/2,1/2)\to\R$ with compact support and integral $1$ and set
$$
   \Theta := \rho(\theta)d\theta\in\Om^1(S^1). 
$$
Then $\pi^*\Theta$ is dual to the fibre over $0$, and therefore $d$ is
given by the cup product with $[\pi^*\Theta]\in H^1(X_f)$. 

Now suppose that $\phi:W\to W$ is an orientation preserving diffeomorphism of a compact
connected surface with boundary. Then $\phi^*=\id$ on $H^2(W,\p W)$
and the exact sequence simplifies to
\begin{gather}\label{eq:exact-seq}
   0\longrightarrow 
   H^1(W_\phi,\p W_\phi)\overset{i^*}{\longrightarrow}
   H^1(W,\p W)\overset{\id-f^*}{\longrightarrow}
   H^1(W,\p W)\overset{\Theta\wedge}{\longrightarrow} \cr 
   \overset{\Theta\wedge}{\longrightarrow} 
   H^2(W_\phi,\p W_\phi)\overset{i^*}{\longrightarrow}
   H^2(W,\p W)\longrightarrow 0. 
\end{gather}
In particular, we have 
\begin{gather}\label{eq:mapping-torus}
   H^1(W_\phi,\p W_\phi)\cong\ker(\id-\phi^*)\subset H^1(W,\p W),\cr 
   H^2(W_\phi,\p W_\phi)=V\oplus \Theta\wedge H^1(W,\p
   W), 
\end{gather}
where $V\subset H^2(W_\phi,\p W_\phi)$ is a subspace such that the restriction 
$i^*|_V:V\to H^2(W,\p W)$ is an isomorphism.
Let us discuss the two terms in $H^2(W_\phi,\p W_\phi)$. The first
term is a $1$-dimensional $\R$-linear space generated a preimage (under the map $i^*$)
of a cohomology class in $H^2(W,\p W)$ represented 
by a compactly supported 2-form on $W$ of positive
area.
The second term is generated by
compactly supported closed 1-forms on $W$ spanning a complement of
$\im(\id-\phi^*)$ in $H^1(W,\p W)$. Each such 1-form $\alpha$ induces
a closed 2-form 
$$
   \Theta\wedge\alpha = d\theta\wedge\bigl(\rho(\theta)\alpha\bigr)
$$
on $W_\phi$. Note that $\rho(\theta)\alpha$ is a global 1-form on
$W_\phi$, but in general it is not closed. This finishes our
discussion of mapping tori. 

Now consider an open book $(M^3,\pi)$ with page $\Sigma$, and binding $B$. 
Let $(\phi_l,r_l,\theta)$ be standard coordinates as above near
each binding component $B_l$, $l=1,\dots,n$.   
Consider the following  
part of the exact cohomology sequence of the pair $(M,B)$:
$$
   \R^n\cong H^1(B)\stackrel{d^*}\longrightarrow
H^2(M,B)\stackrel{j^*}\longrightarrow H^2(M)\longrightarrow
H^2(B)=0,
$$ 
where $d^*$ is the connecting homomorphism and $j^*$ is the map
induced by the natural forgetful map. 

\begin{lemma}\label{algtop}
(a) Any de Rham cohomology class $\eta\in H^2(M)$ has a representative
of the form 
$d\theta\wedge \beta$ for some 1-form $\beta$ with compact support in
$M\setminus B$. 

(b) Any de Rham cohomology class $\eta\in\ker(j^*)\subset H^2(M,B)$ has a
representative of the form
$$
   \sum_{l=1}^nc_ld\bigl(\sigma(r_l)d\phi_l\bigr)
$$
with constants $c_l\in\R$ and a non-increasing function
$\sigma:[0,1]\to[0,1]$ which equals $1$ near $0$ and $0$ near $1$.  
\end{lemma} 

\begin{proof}
Set $W:=\bar\Sigma=\Sigma\cup B$ and let $\phi$ be the monodromy of the open book.
Note that $H^2(M,B)=H^2(W_\phi,\p W_\phi)$. 
For part Part (b)
note that in the last displayed exact sequence $\ker(j^*)=Im(d^*)$. Now $H^1(B_l)$ is generated by $[d\phi_l]$ and
$d^*[d\phi_l]$ is represented by the 2-form
$d\bigl(\sigma(r_l)d\phi_l\bigr)$. For part (a) recall that the space $V$ in Equation~\eqref{eq:mapping-torus} was not uniquely defined.
The above description of $Im(d^*)$ implies that we can take $V:=d^*(H^1(B))$.
The discussion of the term $\Theta\wedge H^1(W,\p W)$
after Equation~\eqref{eq:mapping-torus} concludes the proof.
\end{proof}

\subsection{Proof of the Existence Theorem~\ref{obd1}}\label{ss:obd1}

Consider an open book $(B,\pi)$, a function $s:\pi_0(B)\rightarrow \{+,-\}$ 
and a cohomology class $\eta\in H^2(M)$ as in Theorem~\ref{obd1}. We proceed in two steps: First we
construct some SHS supported by $(B,\pi)$ whose Reeb vector field and
stabilizing 1-form are given by $\p_\theta$ resp.~$d\theta$ outside a
neighbourhood of the binding; in the second step we arrange the
desired cohomology class $\eta$. 

{\bf Step 1. }Since we are in dimension $3$, 
a stable Hamiltonian structure $(\om=i_RV,\lambda)$ is determined by a
triple $(R,V,\lambda)$ consisting of a vector field $R$, a volume form
$V$ and a 1-form $\lambda$ such that
$$
   L_RV=0,\qquad \lambda(R)=1,\qquad i_Rd\lambda=0.   
$$
We begin by constructing the volume form $V$. Pick a positive area
form $\tilde\om$ on $\Sigma$ satisfying $\tilde\om=d\phi_l\wedge
r_ldr_l$ near each component $B_l$ of the binding. A simple application
of Moser's trick shows that after a compactly supported isotopy of
the monodromy map $\phi$ we may assume that
$\phi^*\tilde\om=\tilde\om$. Thus $\tilde\om$ induces a maximally
nondegenerate form $i_*\tilde\om$ on $M\setminus B$ whose kernel is
spanned by the natural vector field $\partial_{\theta}$, and which
agrees with $d\phi_l\wedge r_ldr_l$ near each component $B_l$ of the
binding. Therefore, the volume form $d\theta\wedge i_*\tilde\om$ on
$M\setminus B$ coincides with 
$$
  V_{loc}:=d\theta\wedge d\phi_l\wedge r_ldr_l
$$ 
near $B_l$ and thus extends as a volume form $V$ on
$M$. Note that the vector field $\p_{\theta}$ on $M\setminus B$  
preserves the form $V\bigl|_{M\setminus B}=d\theta\wedge
i_*\tilde\om$, so the triple $(\p_\theta,V\bigl|_{M\setminus
  B},d\theta)$ defines a SHS on $M\setminus B$. 

Unfortunately, the vector field $\p_{\theta}$ on $M\setminus B$
extends as $0$ over $B$. So we need to modify $\p_{\theta}$ 
in a neighbourhood of the binding. This works as follows.
Let $W_{l}=\{r\le r_0\}$ denote a neighbourhood of $B_l$ where
$V=V_{loc}$. We set $R=s(B_l)\partial_{\phi}+\partial_{\theta}$ 
(i.e. $\om=s(B_l)rdr\wedge d\theta-rdr\wedge d\phi$) near  
and $\lambda=s(B_l)d\phi$ near $B_l$; we set $R=\p_\theta$ and $\lambda=d\theta$
near $\p W_l$. We extend $R$ to the whole of $W_l$ such that it is always linear on
the tori $\{r=const\}$ (i.e of the form $a(r)\p_{\theta}+b(r)\p_\phi$) and 
thus preserves the volume form $V_{loc}$. Moreover, we require that the
$\p_\theta$ component of $R$ is always positive. We
apply Lemma~\ref{lem:stablin} (with slopes $k_-=(s(B_l)1,-1)$ and
$k_+=(0,-1)$ and constants $c_-=c_+=1$)
to find $\lambda$ stabilizing $i_RV_{loc}$ and respecting the
conditions near $B_l$ and $\p W_l$. 
This way we have constructed a SHS $(i_RV,\lambda)$ supported by
$(B,\pi)$ such that $R$ restricts to $B_l$
positively or negatively according to $s(B_l)$. Moreover,
$R=\p_\theta$ and $\lambda=d\theta$ outside the neighbourhood
$\cup_{l=1,...,n}W_{l}$ of $B$. 

{\bf Step 2. }
We still have to take care of the cohomology class. 
Set $\tilde\eta:=[i_RV]\in H^2(M)$. According to Lemma~\ref{algtop}
(a) the cohomology class $\eta-\tilde\eta$ can be represented by
$d\theta\wedge\beta$ with a $1$-form $\beta$ supported in
$M\setminus B$. By modifying $\beta$ slightly if necessary,
we can achieve that its support is contained in $M\setminus
\cup_{l=1,...,n}W_{l}$. Define the closed 2-form
$$\om:=i_RV+d\theta\wedge\beta.$$ 
Its restriction to each page agrees with that of $i_RV$ and is hence a positive
area form, so it satisfies condition (OB2) in
Definition~\ref{def:openbook}. Moreover, $\om$ represents the desired
cohomology class: 
$$
[\om]=[i_RV]+[d\theta\wedge \beta]=\tilde\eta+(\eta-\tilde
\eta)=\eta.
$$ 
We claim that the 1-form $\lambda$ from Step 1 stabilizes
$\om$. Indeed, on $M\setminus supp(\beta)$ we have that $\om=i_RV$ and
so $\lambda$ stabilizes $\om$ by construction. 
On $supp(\beta)\subset M\setminus \cup_{l=1,...,n}W_{l}$ 
the vector field $R$ and the 1-form $\lambda$ are given by
$$
   R|_{supp(\beta)}=\p_{\theta}\,,\,\,\,\lambda|_{supp(\beta)}=d\theta.
$$
In particular, we have $d\lambda=0$ and 
$$
   \lambda\wedge \om=
   d\theta\wedge (i_{\p_{\theta}}V+d\theta\wedge\beta)=V.
$$  
This shows that $(\om,\lambda)$ is a SHS, and since $\lambda$
satisfies condition (OB1) by construction, $(\om,\lambda)$ is
supported by $(B,\pi)$. 
This finishes the proof of Theorem \ref{obd1}.

\begin{remark}
Note that the SHS $(\om,\lambda)$ constructed in the preceding proof
has the following property: Near each component 
$B_l$ of the binding, $(\om,\lambda)$ is $T^2$-invariant with respect 
to the $T^2$ action given by shifts along $\phi_l$ and $\theta$.
\end{remark}

\subsection{Hamiltonian structures on mapping tori}
\begin{lemma}\label{lm:mtl}
Let $(W,\p W)$ be a compact oriented surface with boundary and let
$X\rightarrow S^1$ be a fibration over the circle with fibre $W$. Let
$(\om_0,\lambda_0)$ and $(\om_1,\lambda_1)$ be two SHS on $X$ with the
following properties:
\begin{itemize}
\item on some neighbourhood $N$ of $\p X$ we have
  $\om_0=\om_1$ and $\lambda_0=c_0d\theta,\, \lambda_1=c_1d\theta$ for
  some positive constants $c_0,\,c_1$;  
\item $d\theta\wedge\om_0>0$ and $d\theta\wedge \om_1>0$; 
\item $[\om_1-\om_0]=0\in H^2(X,\p X)$. 
\end{itemize}
Then $(\om_0,\lambda_0)$ and $(\om_1,\lambda_1)$ are homotopic via a 
stable homotopy $(\om_t,\lambda_t)$ satisfying 
$d\theta\wedge \om_t>0$ and $[\om_t-\om_0]=0\in H^2(X,\p
X)$ for all $t$. Moreover, on $N$ we have $\om_t=\om_0=\om_1$ and
$\lambda_t=\rho(t)d\theta$ 
for a positive function $\rho:[0,1]\to \R$ with $\rho(0)=c_0$ and
$\rho(1)=c_1$.
\end{lemma}

\begin{proof}
As both $\lambda_0$ and $d\theta$ stabilize $\om_0$, we can homotope
$(\om_0,\lambda_0)$ to $(\om_0,d\theta)$ by interpolating linearly
between $\lambda_0$ and $d\theta$. Since on $N$ we have $\lambda_0=c_0d\theta$,
this homotopy has the desired form when restricted to $N$. Similarly for $(\om_1,\lambda_1)$. 
So it remains to homotope $(\om_0,d\theta)$ to $(\om_1,d\theta)$.
Set 
$$
   \om_t := (1-t)\om_0+t\om_1,\qquad t\in [0,1].
$$ 
This a homotopy of closed $2$-forms rel $\p X$ joining $\om_0$ and
$\om_1$. Now $d\theta\wedge\om_0>0$ and $d\theta\wedge \om_1>0$ 
imply that $d\theta\wedge \om_t>0$ for all $t$, so
$\{(\om_t,d\theta)\}_{t\in [0,1]}$ is a homotopy of SHS transverse to
the fibers. The last condition holds since 
$[\om_t-\om_0]=[t(\om_1-\om_0)]=0\in H^2(X,\p X)$.
\end{proof}

\subsection{Standardization near the binding}\label{ss:binding}
For the remainder of Section \ref{sec:openbook} the following
terminology will prove useful. 

\begin{definition}\label{def:sspecial} 
Let $s\in \{+,-\}$ be a sign.
We call a SHS $(\om,\lambda)$ on $S^1\times D^2$ {\em $s$-special} (in
words --- positively or negatively special) if  
$$
   \om = rdr\wedge(sd\theta-d\phi)
$$
\end{definition}

It is important to note that if for a positive real number $C$ the SHS
$(C\om,\lambda)$ is $s$-special 
on $S^1\times D^2$ (with $D^2=\{r\le b\}$ for some $b>0$), then we can
homotope $\om$ rel $\p(S^1\times D^2)$, keeping the same kernel
foliation, to get that on a smaller solid torus the SHS
$(\om,\lambda)$ is $s$-special. Indeed, constancy of the slope allows  
us to keep $\lambda$ stabilizing $\om$ fixed throughout the
homotopy. Let $h(r):=\frac{1}{2}r^2(s1,-1)$. Take a diffeomorphism
$\rho$ of $(0,b]$ supported away from $b$ and such that for some
$\eps>0$ we have $\rho(r)=C^{1/2}r$ for $r\le \eps$.  
Since the linear interpolation between any two orientation preserving
diffeomorphisms of $(0,b]$ supported away from the $b$ is running
through diffeomorphisms of the same type, we set 
$\rho_t(r):=(1-t)r+t\rho(r)$ and $h_t:=h\circ\rho_t$. Now
$\om_t=d\alpha_{h_t},\, t\in [0,1]$ is the required homotopy, since
$\om_0=\om$ and $(\om_1,\lambda)$ is $s$-special on $\{r\le \eps\}$.  
 
Now we show that a SHS can be deformed to be special near the binding
of an open book, or more 
generally, near a finite collection of periodic orbits. 

\begin{prop}[Standardization near the binding]\label{prop:obd}
Let $(\om,\lambda)$ be a SHS supported by an open book $(B,\pi)$. Then
there exists a homotopy of SHS $(\om_t,\lambda_t)$ supported by
$(B,\pi)$ such that $(\om_0,\lambda_0)=(\om,\lambda)$ and
$(\om_1,\lambda_1)$ is $s(B_l)$-special near each binding component $B_l$. 
\end{prop}

\begin{proof}
Let $f:=d\lambda/\om$. For each component $B_l$ of the binding we
distinguish three cases.\\ 
Case 1: $f\equiv a\neq 0$ on a neighbourhood of $B_l$.\\  
Case 2: $f\equiv 0$ on a neighbourhood of $B_l$.\\
Case 3: $f$ is not constant on any neighbourhood of $B_l$.  

In Case 3 we apply Corollary~\ref{cor:thick-a} to the value
$a:=f(B_l)$. So we find a new stabilizing 1-form $\tilde\lambda$ for
$\om$ such that $\tilde f=d\tilde\lambda/\om$ is constant on a
neighbourhood of $B_l$. Since $\om$ has not changed,
$(\om,\tilde\lambda)$ is still supported by $(B,\pi)$ and so are the
SHS $(\om_t,\lambda_t)=\bigl(\om,(1-t)\lambda+t\tilde\lambda\bigr)$,
$t\in[0,1]$. Thus Case 3 reduces to Cases 1 and 2, which are treated
in Lemmas~\ref{lem:obd-cont} and~\ref{lem:obd-fol} below.  

Note that in Case 1 the sign of $a$ coincides with the
sign of $\lambda$ as a contact form. If $a>0$, then $\lambda$ is positive.
We apply Lemma \ref{lem:obd-cont} to the form $\lambda$ and note that the sign of 
$T=\int_{\gamma}\lambda$ is just $s(B_l)$. This way we obtain a homotopy rel
$\p U$ of positive contact forms $\{\lambda_t\}_{\{t\in [0,1]\}}$ on $U$ from $\lambda_0=\lambda$ to
$\lambda_1$ which agrees near $\gamma$ with a positive multiple of
$r^2d\theta+(1-r^2)d\phi$ if $s(B_l)=+$, or $-r^2d\theta-(1+r^2)d\phi$ if $s(B_l)=-$. Then
$(\om_t=\frac{1}{a}d\lambda_t,\lambda_t)$ on $U$ gives the desired
stable homotopy. 

If $a<0$, then $\lambda$ negative. As before the sign of 
$T=\int_{\gamma}\lambda$ is just $s(B_l)$.
Let $Rev_\phi$ denote the map on a small tubular
neighbourhood $U$ of $B_l$ sending $\phi$ to $-\phi$. This is 
just reversing the direction along this binding component. Then the
pullback $Rev_\phi^*\lambda$ is a positive contact form. Moreover,
$d\theta\wedge d\lambda<0$ (as $a$ is negative) implies that
$d\theta\wedge d(Rev_\phi^*\lambda)>0$. This means that we can apply  
Lemma \ref{lem:obd-cont} to the form
$\tilde\lambda:=Rev_\phi^*\lambda$ with 
$T=\int_{\gamma}\tilde\lambda$ having the sign opposite to $s(B_l)$. Applying $Rev_\phi^*$ to the
homotopy provided by Lemma~\ref{lem:obd-cont} we obtain a homotopy rel
$\p U$ of negative contact forms $\{\lambda_t\}_{\{t\in [0,1]\}}$ on $U$ from $\lambda_0=\lambda$ to
$\lambda_1$ which agrees near $\gamma$ with a positive multiple of
$-r^2d\theta+(1+r^2)d\phi$ if $s(B_l)=+$, or $r^2d\theta-(1-r^2)d\phi$ if $s(B_l)=-$. Then
$(\om_t=\frac{1}{a}d\lambda_t,\lambda_t)$ on $U$ gives the desired
stable homotopy. Note that for both signs of $a$ we get that $\om_1$ restricts near $\gamma$
as a positive multiple of either 
$rdr\wedge(d\theta-d\phi)$ if $s(B_l)=+$ or $-rdr(d\theta+d\phi)$ if $s(B_l)=-$.
\end{proof}

\begin{corollary}[Standardization near periodic orbits]
Let $(\om,\lambda)$ be a SHS on a closed 3-manifold $M$ and
$\gamma_1,\dots,\gamma_k$ periodic orbits. Let $s_1,\dots s_k$ be a finite sequence of signs. 
Then there exists a homotopy of SHS $(\om_t,\lambda_t)$ having the $\gamma_i$ as a
periodic orbits such that $(\om_0,\lambda_0)=(\om,\lambda)$ and
$(\om_1,\lambda_1)$ is $s_i$-special near the orbits
$\{\gamma_i\}_{i=1,...,k}$.  
\end{corollary}

\begin{proof}
It suffices to consider the case of one periodic orbit $\gamma$ and one sign $s\in \{-,+\}$. 
If $f:=d\lambda/\om$ is not constant on any neighbourhood of $\gamma$
we modify $\lambda$ as in the preceding proof. So we may assume that
$f$ is constant on a neighbourhood $U\cong S^1\times D^2$ of
$\gamma$. 

Assume that $s=+$. Let $\phi$ be the coordinate along
$\gamma=S^1\times\{0\}$ and $(r,\theta)$ be polar coordinates on
$D^2$. For any $k\in\Z$, the map $\pi_k(\phi,r,\theta) := k\phi +
\theta$ (for $r>0$) defines an open book decomposition near
$\gamma$. A short computation shows that the Reeb vector field of
$(\om,\lambda)$ is transverse to the pages of this open book for $k$
sufficiently large. 
Applying Lemmas~\ref{lem:obd-cont} and~\ref{lem:obd-fol} below to this
open book decomposition on $U$ yields the result. 

If $s=-$. Let $\phi$ be the coordinate along $\gamma=S^1\times\{0\}$ in the reversed direction.
Let $r,\theta,\pi_k$ be as above. The Reeb vector field of $(\om,\lambda)$ is transverse to the pages of this open book for $k$ negative and sufficiently large in the absolute value.

\end{proof}

For the remainder of this subsection, we consider a neighbourhood $U$
of a component $\gamma$ of the binding. We fix coordinates 
$(\phi,r,\theta)$ on 
$$
   U \cong S^1\times D^2
$$
supported by the open book, so $\phi$ is the coordinate along
$\gamma=S^1\times\{0\}$ and $(r,\theta)$ are polar coordinates on
$D^2$ such that the open book projection is given for $r>0$ by
$\pi(\phi,r,\theta)=\theta$. We will also sometimes use cartesian
coordinates $(x,y)$ on $D^2$. Note that a SHS $(\om,\lambda)$ on $U$
is supported by the open book iff 
\begin{equation}\label{eq:supported}
   d\theta\wedge \om > 0 \text{ on the set }\{r>0\}. 
\end{equation}
(It follows that the 3-form extends continuously to $r=0$ but may
become zero there). We consider a SHS $(\om,\lambda)$
satisfying~\eqref{eq:supported} with $f=d\lambda/\om$ being {\em
  constant}.  

In subsequent constructions we will repeatedly use the following 
technical lemma.

\begin{lemma}\label{lm:cutoff}
For all $\delta,\varepsilon>0$ there exists a smooth function
$\rho:[0,\delta]\rightarrow[0,1]$ with the following properties:
$\rho$ is nonincreasing, constant $1$ in a neighbourhood of $0$,
constant $0$ in a neighbourhood of $\delta$, and 
$$
   |x\rho'(x)|<\varepsilon
$$ 
for all $x\in [0,\delta]$. 
\end{lemma}

\begin{proof}
The function $f(x):=\eps\ln(\delta/x)$ satsifies $x f'(x)=-\eps$,
$f(\delta)=0$ and $f(\delta e^{-1/\eps})=1$. Now shift the
function $\max(f,1)$ slightly to the left, extend by $0$ to the right
and smoothen it. 
\end{proof}

{\bf The contact case. }

\begin{lemma}\label{lem:obd-cont}
Let $\lambda$ be a positive contact form on $U=S^1\times D^2$ 
satisfying~\eqref{eq:supported} (with $\om=d\lambda)$. Then
there exists a homotopy rel $\p U$ of contact forms $\lambda_t$ 
satisfying~\eqref{eq:supported} such that $\lambda_0=\lambda$
and $\lambda_1=|T|\lambda_\st$ near $\gamma=S^1\times\{0\}$, where
$T=\int_\gamma\lambda$ and
$$
   \lambda_\st: = \begin{cases}
   r^2d\theta + (1-r^2)d\phi & \text{ if }T>0, \\
   -r^2d\theta - (1+r^2)d\phi & \text{ if }T<0.
   \end{cases}
$$
\end{lemma}
As preparation for the proof, let us derive under which conditions the
interpolation between two contact forms is again contact. 

\begin{lemma}\label{lem:interpol}
Let $\mu,\lambda$ be two positive contact forms with Reeb vector
fields $R_\mu,R_\lambda$ on a compact oriented
3-manifold $M$ (possibly with boundary). Fix a metric on $M$ to
define norms of differential forms. Let $f:M\to[0,1]$ be a smooth
function. Then
$$
   \tilde\lambda := (1-f)\lambda + f\mu
$$
is again a positive contact form provided that $\lambda(R_\mu)\geq 0$,
$\mu(R_\lambda)\geq 0$, and
\begin{equation}\label{eq:interpol}
   |df|\,|\lambda|\,|\mu-\lambda| < \frac{1}{2}\min(\lambda\wedge
   d\lambda,\mu\wedge d\mu). 
\end{equation}
\end{lemma}

\begin{proof}
We compute
\begin{align*}
   d\tilde\lambda 
   &= df\wedge(\mu-\lambda) + (1-f)d\lambda + f\,d\mu,\cr
   \tilde\lambda\wedge d\tilde\lambda 
   &= (1-f)^2\lambda\wedge d\lambda + f^2\mu\wedge d\mu +
   f(1-f)[\lambda\wedge d\mu+\mu\wedge d\lambda] +
   df\wedge\mu\wedge\lambda. 
\end{align*}
On the right hand side of the last equation the sum of the first two
terms is $\geq \frac{1}{2}\min(|\lambda\wedge
d\lambda|,|\mu\wedge d\mu|)$. The third term in nonnegative under
the hypothesis $\lambda(R_\mu)\geq 0$, $\mu(R_\lambda)\geq 0$. The
last term can be rewritten as $df\wedge(\mu-\lambda)\wedge\lambda$, so
its norm is at most $|df|\,|\lambda|\,|\mu-\lambda|$ and the lemma
follows. 
\end{proof}

\begin{lemma}\label{lem:interpol2}
Let $\lambda$ and $\mu$ be two contact forms on $S^1\times D^2$ 
Let $\phi$ be the angular coordinate on
$S^1$ and $(r,\theta)$ be polar coordinates on $D^2$. 
Let $\sigma:[0,1]\to\R$ be a radial function, constant near $r=0$ and
supported in $[0,\delta]$, with $|r\sigma'(r)|\le \varepsilon$. 
Define the 1-form $\tilde\lambda := (1-\sigma(r))\lambda +
\sigma(r)\mu$. 

(i) If the core circle $\{r=0\}$ (oriented by $d\phi$)  
is a common Reeb orbit for $\lambda$ and $\mu$ and $\lambda=\mu$
along $\{r=0\}$, then $\tilde\lambda$ is again contact for $\delta,
\varepsilon$ sufficiently small. 

(ii) If $d\theta\wedge d\lambda\geq \beta>0$, $d\theta\wedge
d\mu\geq \beta>0$ and $dr\wedge d\theta\wedge (\mu-\lambda)=O(r)$ near
$r=0$, then $\tilde\lambda$ satisfies $d\theta\wedge
d\tilde\lambda>0$ for $\delta, \varepsilon$ sufficiently small. 
\end{lemma}

\begin{proof}
For (i) let us check the conditions in Lemma~\ref{lem:interpol} (with
$f=\sigma$ and and $M$ the region $\{r\leq\delta\}$). 
Since $\lambda=\mu$ and $R_\lambda=R_\mu$ at $r=0$, the
conditions $\lambda(R_\mu)\geq 0$ and $\mu(R_\lambda)\geq 0$ are
satisfied for $\delta$ sufficiently small. For
condition~\eqref{eq:interpol} set  
$$
   C := \frac{1}{2}\min(\lambda\wedge d\lambda,\mu\wedge d\mu) > 0
$$
and note that $|d\sigma|=|\sigma'(r)|$ and $|\mu-\lambda|=O(r)$,
so the estimate $|r\sigma'(r)|\le \varepsilon$ gives us
$$
   |d\sigma|\,|\lambda|\,|\mu-\lambda| = r|\sigma'(r)|O(1) = \eps
   O(1) < C 
$$
for $\eps$ sufficiently small. Thus $\tilde\lambda$ is a contact
form.

(ii) follows from 
\begin{align*}
   d\theta\wedge d\tilde\lambda 
   &= (1-\sigma)d\theta\wedge d\lambda +
   \sigma d\theta\wedge d\mu + \sigma'(r)d\theta\wedge
   dr\wedge(\mu-\lambda) \cr 
   &\geq \beta - C|\sigma'(r)r| \geq \beta-C\eps
   > 0
\end{align*}
for $\delta, \varepsilon$ sufficiently small. 
\end{proof}

\begin{proof}[Proof of Lemma~\ref{lem:obd-cont}]
{\bf Step 1. }
By assumption, $\gamma=S^1\times\{0\}$ is an orbit of the Reeb vector
fields $R,R_\st$ of $\lambda,\lambda_\st$. In particular,
$\lambda|_\gamma,\lambda_\st|_\gamma$ are volume forms of total
volume $T=\int_\gamma\lambda$ and $sign(T)$, respectively. So after pulling
back $\lambda$ by an isotopy rel $\p U$ of diffeomorphisms of the form
$$
   F_t(\phi,r,\theta)= \bigl(f_t(\phi,r),r,\theta\bigr)
$$
we may assume that $\lambda|_\gamma = |T|\lambda_\st|_\gamma$. Note that 
condition~\eqref{eq:supported} is preserved because
$F^*d\theta=d\theta$. To simplify notation, we will replace $\lambda$
by $\lambda/|T|$ (and insert back $|T|$ at the end of the proof), so
we have $\lambda|_\gamma = 
\lambda_\st|_\gamma$. Note that now $R=sign(T)\p_\phi$ along $\gamma$. 

{\bf Step 2. }
Next we improve condition~\eqref{eq:supported}. Pick a non-increasing
function $h;[0,1]\to[0,1]$ which vanishes near $1$ and equals
$1-r^2/2$ near $0$. Since $\lambda$ is a contact form, for
sufficiently small $\beta>0$ the form $\lambda+\beta h(r)d\phi$ is
also contact. Note that  
$$
   d\theta\wedge d(\lambda+\beta h(r)d\phi) = d\theta\wedge
   d\lambda - \beta h'(r)dr\wedge d\theta\wedge d\phi.
$$ 
Both terms are nonnegative and near $r=0$ the second term equals
$\beta r\,dr\wedge d\theta\wedge d\phi$, so the whole expression has
norm $\geq\beta$. Therefore, after replacing $\lambda$ by
$\lambda+\beta h(r)d\phi$, is satisfies the following quantified
version of condition~\eqref{eq:supported}:
\begin{equation}\label{eq:supported2}
   d\theta\wedge d\lambda \geq \beta > 0. 
\end{equation}
Note that the linear homotopy from $\lambda$ to $\lambda+\beta
h(r)d\phi$ is contact and satisfies~\eqref{eq:supported}, and the
condition $\lambda|_\gamma = \lambda_\st|_\gamma$ from Step 1 is
preserved. 

{\bf Step 3. }
We write 
$$
   \lambda = l_1dr+l_2d\theta+l_3d\phi
$$
with functions $l_i(\phi,r,\theta)$. Recall that in cartesian
coordinates 
$$
   dr = \frac{x\,dx + y\,dy}{r},\qquad d\theta =
   \frac{x\,dy-y\,dx}{r^2}. 
$$
Thus smoothness of $\lambda$ at $r=0$ implies that $l_1=O(r)$ and
$l_2=O(r^2)$ near $r=0$. Moreover, the condition $\lambda|_\gamma =
\lambda_\st|_\gamma$ from Step 1 implies $l_3=sign(T)+O(r)$. Let us compute 
$$
   d\lambda = (l_{2r}-l_{1\theta})dr\wedge d\theta +
   (l_{3r}-l_{1\phi})dr\wedge d\phi +
   (l_{2\phi}-l_{3\theta})d\phi\wedge d\theta. 
$$
From $i_{\p_\phi}d\lambda=0$ along $\gamma$ we deduce
$l_{3r}|_{r=0}=l_{1\phi}|_{r=0}=0$ (since $l_1=O(r)$), so we have found
the conditions
\begin{equation}\label{eq:l}
   l_1=O(r),\qquad l_2=O(r^2),\qquad l_3=sign(T)+O(r^2). 
\end{equation}
{\bf Step 4. }
Take now a cutoff function $\rho$ as in Lemma~\ref{lm:cutoff}, with
small constants $\delta,\eps>0$ to be specified later, and consider
the 1-form
$$
   \lambda_t := (1-t\rho(r))\lambda + t\rho(r)\lambda_\st. 
$$
Let us check the conditions in Lemma~\ref{lem:interpol2} (with $\sigma=t\rho$,
$\mu=\lambda_\st$ and $M$ the region $\{r\leq\delta\}$). Since
$\lambda=\lambda_\st=sign(T)d\phi$ and $R=R_\st=sign(T)\p_\phi$ at $r=0$, the
conditions of Lemma~\ref{lem:interpol2} (i) are satisfied. The
conditions of Lemma~\ref{lem:interpol2} (ii) hold in view
of~\eqref{eq:supported2} (and the analogous condition for
$\lambda_\st$) and 
$$
   dr\wedge d\theta\wedge(\lambda-\lambda_\st) =
   \Bigl(l_3-sign(T)(1+r^2)\Bigr)dr\wedge d\theta\wedge d\phi =
   O(r^2)dr\wedge d\theta\wedge d\phi = O(r),
$$
where we have used~\eqref{eq:l}.  
Hence $\lambda_t$ is a contact form that satisfies
condition~\eqref{eq:supported} and 
agrees with $\lambda$ near $\p U$. Since
$\lambda_1=\lambda_\st$ near $\gamma$, this concludes (after
inserting back the constant $|T|$ from Step 1) the proof of
Lemma~\ref{lem:obd-cont}.  
\end{proof}

{\bf The foliation case.}

\begin{lemma}\label{lem:obd-fol}
Let $(\om,\lambda)$ be a SHS on $U=S^1\times D^2$ with $d\lambda=0$
and satisfying~\eqref{eq:supported}. Then
there exists a homotopy rel $\p U$ of SHS $(\om_t,\lambda_t)$
with $d\lambda_t=0$ and satisfying~\eqref{eq:supported} such that
$(\om_0,\lambda_0)=(\om,\lambda)$ and 
$$
   \lambda_1=T\,d\phi,\qquad \om_1 = k\,r\,dr\wedge (sign(T)d\theta - d\phi) 
$$
near $\gamma=S^1\times\{0\}$, for some constants $k>0$ and
$T=\int_\gamma\lambda$. 
\end{lemma}

\begin{proof}
{\bf Step 1. } 
With $T=\int_\gamma\lambda$ the forms $\lambda$ and $T\,d\phi$ are
cohomologous, thus $Td\phi=\lambda+df$ for some function
$f:U\to\R$. Moreover, both forms stabilize $\om$ on some neighbourhood
$V\subset U$ of $\gamma$. Pick a cutoff function $g:U\to\R$ which
vanishes near $\p U$ and equals $1$ on $V$. Then $\lambda_t := \lambda
+ td(gf)$ is a homotopy rel $\p U$ of closed 1-forms stabilizing $\om$
with $\lambda_0=\lambda$ and $\lambda_1=T\,d\phi$ on the neighbourhood
$V$ of $\gamma$. In the following steps we will keep $\lambda_1$ fixed
and modify $\om$ in $V$ through HS stabilized by $\lambda_1$, or
equivalently, by $Td\phi$. 

{\bf Step 2. }
Recall that $(x,y)$ are cartesian coordinates on $D^2$. Let
$$
   \alpha = A\,dx+B\,dy+C\,d\phi
$$ 
be the chosen primitive of $\om$ in $U$. Consider the function  
$$
   h(\phi,x,y):=-A(\phi,0,0)x-B(\phi,0,0)y
$$ 
and the cutoff function $g$ from Step 1. Then $\alpha+d(gh)$ is a
primitive of $\om$ which agrees with $\alpha$ near $\p U$ and whose
coefficients in front of $dx$ and $dy$ vanish at $\gamma$. We denote
this new primitive again by $\alpha$. 

{\bf Step 3. }
Let $\alpha=A\,dx+B\,dy+C\,d\phi$ be the primitive of $\om$ from Step
2 with $A\bigl|_{\gamma}=B\bigl|_\gamma=0$. This implies that $A=O(r)$
and $B=O(r)$ at $\gamma$, thus $d\theta\wedge d(A\,dx+B\,dy)$ is bounded
in a neighborhood of $\gamma$. As $\gamma$ is a Reeb orbit of
$(\om,\lambda)$, the expression $d\theta\wedge \om$ is bounded and
thus $d\theta\wedge d(C\,d\phi)$ is also bounded. Let $\sigma(r)$ be a
radial cutoff function supported in $V$ which equals $1$ on a
neighbourhood $W\subset V$ of $\gamma$ Consider the compactly
supported homotopy $\{\alpha_t\}_{t\in [0,1]}$ of $1$-forms 
$$
   \alpha_t:=\alpha-t\sigma(r)C\,d\phi
$$   
on $V$ and set $\om_t:=d\alpha_t$. Since by construction $\om_t-\om$
is proportional to $d\phi$, it has zero wedge with $d\phi$ we thus
$\om_t$ is a HS stabilized by $d\phi$ for all $t\in [0,1]$. 
Note that $\alpha_0=\alpha$ and 
$$
   \alpha_t|_W = \alpha-tC\,d\phi,\qquad \alpha_1|_W = A\,dx+B\,dy
$$  
Observe that {\em $d\theta\wedge \om_t$ is bounded uniformly in $t$} 
since both $d\theta\wedge d\alpha$ and $d\theta\wedge d(C\,d\phi)$ are
bounded. This boundedness will be traced throughout as we go in Steps
4--5 and then used in Step 6 to get transversality to the pages. 

{\bf Step 4. } 
After the homotopy in Step 3 we can assume that 
$$
   \alpha=A\,dx+B\,dy
$$ 
is a primitive of $\om$ in a neighbourhood $W\subset V$ of $\gamma$
and $A$ and $B$ both vanish at $\gamma$. Next we want to replace $A$
and $B$ with their linear (in $(x,y)$) parts at $\gamma$. For this we set
\begin{gather*}
   A_1(\phi,x,y) := A_x(\phi,0,0)x+A_y(\phi,0,0)y,\cr
   B_1(\phi,x,y) := B_x(\phi,0,0)x+B_y(\phi,0,0)y.
\end{gather*}
Let $\rho(r)$ be a cutoff function as in Lemma~\ref{lm:cutoff}, with
small $\delta,\eps$ to be specified later and support in $W$, and set  
$$
   \alpha_{lin}:=A_1dx+B_1dy,\qquad 
   \alpha_t:=\alpha-t\rho(r)(\alpha-\alpha_{lin}),\qquad
   \om_t:=d\alpha_t 
$$
for $t\in [0,1]$. Note that at $\gamma$ we have the following
estimates:  
$$
   \alpha-\alpha_{lin}=O(r^2), \qquad d(\alpha-\alpha_{lin})=O(r).
$$ 
They imply
$$
   d\phi\wedge d\bigl(\rho(r)(\alpha-\alpha_{lin})\bigr) = \rho\,
   d\phi\wedge d(\alpha-\alpha_{lin})+\rho'(r)d\phi\wedge dr\wedge
   (\alpha-\alpha_{lin}) = O(r). 
$$
Thus $d\phi\wedge\om_t>0$ for $\delta$ sufficiently small, which shows
that the $\om_t$ are HS stabilized by $d\phi$.  
Note that $\alpha_0=\alpha$ and $\alpha_1=\alpha_{lin}$ on a
neighbourhood of $\gamma$. Moreover, $d\theta\wedge \om_t$ remains
bounded because the first jet of $\alpha_t$ at $\gamma$ remains
unchanged during the homotopy. 

{\bf Step 5. }
After the homotopy in Step 4 we can assume that $\alpha=Adx+Bdy$ is a
primitive of $\om$ in some neighbourhood $Z\subset W$ of
$\gamma$ and $A$ and $B$ both linear in $x$ and $y$ on $Z$, i.e.
$$
   A=A_1x+A_2y,\qquad B=B_1x+B_2y
$$ 
with functions $A_i,B_i$ of $\phi$. Note that 
$$
   Td\phi\wedge d\alpha=T(B_1-A_2)d\phi\wedge dx\wedge dy>0,
$$
thus $T(B_1-A_2)>0$ for all $\phi\in S^1$. We set 
$$
   c:=min_{S^1}|B_1-A_2|>0,\qquad 
   M:=max_{S^1}sign(T)B_1,\qquad   
   m:=min_{S^1}sign(T)A_2.
$$ 
Let $\rho$ be a function as in Lemma~\ref{lm:cutoff}, with 
$\delta>0$ so small that $\{r^2\le \delta\}\subset Z$ and
$\varepsilon>0$ to be specified later. Set 
\begin{align*}
   A_t &:= \bigl(1-t\rho(r^2)\bigr)A_1x+\bigl[t\rho(r^2)sign(T)m+
   \bigl(1-t\rho(r^2)\bigr)A_2\bigr]y, \cr
   B_t &:= \bigl[t\rho(r^2)sign(T)M+\bigl(1-t\rho(r^2)\bigr)B_1\bigr]x +
   \bigl(1-t\rho(r^2)\bigr)B_2y, \cr
   \alpha_t &:= A_tdx+B_tdy,\quad t\in[0,1].
\end{align*}
This is a homotopy of $1$-forms supported in $Z$. First we check that 
the $\om_t:=d\alpha_t$ are HS stabilized by $d\phi$. So we need to check
the inequality
$$ 
   Td\phi\wedge d\alpha=T(B_{tx}-A_{ty})d\phi\wedge dx\wedge dy>0.
$$
We write out 
\begin{align*}
   T(B_{tx}-A_{ty}) = &[t\rho(r^2)|T|M+T(1-t\rho(r^2))B_1] \cr
   &+ T\{2t(M-B_1)x^2\rho'(r^2)-2tB_2xy\rho'(r^2)\} \cr
   &- [t\rho(r^2)|T|m+T(1-t\rho(r^2))A_2] \cr
   &- T\{2t(m-A_2)y^2\rho'(r^2)-2tA_1xy\rho'(r^2)\}.
\end{align*}
Using $M-m\geq c$, we estimate terms in square brackets from below and
the other terms from above to apply the triangle inequality:
\begin{align*}
   &[t\rho(r^2)|T|M+T(1-t\rho(r^2))B_1]-[t\rho(r^2)|T|m+T(1-t\rho(r^2))A_2] \cr
   &= t\rho(r^2)|T|(M-m)+T(1-t\rho(r^2))(B_1-A_2)\ge
   |T|(t\rho(r^2)(M-m)+(1-t\rho(r^2))c) \cr
   &\ge |T|c, \cr
   &|\{2t(M-B_1)x^2\rho'(r^2)-2tB_2xy\rho'(r^2)\} -
   \{2t(m-A_2)y^2\rho'(r^2)-2tA_1xy\rho'(r^2)\}| \cr 
   &\le 2t\varepsilon(|M-B_1|+|B_2|+|m-A_2|+|A_1|) \cr
   &< |T|c
\end{align*}
for $\varepsilon$ small enough. 

Second, near $r=0$ the coefficients of $\alpha_t$ are linear in
$(x,y)$. In particular $A_t=O(r)$ and $B_t=O(r)$ at $\gamma$, thus
$d\theta\wedge \om_t$ is bounded uniformly in $t$. Moreover, for 
$t=1$ we have 
$$
   \alpha_1=sign(T)my\,dx+sign(T)Mx\,dy,\qquad \om_1=sign(T)(M-m)dx\wedge dy.
$$
{\bf Step 6. } 
In Steps 2-5 we constructed a homotopy $\{\alpha_t\}_{t\in [0,1]}$ of
1-forms supported is $V$ such that the HS $\om_t=d\alpha_t$ are    
all stabilized by $d\phi$, and $d\theta\wedge \om_t$ is bounded
uniformly in $t\in [0,1]$. Note that $\om_t=\om$ outside the
relatively compact neighbourhood $W\subset V$ of $\gamma$ from Step
3. We may assume that all neighbourhoods are chosen radially
symmetric. 

Pick a function $F(r)$ with support in $V$ and 
$$
   F\bigl|_W = 1 - r^2/2
$$
Set 
$$
   \alpha_{tr}:=F(r)d\phi,\qquad \om_{tr}:=d\alpha_{tr} =
   F'(r)dr\wedge d\phi.
$$ 
Note the following two crucial properties of $\om_{tr}$: 
$$
   d\phi\wedge\om_{tr}=0,
$$ 
and
$$
   d\theta\wedge\om_{tr}|_W = d\phi\wedge rdr\wedge d\theta,
$$ 
i.e.~$d\theta\wedge\om_{tr}$ is positive and bounded from below. The
first property shows that for any $t\in [0,1]$ and any constant $k$
the form $\om_t+k\om_{tr}$ is stabilized by $d\phi$. The second
property together with boundedness of $d\theta\wedge \om_t$ uniformly
in $t$ shows that for large enough positive $k$ we have
$d\theta\wedge(\om_t+k\om_{tr})>0$ on $M$ for all $t$. Since
$d\theta\wedge\om_0>0$, there exists a nonnegative function
$\mathcal{A}:[0,1]\to\R$ with $\AA(0)=0$ 
 and $\AA(1)=k$ such that the homotopy  
$$
   \{(\tilde\om_t=\om_t+\AA(t)\om_{tr},d\phi)\}_{t\in [0,1]}
$$ 
of SHS has the property that 
$$
   d\theta\wedge\tilde\om_t>0
$$ 
for all $t\in[0,1]$. Finally, note that 
$$
   \tilde\om_1 = r\,dr\wedge\bigl(sign(T)(M-m)d\theta - k\,d\phi\bigr)
$$
in some neighbourhood $U_1=\{r\leq r_1\}$ of $\gamma$.

{\bf Step 7. }For $r_1$ as in Step 6, let $\rho(r)$ be a positive
function which equals $\frac{M-m}{k}-1$ near $0$ and $0$ near $r_1$. The form
$\tilde\om_1$ on $U_1=\{r\leq r_1\}$ from Step 6 can be linearly
homotoped to 
$$
   r\,dr\wedge\bigl(sign(T)(M-m)d\theta - k(\rho+1)\,d\phi\bigr).
$$
The resulting homotopy of HS is stabilized by $Td\phi$ and
transverse to the pages. Moreover the cohomology class remains
constant as the form $\rho(r)rdr\wedge d\phi$ has a primitive
compactly supported in $U_1$. The last $2$-form agrees with  
$$
   (M-m)\,r\,dr\wedge\bigl(sign(T)d\theta - d\phi\bigr)
$$
in some smaller neighbourhood of $\gamma$. This
completes the proof of Lemma~\ref{lem:obd-fol}, and hence of
Proposition~\ref{prop:obd}. 
\end{proof}

\subsection{Proof of the Uniqueness Theorem~\ref{obd2}}\label{ss:obd2}

Let $(\phi_l,r_l,\theta)$ be standard coordinates as above near
each binding component $B_l$, $l=1,\dots,n$.
Let $(\om,\lambda)$ and $(\tilde\om,\tilde\lambda)$ be two SHS
supported by the open book $(B,\pi)$ such that $[\om]=[\tilde\om]$ and $s_l(\om,\lambda)=s_l(\tilde\om,\tilde\lambda)=:s_l$. 

{\bf Step 1. } We apply Proposition \ref{prop:obd} to the SHS
$(\om,\lambda)$ and $(\tilde\om,\tilde\lambda)$ to homotope them to
$(\om_1,\lambda_1)$ resp.~$(\tilde\om_1,\tilde\lambda_1)$ which are
$s_l$-special near each $B_l$. 

{\bf Step 2. } Now we may assume that $(\om,\lambda)$ and
$(\tilde \om,\tilde\lambda)$ are both $s_l$-special when restricted to 
some tubular neighbourhood $V_l$ of a  binding component $B_l$. In
particular, $\om=\tilde\om$ on $V_l$.  
After a further homotopy of stabilizing forms rel $\p V_l$ we may
assume that $\tilde\lambda=m_l\lambda$ for some $m_l>0$ on some
smaller tubular neighbourhood (called $V_l$ again) of $B_l$. Now we
homotope $\lambda$ and $\tilde\lambda$ rel $\p V_l$ to achieve that 
$\lambda|_{\{r_l\in [a,b]\}}=n_ld\theta$ for some $n_l>0$ (and thus
$\tilde\lambda|_{\{r_l\in [a,b]\}}=m_ln_ld\theta$). The parameters $a$
and $b$ depend on $l$, but we suppress the dependence from the
notation. 
 
{\bf Step 3. } Next we adjust the relative cohomology class 
$\eta:=[\tilde\om-\om]\in H^2(M,B)$. Since $[\tilde\om-\om]=0\in H^2(M)$, it follows from
Lemma~\ref{algtop} that $\eta$ has a representative of the form
$$
   \sum_{l=1}^nc_ld\bigl(\sigma(r_l)d\phi_l\bigr)
$$
with constants $c_l\in\R$ and a non-increasing function
$\sigma:[a,b]\to[0,1]$ which equals $1$ near $a$ and $0$ near $b$. For
$t\in[0,1]$ set  
$$
   \om_t := \om + t\sum_{c_l\geq 0}c_ld\bigl(\sigma(r_l)d\phi_l\bigr),
   \qquad 
   \tilde\om_t := \tilde\om +
   t\sum_{c_l<0}(-c_l)d\bigl(\sigma(r_l)d\phi_l\bigr). 
$$
Now $\sigma'\leq 0$ implies
$$
   d\theta\wedge\om_t = d\theta\wedge\om - t\sum_{c_l\geq
   0}c_l\sigma'(r_l)dr_l\wedge d\theta\wedge d\phi_l > 0,
$$
which in view of $\lambda|_{\{r_l\in [a,b]\}}=n_ld\theta$ shows that
$(\om_t,\lambda)$ is a homotopy of SHS supported by
$(B,\pi)$. Similarly for $(\tilde\om_t,\tilde\lambda)$. The resulting
HS $\om_1$ and $\tilde\om_1$ coincide in some neighbourhood of 
$\{r\leq a\}$ and
$\lambda=n_ld\theta,\, \tilde\lambda=n_lm_ld\theta$ near $\{r=a\}$, and moreover
$[\tilde\om_1-\om_1]=0\in H^2(M,B)$. 

{\bf Step 4. } Now we restrict $(\om_1,\lambda)$ and
$(\tilde\om_1,\tilde\lambda)$ from Step 3 to the mapping torus 
$$
   X:=M\setminus \{r<a\}\overset{\pi}{\longrightarrow} S^1
$$
with fibre $W:=\Sigma\setminus \{r<a\}$. Let $N$ be a neighbourhood
of $\p X$ on which $(\om_1,\lambda)$ and
$(\tilde\om_1,\tilde\lambda)$ coincide and are special with
$\tilde\lambda=m_l\lambda=m_ln_ld\theta$. Since $[\tilde\om_1-\om_1]=0\in H^2(X,\p X)$, 
Lemma~\ref{lm:mtl} connects $(\om_1,\lambda)$ and
$(\tilde\om_1,\tilde\lambda)$ by a stable homotopy supported by $(B,\pi)$.
The Hamiltonian part of this homotopy is rel $\p X$ and any two stabilizing forms from the homotopy 
differ by a (time dependent) constant in $N$. Thus the homotopy extends over $M$. 
This completes the proof of Theorem~\ref{obd2}.

\section{Structure results in dimension three}\label{sec:structure}

\subsection{A structure theorem}\label{ss:structure}

\begin{theorem}\label{thm:structure}
Let $(\om,\lambda)$ be a SHS on a closed $3$-manifold $M$ and set
$f:=d\lambda/\om$. 
Then there exists a (possibly disconnected and possibly with boundary)
compact $3$-dimensional submanifold $N$ of $M$, invariant under the Reeb flow,
a (possibly empty) disjoint union
$U=U_1\cup\dots\cup U_k$ of compact integrable regions
and a stabilizing 1-form $\tilde\lambda$ for $\om$ with the following
properties: 
\begin{itemize}
\item $\inn U\cup \inn N=M$;
\item the proportionality coefficient 
$\tilde f:=d\tilde\lambda/\om$ is constant on each connected component of $N$; 
\item on each $U_i\cong [0,1]\times T^2$ the SHS $(\om,\lambda)$ is
  $T^2$-invariant and $f(r,z)=\alpha_ir+\beta_i$ for constants
  $\alpha_i>0$, $\beta_i\in\R$;
\item $\tilde\lambda$ is $C^1$-close to $\lambda$. 
\end{itemize}
Moreover, if $\om$ is exact we can arrange that $\tilde f$ attains
only nonzero values on $N$. 
\end{theorem}

\begin{proof}
Let $\tilde f$ and the decomposition $M=\cup_{i=1}^kU_i'\cup N=M$ be
obtained from Corollary~\ref{cor:thick-Z}. Recall that the $U_i'\cong
(a_i',b_i')\times T^2$ are open and for $r$ sufficiently close to $a_i'$
or $b_i'$ we have $\{r\}\times T^2\subset N$. Thus we may replace each
interval by a slightly smaller closed interval
$[a_i,b_i]\subset(a_i',b_i')$ such that the closed integrable regions
$U_i:=[a_i,b_i]\times T^2$ satisfy $\cup_i\inn U_i\cup \inn N=M$. 
    
Next recall from Corollary~\ref{cor:thick-Z} that on each
$U_i\cong[a_i,b_i]\times T^2$ the function $f$ is given by the
projection onto the first factor and $\tilde f=\sigma\circ f$ for a
function $\sigma:\R\to\R$. So the function $\tilde f$ is
constant on the tori $\{r\}\times T^2$, $r\in [a_i,b_i]$. Therefore,
for each integrable region $U_i$ the full version of
Theorem~\ref{thm:integr} applies to give an identification
$U_i\cong[0,1]\times T^2$ (linear on the first factor) in which both
$\om$ and $\tilde\lambda$ are $T^2$-invariant.   

For the last statement, denote by $N^+,N^0,N^-$ the union of
components of $N$ on which $\tilde f$ is positive (resp.~zero,
negative). We want to get rid of $N^0$ if $[\om]=0$.
Pick a primitive $\beta$ of $\om$
and consider the 1-form $\lambda_t:=\tilde\lambda+t\beta$ for small positive
$t$. Then $d\lambda_t/\om=\tilde f+t$, so $\lambda_t$ stabilizes $\om$ and
has proportionality factor is shifted by $t$. We choose $t$ such that
$0<t<-\min_{N^-}\tilde f$. Then $\tilde f+t$ is positive on $N^+\cup N^0$ and
negative on $N^-$, so the stabilizing 1-form $\lambda_t$ is nonzero on $N$.
Clearly, we can choose $\lambda_t$ $C^1$-close to $\tilde\lambda$ and
hence to $\lambda$. Renaming back $\lambda_t$ to $\tilde\lambda$, this
concludes the proof of Theorem~\ref{thm:structure}. 
\end{proof}

Theorem~\ref{thm:structure} is the formulation of the structure
theorem that will be most useful for applications. The following
corollary gives an alternative formulation in which the structure is
more restrictive but we lose $C^1$-closeness.

\begin{corollary}\label{cor:structure}
Every stable Hamiltonian structure on a closed 3-manifold $M$ is 
stably homotopic to a SHS $(\om,\lambda)$ for which 
there exists a (possibly disconnected and possibly with boundary)
compact $3$-dimensional submanifold $N=N^+\cup N^-\cup N^0$ of $M$,
invariant under the Reeb flow, and a (possibly empty) disjoint union
$U=U_1\cup\dots\cup U_k$ of compact integrable regions with the
following properties:  
\begin{itemize}
\item $\inn U\cup \inn N=M$;
\item $d\lambda=\pm\om$ on $N^\pm$ and $d\lambda=0$ on $N^0$;
\item on each $U_i\cong [0,1]\times T^2$ we have
  $(\om,\lambda)=(\om_h,\lambda_g)$ for functions\\ $h,g:[0,1]\to\C$
  satisfying~\eqref{eq:g}. 
\end{itemize}
Moreover, we can always arrange that $N^+$ is nonempty, and if $\om$
is exact we can arrange that $N^0$ is empty. 
\end{corollary}

\begin{proof}
Let a SHS $(\om,\lambda)$ be given. After applying the homotopy in
Proposition~\ref{prop:contact-region}, we may assume that
$d\lambda=\om$ on some open region $V\subset M$, so $1$ is a singular
value of the function $f=d\lambda/\om$. 

After applying
Theorem~\ref{thm:structure} and renaming $\tilde\lambda,\tilde f$ back
to $\lambda,f$ we may assume that there exist invariant compact
$3$-dimensional submanifolds $N$ with the following properties: 
\begin{itemize}
\item $\inn U\cup \inn N=M$;
\item on each connected component $N_i$ of $N$ we have $f\equiv c_i$
  for some constant $c_i\in\R$; 
\item on each $U_i\cong [0,1]\times T^2$ the SHS $(\om,\lambda)$ is
  $T^2$-invariant.
\end{itemize}
Moreover, if $\om$ is exact we can arrange that $c_i\neq 0$ for all
$i$. Since $1$ was a singular value of the original function $f$,
there will be some positive $c_i$. 

By Lemma~\ref{lem:t2inv}, after shrinking the $U_i$ may assume that on
each $U_i\cong[0,1]\times T^2$ we have
$(\om,\lambda)=(\om_h,\lambda_g)$ for functions $h,g:[0,1]\to\C$.  
By Remark~\ref{rem:nonconst}, after a stable homotopy supported in
$\inn U$ we can assume that the slope function $h'/|h'|$ is
nonconstant on each $U_i\setminus N$. 

Now we define a homotopy of stabilizing 1-forms $\lambda_t$,
$t\in[0,1]$, for $\om$ as follows. On each component $N_i$ with
$c_i=0$ we set $\lambda_t:=\lambda$. On each component $N_i$ with
$c_i\neq 0$ we set 
$$
   \lambda_t := (1-t)\lambda + t|c_i|^{-1}\lambda.
$$
We use Proposition~\ref{prop:stabhom} to extend this homotopy to a
homotopy of stabilizing 1-forms over all integrable regions $U_i$. The
homotopy starts at $\lambda_0=\lambda$, and the proportionality
coefficient $f_1=d\lambda_1/\om$ of $\lambda_1$ takes only values
$0,1,-1$ on $N$. Renaming $\lambda_1$ back to $\lambda$, this
consludes the proof of Corollary~\ref{cor:structure}.  
\end{proof}

\subsection{Discreteness of homotopy classes of stable Hamiltonian structures}

\begin{thm}\label{thm:discrete}
Let $(\bar\om,\bar\lambda)$ be a SHS on a closed 3-manifold $M$. Then
for every SHS $(\om,\lambda)$ which is sufficiently $C^2$-close to
$(\bar\om,\bar\lambda)$ and satisfies
$[\om]=[\bar\om]\in H^2(M)$ there exists a stable homotopy
$(\om_t,\lambda_t),\, t\in [0,1]$ such that $\om_0=\bar\om$ and
$\om_1=\om$. Moreover, the homotopy $\om_t$ can be chosen $C^1$-small.
\end{thm}


\begin{cor}\label{cor:countable}
For any closed oriented 3-manifold $M$ and cohomology class $\eta\in
H^2(M;\R)$ there are at most countably many homotopy classes of SHS
representing $\eta$.  
\end{cor}

\begin{proof}
Consider a family of pairwise non-homotopic SHS
$(\om_i,\lambda_i)_{i\in I}$ representing the class $\eta$, for some
index set $I$. Consider the space $\SHS_\eta$ of
SHS representing $\eta$, equipped with the $C^2$-topology. By
Theorem~\ref{thm:discrete}, each $\om_i$ has an open 
neighbourhood $\UU_i$ in $\SHS_\eta$ such that all elements in $\UU_i$
are stably homotopic to $(\om_i,\lambda_i)$. Hence
$\UU_i\cap\UU_j=\emptyset$ for $i\neq j$, so the $(\UU_i)_{i\in I}$
are disjoint open sets in $\SHS_\eta$. Since $\SHS_\eta$ is second
countable (as a subset of the second countable linear space of pairs
of 2-and 1-forms), this is only possible if $I$ is countable.   
\end{proof}

\begin{proof}[Proof of Theorem~\ref{thm:discrete}]
{\bf Step 1. }
Fix a SHS $(\bar\om,\bar\lambda)$ on a closed 3-manifold
$M$. Pick an adapted decomposition $M=\cup_iU_i\cup N$
as in Theorem~\ref{thm:structure}.

Consider a SHS $(\om,\lambda)$ sufficiently $C^2$-close to
$(\bar\om,\bar\lambda)$ with $[\om]=[\bar\om]$. By 
Theorem~\ref{thm1} (a), after pulling back $(\om,\lambda)$ by a
diffeomorphism $C^1$-close to identity, we may assume that
$(\om,\lambda)$ is $T^2$-invariant on $U_i$ and  
$C^1$-close to $(\bar\om,\bar\lambda)$. By Remark~\ref{rem:nonconst}
we can adjust $(\bar\om,\bar\lambda)$ 
so that it has nonconstant slope on each $U_i$. 

By Lemma~\ref{lem:ellipt-small} we can write $\om=\bar\om+d\alpha$ for
a $C^1$-small 1-form $\alpha$. After 
applying Lemma~\ref{lem:t2inv} (b) and shrinking the $U_i$, we may
assume that $\alpha$ is $T^2$-invariant on each $U_i$ (and the slope
of $\bar\om$ is still nonconstant on the new $U_i$).  

Now we replace $\bar\lambda$ by the new
stabilizing 1-form provided by Theorem~\ref{thm:structure}
($C^1$-close to the old one and still
denoted by $\bar\lambda$) such that $\bar f=d\bar\lambda/\bar\om$ is
constant on each connected component of $N$. 

{\bf Step 2. }
Consider first the region $N$. Fix a smooth function
$\tau:[0,1]\to[0,1]$ which equals $0$ near $0$ and $1$ near $1$ and
define $\om_t:=\bar\om+\tau(t)d\alpha$ on $N$.
Denote by $N^c$ the union of components of $N$ on which $\bar f\equiv
c\in\R$ and define 1-forms $\lambda_t$ on $N$ by
$\lambda_t:=\bar\lambda+\tau(t)c\alpha$ on $N^c$. These 1-forms are
$C^1$-close to $\bar\lambda$ and satisfy
$d\lambda_t=c(\bar\om+\tau(t)d\alpha)=c\om$ on $N^c$, 
so $\lambda_t$ stabilizes $\om_t$ on $N$. It remains to extend the homotopy 
$(\om_t,\lambda_t)$ over the integrable regions $U_i$.  

{\bf Step 3. }
Consider an integrable region $U_i\cong[0,1]\times T^2$. After 
applying Lemma~\ref{lem:t2inv} (c) and shrinking the $U_i$, we can write
$\bar\om=d\alpha_{\bar h}$, $\bar\lambda=\lambda_{\bar g}$, $\om=d\alpha_h$,
$\lambda=\lambda_g$, and $\alpha=\alpha_\xi$ for functions $\bar
h,\bar g,\xi,h=\bar h+\xi,g:[0,1]\to\C$, where the pairs $(\bar h,\bar g)$
and $(h,g)$ satisfy~\eqref{eq:g} and $\xi$ is
$C^1$-small. For some small $\eps$ the homotopy $\om_t$ is already defined on 
$[0,\eps]\cup [1-\eps,1]\times T^2$ and we can write it as
$\om_t=d\alpha_{h_t}$ for $h_t=\bar h+\tau(t)\xi$. 
Using Remark~\ref{rem:nonconst} we adjust $\xi$ on $[\eps,1-\eps]$ so
that the slope $(\bar h+\xi)'/|(\bar h+\xi)'|$ is nonconstant. We
extend the homotopy $h_t$ from $[0,\eps]\cup [1-\eps,1]$ to $[0,1]$ so
that the slope $h_t'/|h_t'|$ is nonconstant on $[\eps,1-\eps]$
throughout the homotopy and $h_t$ stays in a small $C^1$-neighbourhood
of $\bar h$. Since on  $[0,\eps]\cup [1-\eps,1]\times T^2$ the
homotopy $\om_t$ is stabilized by $\alpha_t$,
Proposition~\ref{prop:stabhom} allows us to extend this homotopy
thoughout $[0,1]\times T^2$. 
This concludes the proof of Theorem~\ref{thm:discrete}. 
\end{proof}

\subsection{Approximation of stable Hamiltonian structures by
  Morse-Bott ones}\label{subsec:Morse-Bott}

\begin{theorem}\label{thm:Morse-Bott}
For any SHS $(\om,\lambda)$ on a closed oriented 3-manifold $M$ and
any $\eps>0$ there exists a stable homotopy
$\gamma=\{(\om_t=\om+d\mu_t,\lambda_t)\}_{t\in [0,1]}$ such that
$(\om_1,\lambda_1)$ is Morse-Bott and 
$$
   \|\gamma\|_{C^1} :=
   \max_{t\in[0,1]}(\|\dot\mu_t\|_{C^1}+\|\dot\lambda_t\|_{C^1})<\eps.  
$$
\end{theorem}

We will see in Section~\ref{subsec:sft} that this theorem is a crucial
ingredient in the proof that SFT defines a homotopy invariant of SHS
in dimension 3. 

Let us first recall the definition of Morse-Bott from~\cite{BEHWZ}. 
A SHS $(\om,\lambda)$ with Reeb vector field $R$ is called {\em
  Morse-Bott} if the following holds: For all $T>0$
the set $\mathcal{N}_T\subset M$ formed by the $T$-periodic Reeb
orbits is a closed submanifold, the rank of
$d\lambda|_{\mathcal{N}_T}$ is locally constant, and
$T_p\mathcal{N}_T=\ker(D_p\Phi_T-\id)$ for all
$p\in\mathcal{N}_T$, where $\Phi_t$ is the Reeb flow. Note that
the case $\dim\mathcal{N}_T=1$ corresponds to nondegeneracy of
closed Reeb orbits. In this case we call the SHS {\em Morse}. 

Next we prove a version of Theorem~\ref{thm:Morse-Bott} for
$T^2$-invariant SHS. Consider an integrable region $I\times T^2$ with
SHS $(\om_h,\lambda_g)$ as in Section~\ref{subsec:t2inv},
$h=(h_1,h_2)$ and $g=(g_1,g_2)$.  
Its Reeb vector field is given by
$$
   R = \frac{-h_2'\p_\theta+h_1'\p_\phi}{h_1'g_2-h_2'g_1}.
$$
So $T^2$-families of periodic Reeb orbits occur whenever the slope
$h'/|h'|$ is rational.  

\begin{lemma}\label{lem:Morse-Bott}
(a) The SHS $(\om_h,\lambda_g)$ on $I\times T^2$ is Morse-Bott iff
$h_1''h_2'-h_1'h_2''\neq 0$ whenever $h'/|h'|$ is rational.
 
(b) Assume that $(\om_h,\lambda_g)$ is already Morse-Bott in a
neighbourhood of $\p I\times T^2$. Then for each $\eps>0$ there exists
a stable homotopy
$\gamma=\{(\om_{h_t},\lambda_{g_t})\}_{t\in [0,1]}$ such that
$(\om_{h_1},\lambda_{g_1})$ is Morse-Bott and $\|\gamma\|_{C^1}<\eps$.   
\end{lemma}

\begin{proof}
(a) The Reeb flow for time $T$ is given by
$$
   \Phi_T \begin{pmatrix} r \\ \theta \\ \phi \end{pmatrix}
   = \begin{pmatrix} r \\ \theta-k_2(r)T \\ \phi + k_1(r)T
   \end{pmatrix}, \qquad
   k_i := \frac{h_i'}{h_1'g_2-h_2'g_1}
$$
and its linearization equals 
$$
   D\Phi_T = \begin{pmatrix} 1 & 0 & 0 \\ 
      -k_2'T & 1 & 0 \\
      k_1'T & 0 & 1 \end{pmatrix}.
$$
Thus $\ker(D\Phi_T-\id)={\rm span}\{\p_\theta,\p_\phi\}$ iff
$(k_1',k_2')\neq(0,0)$. A short computation shows that the latter
condition is equivalent to $h_1''h_2'-h_1'h_2''\neq 0$. 

(b) Suppose that $(\om_h,\lambda_g)$ is Morse-Bott
near $\p I$, but not on all of $I$ (otherwise there is nothing to
show). Then $h'/|h'|$ is not constant on $I$. By replacing $h$ with
$\tilde h:=h+\xi$ for a $C^\infty$-small perturbation $\xi$ supported
away from the boundary we can arrange that
$\tilde h_1''\tilde h_2'-\tilde h_1'\tilde h_2''\neq 0$ whenever
$\tilde h'/|\tilde h'|$ is rational, and 
$\tilde h'/|\tilde h'|$ is still nonconstant. Then according to
Lemma~\ref{lem:g-pert} we can $C^1$-perturb $g$ accordingly rel
boundary so that~\eqref{eq:g} continues to hold.

It remains to check $C^1$-smallness of the corresponding stable
homotopy $\gamma$. It can be written as
$(\om_t,\lambda_t)=(d\alpha_h+td\alpha_\xi,\lambda_{g_t})$, $t\in [0,1]$,
where $g_t$ is obtained from $h+t\xi$ by Lemma~\ref{lem:g-pert}. 
We can take as  primitives $\mu_t=t\alpha_\xi$, so that
$\dot\mu_t=\alpha_\xi$, which is $C^1$-small if $\xi$ is. The
derivative $\dot\lambda_{g_t}$ is $C^1$-small because $\dot g_t$ is
due to the estimate by $\|\dot g_t\|_{C^1}\leq \|\xi\|_{C^1}$ from
Lemma~\ref{lem:g-pert} (b).  
\end{proof}

\begin{proof}[Proof of Theorem~\ref{thm:Morse-Bott}]
We will construct the homotopy $\gamma$ in 4 steps, each step
providing a $C^1$-small homotopy $\{\gamma_i(t)\}_{t\in[0,1]}$. The
desired homotopy $\{\gamma(t)\}_{t\in[0,1]}$ will then be given by
$\gamma(t)=\gamma_i(4t)$ for $t\in[(i-1)/4,i/4]$. Since
$\|\gamma\|_{C^1}\leq 4\max_i\|\gamma_i\|_{C^1}$, the homotopy
$\gamma$ will be $C^1$-small as well. 
 
{\bf Step 1. }
Given $(\om,\lambda)$, let $\tilde\lambda$ be the stabilizing $1$-form
for $\om$ given by Theorem~\ref{thm:structure}.
We linearly homotope $\lambda$ to $\tilde\lambda$ by
$\lambda_t:(1-t)\lambda+t\tilde\lambda$. Since $\tilde\lambda$ is
$C^1$-close to $\lambda$ are $C^1$ the corresponding stable homotopy
$\gamma=\{\om,\lambda_t\}_{t\in[0,1]}$ is $C^1$-small. We rename
$\tilde\lambda$ back to $\lambda$ and $\tilde f$ back to $f$. 

{\bf Step 2. }
Let $U_i\cong [0,1]\times T^2$ be one of the integrable regions from
Theorem~\ref{thm:structure}.
Recall that in the situation of Theorem~\ref{thm:structure}, for some
$\varepsilon>0$ and some constants $c_-,c_+\in\R$ we have 
$d\lambda=c_-\om$ for $r\in [0,2\eps]$ and $d\lambda=c_+\om$ for
$r\in [1-2\eps,1]$. Since $(\om,\lambda)$ is $T^2$-invariant on
$[0,1]\times T^2$, by Lemma~\ref{lem:t2inv} we can write
$\om=d\alpha_h$ and $\lambda=\lambda_g+g_3(r)dr$ 
for functions $h,g:[0,1]\to\C$ satisfying~\eqref{eq:g} and
$g_3:[0,1]\to\R$. (The term $g_3(r)dr$ does not affect the
stabilization property and will remain unchanged throughout the
following discussion). Moreover, we have
$g=c_\pm h+d_\pm$ on $[0,2\eps]$ resp~$[1-2\eps,1]$, for constants
$d_\pm\in\C$. Consider a $C^\infty$-small perturbation $\xi$ supported
in $(0,2\eps)\cup(1-2\eps,1)$ of $h$ such that the slope
$(h+\xi)'/|(h+\xi)'|$ is constant irrational near $\eps$ and
$1-\eps$. The corresponding stable homotopy
$\gamma=\{\om_{h_t},\lambda_{g_t}+g_3(r)dr\}_{t\in[0,1]}$ with
$h_t=h+t\xi$ and $g_t=c_\om h_t+d_\pm$ on $[0,2\eps]$
resp.~$[1-2\eps,1]$ is clearly $C^1$-small. We rename $(h_1,g_1)$ back
to $(h,g)$.

{\bf Step 3.} 
After Step 2 we may assume that the slope $h'/|h'|$ is constant
irrational near $\eps$ and $1-\eps$. 
According to Lemma~\ref{lem:Morse-Bott}, we can modify $(h,g)$ by a
$C^1$-small homotopy supported in $(\eps,1-\eps)$ to make the SHS
Morse-Bott on a neighbourhood of $[\eps,1-\eps]\times T^2$. After
having performed this homotopy, denote by $Y\subset M$ the union of the
$[\eps,1-\eps]\times T^2$ for all integrable regions. Then
$M\setminus \inn Y=\cup_iN_i$ is a disjoint union of compact regions
$N_i$ on which $d\lambda=c_i\om$  for constants $c_i\in\R$. Moreover,
each boundary component of $N_i$ is a 2-torus belonging to an
integrable region in which the Reeb vector field has irrational
slope.

{\bf Step 4. }
Consider a region $N=N_i$ as in Step 3 on which $d\lambda=c\om$ for
some constant $c\in\R$. Note that the Reeb vector field has no
periodic orbits near $\p N$. According to Theorem B.1
in~\cite{CF}, we find a $C^\infty$-small 1-form $\nu$ compactly
supported in $\inn N$ such that all periodic orbits of $\om+d\nu$ in
$N$ are nondegenerate. (This is proved in~\cite{CF} in the case
without boundary; it carries over to the case with boundary provided
there are no periodic orbits near the boundary). We replace $\lambda$
by the stabilizing form $\lambda+c\nu$ on $N$. The corresponding
stable homotopy $\{\om+td\nu,\lambda+tc\nu\}_{t\in[0,1]}$ is
$C^1$-small. Performing such a perturbation on
every region $N_i$, we obtain the desired Morse-Bott SHS and
Theorem~\ref{thm:Morse-Bott} is proved.  
\end{proof}


\section{Deformations of stable Hamiltonian structures}\label{sec:pert}

Throughout this section, $(\omega,\lambda)$ is a fixed stable
Hamiltonian structure with Reeb vector field $R$. 

\begin{definition}
A {\em (possibly non-exact) deformation of $\om$} is (the germ of) a smooth family
$\{\omega_t\}_{t\in [0,\varepsilon)}$ of closed 2-forms with
$\om_0=\om$. It is called {\em exact} if $[\dot\om_t]=0$. 
A {\em (possibly non-exact) stable deformation of $(\om,\lambda)$} is (the germ of) a smooth
family $\{(\omega_t,\lambda_t)\}_{t\in [0,\varepsilon)}$ of SHS with
$(\om_0,\lambda_0)=(\om,\lambda)$.  
We call a deformation $\{\omega_t\}_{t\in [0,\varepsilon)}$ on $\om$
{\em stabilizable} if for some $\delta\leq\eps$ there exists a smooth
family $\{\lambda_t\}_{t\in [0,\delta)}$ of stabilizing 1-forms for
$\om_t$ with $\lambda_0=\lambda$.  
\end{definition}

Note that, by contrast to the rest of the paper, in this section we
consider exact as well as non-exact deformations. 

Our goal in this section is to describe the set of those deformations
that are stabilizable. In full generality this question seems to be
quite hard and we were able to answer it only in some special
cases. In particular, we provide examples of (exact as well as
non-exact) deformations that are not stabilizable. 

\subsection{Linear stabilizability}\label{subsec:linstab}

Linearizing at $t=0$ gives a necessary linear condition for
stabilizability. To derive it, consider a (possibly non-exact) stable
deformation $\{(\omega_t,\lambda_t)\}_{t\in [0,\varepsilon)}$ of
$(\om_0,\lambda_0)=(\om,\lambda)$ with Reeb vector fields $R_t$.
Differentiating the equations $i_{R_t}\om_t=0=i_{R_t}d\lambda_t$ at
$t=0$ yields 
$$
   i_{\dot R_0}\om+i_R\dot\om_0 = 0 = i_{\dot
   R_0}d\lambda+i_Rd\dot\lambda_0.  
$$
For given $\dot\om_0$, we consider the first equation
\begin{equation}
   i_{\hat R}\om = -i_R\dot\om_0 
\label{hatR}
\end{equation}
as an equation for $\hat R=\dot R_0$. 
Note, first, that Equation \eqref{hatR} is always solvable because its
right-hand side of it evaluates to zero on the kernel of $\omega$
and, second, that the difference between any two solutions of this
equation is a multiple of $R$. Fixing a solution $\hat R$ of Equation
\eqref{hatR}, we consider the second equation 
\begin{equation}  
   i_Rd\hat\lambda=-i_{\hat R}d\lambda
\label{obstr1}
\end{equation}
as an equation for the $1$-form $\hat\lambda$. 
Note that the right hand side of Equation \eqref{obstr1} does not
depend on the choice of $\hat R$ because the difference between any
two such choices is a multiple of $R$ contracting to zero with
$d\lambda$. 

We call a deformation $\{\omega_t\}_{t\in [0,\varepsilon)}$ of $\om$
{\em linearly stabilizable} if equation~\eqref{obstr1} has a solution,
where $\hat R$ is determined by
equation~\eqref{hatR}. Note that this condition depends only on
$\dot\om_0$. We can reformulate it as follows:

\begin{lemma}\label{reform}
A (possibly non-exact) deformation $\{\omega_t\}_{t\in
  [0,\varepsilon)}$ is linearly 
stabilizable if and only if there exists a smooth family
$\{(\lambda_t,R_t)\}_{t\in [0,\eps)}$ of 1-forms and
vector fields satisfying $(\lambda_0,R_0)=(\lambda,R)$,
$i_{R_t}\om_t=0$, and $i_{R_t}d\lambda_t=O(t^2)$.  
\end{lemma}

\begin{proof} 
The ``if'' follows from the preceding discussion. For the ``only if'',
assume that there exists $\hat\lambda$ solving Equation
\eqref{obstr1}. Consider the linear family of $1$-forms
$\lambda_t:=\lambda_0+t\hat\lambda$ for $t\in [0,\varepsilon)$. 
Pick a smooth family of vector fields $R_t$ spanning $\ker\om_t$ with
$R_0=R$ (we may normalize it by $\lambda_t(R_t)=1$ for small $t$). 
As above, differentiating $i_{R_t}\om_t=0$ at $t=0$ shows that $\hat
R=\dot R_0$ solves Equation \eqref{hatR}. In order to check that
$i_{R_t}d\lambda_t=O(t^2)$ we differentiate 
the expression $i_{R_t}d\lambda_t$ at $t=0$. This gives us 
$$
   \ddtt i_{R_t}d\lambda_t=
\ddtt i_{R_t}(d\lambda_0+td\hat\lambda)=i_{\dot
   R_0}d\lambda+i_Rd\hat\lambda
$$
and the last expression is zero because $\hat\lambda$ solves Equation
\eqref{obstr1}. 
\end{proof}

\subsection{Deformations of contact
  structures}\label{subsec:pert-contact} 

In this subsection we assume that the stable Hamiltonian structure is
positive or negative contact, i.e.~$\om=\pm d\lambda$. To warm up, suppose that
the deformation $\{\om_t\}_{[0,\varepsilon)}$ is exact, i.e.~$[\dot
\om_t]=0$. Then we can choose a smooth path of primitives $\sigma_t$
of $(\om_t-\om)$ and $\lambda_t:=\lambda\pm\sigma_t$ is a path of
contact forms stabilizing  $\pm d\lambda_t=\om_t$. So every exact
deformation of a contact structure is stabilizable. 

For non-exact deformations, the linearized problem simplifies because
of the extra information that $\omega=\pm d\lambda$. Indeed, Equation 
\eqref{obstr1} rewrites using equation~\eqref{hatR} as 
$$
   i_Rd\hat\lambda=-i_{\hat R}d\lambda=\mp
   i_{\hat R}\omega=\pm i_R\dot\omega_0,
$$ 
and hence
\begin{equation}
   i_R(d\hat\lambda\mp\dot\omega_0)=0.
\label{solvcont}
\end{equation}
Solvability of this equation for $\hat\lambda$ is equivalent to
existence of a closed $2$-form $\theta$ ($=\dot\om_0\mp d\hat\lambda$)
with 
\begin{equation}
   [\theta]=[\dot\omega_0]\in H^2(M,\R), \qquad i_R\theta=0.
\label{obst}
\end{equation}
This condition has a natural interpretation in terms of foliated
cohomology (see Section~\ref{ss:folcoh}) of the foliation $\LL$
defined by $R$: For a deformation $\{\om_t\}$ of a posive or  negative
contact SHS, condition~\eqref{obst} for linear stabilizability is equivalent
to the condition 
\begin{equation}\label{obst2}
   [\dot\om_0]\in\im[\kappa:H^2_\LL(M)\to H^2(M)]. 
\end{equation}
Thus any deformation with $[\dot\om_0]\notin\im(\kappa)$ will not be
stabilizable. Below we give some examples in which this happens. 

\begin{example}
As in Section~\ref{subsec:left}, let $M$ be a compact quotient of
$PSL(2,\R)$. Let $\alpha_\pm$ be a positive resp.~negative contact form 
coming from a left invariant form on $PSL(2,\R)$, and let $\LL$ be the
foliation generated by its Reeb vector field $R$. By the discussion at
the end of  Section~\ref{subsec:left}, $\kappa:H^2_\LL(M)\to H^2(M)$
is the zero map. 
Hence according to~\eqref{obst2}, a deformation $\om_t$ of the SHS
$(\om,\lambda)=(d\alpha_\pm,\pm\alpha_\pm)$ is linearly stabilizable
if and only if $[\dot\om_0]=0$.   
\end{example} 


As another example, 
let $T^n$ be the $n$-torus with the standard flat Euclidean metric and
$\pi:S^*T^n\cong T^n\times S^{n-1}\to T^n$ its unit cotangent bundle. 
Let 
$$
   i \colon T^n \to T^n \times S^{n-1}=S^* T^n
$$ 
be the inclusion of $T^n\times\{p\}$ for a fixed point $p\in S^{n-1}$.  
Let $R$ be the Reeb vector field of the (contact) Liouville form
$\lambda$ on $S^*T^n$ and $\LL$  the foliation generaged by $R$. 
The following lemma was pointed out to us by J.~Bowden.

\begin{lemma}\label{jonathan}
In the notation above,
$$
   \im[\kappa:H^2_\LL(S^*T^n)\to H^2(S^*T^n)] = \ker[i^*:H^2(S^*T^n)\to
   H^2(T^n)]. 
$$
Therefore, a deformation $\{\om_t\}$ of $(d\lambda,\lambda)$ with
$i^*[\dot\om_0]\neq 0$ is not linearly stabilizable. 
\end{lemma}

\begin{proof}
Assume first that $\theta$ is a closed $2$-form on $S^*T^n$ which
satisfies $\iota_R \theta=0$. We claim that $i^*[\theta]=0$,
i.e.~$\langle [\theta],i_*b \rangle=0$ for each $b \in H_2(T^n)$. 
This immediately follows from the fact that $i_*\big(H_2(T^n)\big)$
is generated by homology classes which are represented by smooth
tori invariant under the Reeb flow. To see the latter fact, write
$T^n$ as $\mathbb{R}^n /\mathbb{Z}^n$ and let $\{e_1,\cdots,e_n\}$
be the standard basis of $\mathbb{R}^n$. It is well known that
the homology of $T^n$ in degree two is generated by the
$2$-dimensional subtori
$$
   T_{jk}=\langle e_j,e_k \rangle/\mathbb{Z}^2 \subset T^n, \quad 
1 \leq j <k \leq n.
$$
Hence we can choose as generators of $i_*\big(H_2(S^*T^n)\big)$ the
$2$-tori 
$$
   \widetilde{T}_{jk}=T_{jk} \times \{e_j\} \subset T^n \times 
S^{n-1}, \quad 1 \leq j <k \leq n.
$$
The Reeb flow $\phi^{\tau}_R$ for $\tau \in \mathbb{R}$ on $S^* T^n$
is given by the geodesic flow. Since geodesics on the torus are
just straight lines we obtain the simple formula
$$
   \phi^{\tau}_R(x,y)=(x+\tau y,y), \quad (x,y) \in T^n \times
   S^{n-1},\,\, \tau \in \mathbb{R}.
$$
Hence the $2$-tori $\widetilde{T}_{jk}$ are invariant under the
geodesic flow. This proves the claim and thus
$\im(\kappa)\subset\ker(i^*)$. 

For $n\geq 3$ we are done since in that case $i^*:H^2(S^*T^n)\to
H^2(T^n)$ is an isomorphism and therefore $\im(\kappa)=\{0\}$. It
remains to show $\ker(i^*)\subset\im(\kappa)$ in the case $n=2$. 
For this, pick coordinates $(x,y,\phi)$ on $S^*T^2$, where $x$ and $y$
are coordinates on $T^2=\R^2/\Z^2$ and $\phi\in S^1$ is the angular
coordinate on the fibre. Then the Reeb vector field is given by
$$
   R = \cos\phi\,\p_x + \sin\phi\,\p_y. 
$$
A short computation shows that a 2-form $\theta$ is closed and
satisfies $i_R\theta=0$ precisely if it is of the form
$$
   \theta = f(\phi)\bigl(\cos\phi\,dy - \sin\phi\,dx\bigr)\wedge d\phi 
$$
for a $2\pi$-periodic function $f:\R\to\R$. Then the values of $[\theta]$
on the subtori $C_x=\{x=0\}$ and $C_y=\{y=0\}$ are
$$
   [\theta](C_x) = \int_0^{2\pi}f(\phi)\cos\phi\,d\phi,\qquad 
   [\theta](C_y) = \int_0^{2\pi}f(\phi)\sin\phi\,d\phi.
$$
Since these values can be arbitrarily prescribed by the choice of $f$,
this shows that every cohomology class in $\ker(i^*)$ can be
represented by such a form $\theta$, and hence
$\ker(i^*)\subset\im(\kappa)$.  
\end{proof}

The case $n=2$ is particularly interesting:

\begin{prop}\label{prop:t3}
In the notation of Lemma~\ref{jonathan}, let $\beta$ be the standard
area form on $T^2\cong \R^2/ \Z^2$ and consider the (non-exact) family of
Hamiltonian structures $\om_t=d\lambda+t\beta$ for $t\in \R$. Then 
\begin{enumerate}
\item each $\om_t$ is stabilizable, but
\item there is no {\it smooth} family $\{\lambda_t\}_{t\in \R}$ of
$1$-forms $\lambda_t$ stabilizing $\omega_t$. (More precisely, 
smoothness always fails at the point $t=0$.)
\end{enumerate}
\end{prop}

Put simply: stabilizing forms exists, but they never form a smooth
family. We emphasize that in (ii) it is {\em not} assumed that
$\lambda_0=\lambda$. 

\begin{proof}
(i) To see stabilizability, note that $\om_0=d\lambda$ is stabilizable
because $\lambda$ is a contact form. For $t\ne 0$ consider the
connection form $\lambda_\st$ for the flat metric on the torus. In
coordinates $(x,y,\phi)$ where $x$ and $y$ are coordinates on 
$T^2$ and $\phi$ is the angular coordinate on the fibre the Liouville
form, the connection form and the standard area form on $T^2$ write
out as 
$$
   \lambda=\cos\phi dx\wedge+\sin\phi dy,\qquad
   \lambda_\st=d\phi,\qquad \beta=dx\wedge dy.
$$ 
It follows that $d\lambda_\st=0$ and $\lambda_\st\wedge
\om_t=t\,dx\wedge dy\wedge d\phi$, so the $1$-form
$sign(t)\lambda_\st$ defines a taut foliation and stabilizes $\om_t$
for $t\ne 0$. 

(ii) To see the second assertion, assume by contradiction there exists
a smooth family $\{\lambda_t\}_{t\in \R}$ stabilizing $\om_t$. We
distinguish two cases according to the proportionality coefficient
$f_0=d\lambda_0/\om_0$ for $t=0$.  

Case 1: $f_0$ is constant. 
Note that in this case $f_0\ne 0$: 
otherwise, the closed stabilizing form $\lambda_0$ would define a
foliation with the exact $2$-form 
$d\lambda$ restricting as an area form to the leaves, which is
impossible. Thus we are in the contact situation for $t=0$, and
therefore the existence of the smooth family $\{\lambda_t\}_{t\in \R}$
stabilizing $\{\om_t\}_{t\in \R}$ contradicts Lemma \ref{jonathan}. 

Case 2: $f_0$ is not constant. 
We will show that this case cannot occur. For this, we need some
dynamical information about the kernel foliation $\ker\om_t$ of
$\om_t$. Since 
$$
   \om_t=-\sin\phi\,d\phi\wedge dx+\cos\phi\,d\phi\wedge
   dy+t\,dx\wedge dy,
$$
the lift of $\ker\om_t$ from $S^1\times T^2$
to $S^1\times \R^2$ is generated by the vector field
$$
   X_t = \cos\,\phi\p_x+\sin\phi\,\p_y - t\,\p_\phi.
$$
It is easy to see that the projection of an orbit of $X_t$ to the
$(x,y)$-plane is a curve, parametrized by arc length, with the
property that its tangent vector is rotating with constant speed $t$.
This means that the projection of any orbit to the $(x,y)$-plane is a
circle. Thus the kernel foliation of $\om_t$ consists of closed
leaves. In particular, $\ker\om_t$ does not possess any
irrational invariant tori. But this contradicts Theorem \ref{thm1}
applied to $(d\lambda,\lambda_0)$. Namely, fix some $\phi_*\in
S^1$. Then  $(S^1\setminus \{\phi_*\})\times T^2\cong (0,2\pi)\times
T^2$ and $\ker\om_0$ descends to each torus $\{\phi\}\times T^2$ as a
linear foliation with slope function $S_0(\phi)=\phi$.  
In particular, $(0,2\pi)\times T^2$ is an integrable region on which
the slope function of $\ker\om_0$ is nowhere constant. The latter
implies that $f_0$ is $T^2$-invariant. Since by assumption $f_0$ is
nonconstant, we can apply Theorem \ref{thm1} and conclude the
existence of irrational invariant tori. This contradiction shows that
Case 2 does not occur and concludes the proof of Proposition~\ref{prop:t3}.  
\end{proof}

We have seen several examples of non-exact deformations which are not
stabilizable. Our next goal is to give an example of an {\em exact}
non-stabilizable deformation. 
For this we need some preparation from functional analysis.

\subsection{Preliminaries from analysis}\label{subsec:anal}

Let $V$ be a smooth real vector bundle over a compact manifold $M$
(possibly with boundary). We consider the space $\Gamma(V)$ 
of all smooth sections of $V$ as a Fr\'echet space and let $F$ be any
closed linear subspace of $\Gamma(V)$. Fix a base point $x_0\in F$ and
set    
$$
   \tilde F:=\{x\in C^{\infty}([0,1],F)\mid x(0)=x_0\}
$$ 
to be the smooth path space of $F$.
Our next goal is give the affine space $\tilde F$
the structure of an affine Fr\'echet space. 
Indeed, let $\pi:[0,1]\times M\rightarrow M$ be the natural projection.
Any path $x\in \tilde F$ can be considered as a section $s_x$ of the
pullback bundle $\pi^*V$  
as follows: for $t\in [0,1]$ and $p\in M$ we set 
$$
   s_x(t,p):=x(t)(p)\in\pi^*V.
$$ 
We identify the path space $\tilde F$ with the corresponding affine
subspace in the space $\Gamma(\pi^*V)$. Now the space of sections
$\Gamma(\pi^*V)$ is Fr\'echet and the subspace $\tilde F$ is closed 
because $F$ is closed in $\Gamma(V)$. Thus $\tilde F$ is itself an
affine Fr\'echet space. The following straightforward lemma will be
crucial for us.
 
\begin{lemma}
For each $T\in (0,1]$ the evaluation map 
$$
   ev_{T}\colon\tilde F\longrightarrow F,\qquad ev_{T}(x):=x(T)
$$
is open and continuous.
\label{ev}
\end{lemma} 

\begin{proof}
The map $ev_T$ is a surjective linear map between affine Fr\'echet
spaces. It 
is continuous because it is obtained by composing on the right with
the map $i_T$ sending $p\in M$ to $(T,p)\in [0,1]\times M$ as follows:
$ev_T(s_x)=s_x\circ i_T$. Finally, $ev_T$ is open by the open mapping
theorem. 
\end{proof}

Note that for any open and continuous map its restriction to any open 
subset of the domain remains open and continuous. We apply this to the
map $ev_T$ as follows: Let $E\subset F$ be any open subset with $x_0\in
E$. The set  
\begin{equation}
\tilde E:=\{x\in C^{\infty}([0,1],E)|x(0)=x_0\} 
\label{paths}
\end{equation}
is an open subset of $\tilde F$, so the restriction of $ev_T$ to
$\tilde E$ is open and continuous. 

Recall that a subset $B$ of a topological space $X$ is called {\em
  Baire set} (or of second category) if it contains a countable
intersection of open and dense sets.
Note that if we have a continuous open map $f\colon X\rightarrow Y$ 
between topological spaces and $B\subset Y$ is a Baire set, then
  $f^{-1}(B)$ is Baire in $X$.  
Indeed, the preimage of an open set is open because the map is
continuous, and the preimage of a dense set is dense because the map
is open, so the preimage of a countable intersection of open and dense
sets is a countable intersection of open and dense sets. We apply this
  to the map $ev_T$ and summarize our discussions as follows. 

\begin{lemma} Let $M$ be a closed manifold, $V\rightarrow M$ a smooth
real vector bundle, $F\subset \Gamma(V)$ a closed linear subspace,
$E\subset F$ an open subset, $x_0\in E$ a base point, $B\subset E$ a
Baire set, and $T\in (0,1]$ a point. 
Then for $\tilde E$ defined by \eqref{paths} and the evaluation map   
$ev_{T}\colon\tilde E\rightarrow E$ at $T$ the preimage $ev_T^{-1}(B)$
is Baire in $\tilde E$. 
\label{genbair}
\end{lemma}

The next step is to see that the space $\tilde E$ is a {\em Baire
space}, i.e.~any Baire subset is dense. This holds because $\tilde
E$ is an open subset of the affine Fr\'echet space $\tilde F$ and open
subsets of complete metrizable spaces have the Baire property. 

Now we fix a Baire set $B\subset E$ and define the following subset of
$\tilde E$:  
$$
   \tilde B:=\{x\in \tilde E\mid\exists\, t_n\to 0:x(t_n)\in B\}.
$$ 
Clearly, 
$$
   \tilde B\supset \bigcap_{n=1}^{\infty}\{x\in \tilde
   E\mid x\left(\frac{1}{n}\right)\in B\} =
   \bigcap_{n=1}^{\infty}ev_{\frac{1}{n}}^{-1}(B).
$$
Each set $ev^{-1}_{\frac{1}{n}}(B)$ is Baire by Lemma \ref{genbair}, 
so their intersection is Baire, thus $\tilde B$ is a Baire set. In
particular, $\tilde B$ is dense in $\tilde E$. 

\subsection{Exact deformations}\label{subsec:experturb}

We now apply the discussion of the previous subsection to Hamiltonian
structures on a closed oriented manifold $M$ as follows (using the
same notation):
\begin{itemize}
\item $V=\Lambda^2(M)$ is the bundle of exterior $2$-forms over $M$;
\item $F\subset\Gamma(V)=\Om^2(M)$ is the space of closed $2$-forms of
  a given cohomology class $\eta$;  
\item $E=\HS_\eta\subset F$ is the open subset of Hamiltonian structures of
  cohomology class $\eta$;  
\item the base point $x_0$ is a fixed Hamiltonian structure $\om_0$.
\end{itemize}

\begin{lemma}
The subset $B\subset \HS_\eta$ of Hamiltonian structures of
cohomology class $\eta$ that are {\em Morse}, i.e.~all periodic
orbits are nondegenerate, is a Baire set. 
\end{lemma}

\begin{proof}
Consider a HS $\om\in\HS_\eta$. As in Section~\ref{subsec:hyper}, we
realize $M$ as the hypersurface $\{0\}\times M$ in its symplectization 
\begin{equation*}
   \Bigl(X:=[-\varepsilon,\varepsilon]\times M,\ 
   \Om:=\om+td\lambda+dt\wedge \lambda\Bigr)
\end{equation*}
Nearby hypersurfaces $M'\subset X$ are graphs over $\{0\}\times M$, so
pulling back $\Om|_{M'}$ under the projection $M\to M'$ along $\R$
yields HS $\om'$ on $M$ cohomologous to $\om$. 
By Theorem B.1 in~\cite{CF}, there exist hypersurfaces $M'$
$C^\infty$-close to $\{0\}\times M$ for which $\Om|_{M'}$ is
Morse. This shows that the set $B$ is dense.

Let us fix a Riemannian metric on $M$ and parametrize all orbits of
$\ker(\om)$, for $\om\in\HS_\eta$, with unit speed. For $N\in\N$
denote by $B_N\subset\HS_\eta$ the set of HS of class $\eta$ for which 
all periodic orbits of period $\leq N$ are nondegenerate. A standard
argument shows that $B_N$ is open. (Consider a sequence $\om_k\notin
B_N$ converging to $\om\in\HS_\eta$; let $\gamma_k$ be degenerate
periodic orbits of $\om_k$ of period $\leq N$; by the Arzela-Ascoli
theorem a subsequence of $\gamma_k$ converges to a degenerate periodic
orbit of $\om$ of period $\leq N$; hence $\om\notin B_N$). Since $B\subset B_N$
and $B$ is dense, we get that $B_N$ is dense. Hence $B=\cap_{N\in\N}B_N$
is a Baire set. 
\end{proof}

We continue in the notation of the previous subsection. Thus
\begin{itemize}
\item $\tilde E$ is the space of smooth paths $\{\om_t\}_{t\in[0,1]}$
  in $\HS_\eta$ starting at $\om_0$;
\item $\tilde B\subset\tilde E$ is the subset of those paths for which
  there exists a sequence $t_n\to 0$ such that $\om_{t_n}$ is Morse. 
\end{itemize}
We will refer to $\tilde E$ as the {\em space of exact deformations of
  $\om_0$}. By the discussion in the previous subsection, $\tilde B$
is a Baire set. This is what we mean when we say ``a Hamiltonian
structure can be deformed slightly to make it Morse''. 

If the HS $\om_0$ is contact, i.e.~$\om_0=d\lambda_0$ for a contact
form $\lambda_0$, the above argument justifies the folk lemma that we
can always deform a contact form slightly to make it Morse. This is in
sharp contrast to the situation with general SHS: Theorem~\ref{thm3}
provides (``quite a lot of'') SHS which cannot even be
$C^2$-approximated by Morse SHS because any $C^2$-close SHS
must have rational invariant tori by Theorem~\ref{thm1}. In
particular, for a SHS $(\om_0,\lambda_0)$ as in Theorem~\ref{thm1}
the deformations in the Baire set $\tilde B$ cannot be
stabilized. This discussion shows

\begin{thm}\label{thm4}
Consider a $1$-dimensional foliation $\LL$ as in
Theorem~\ref{thm2}. Then for any HS $\om_0$ defining $\LL$ there exists
a Baire set $\tilde B$ in the space $\tilde E$ of exact deformations
of $\om_0$ that cannot be stabilized, no matter what stabilizing
$1$-form $\lambda_0$ we take for $\om_0$. 
\end{thm}

\section{Homotopies and cobordisms of stable Hamiltonian
  structures}\label{sec:hom}

A smooth homotopy of contact forms $(\lambda_t)_{t\in[0,1]}$ on a
closed manifold $M$ gives rise to a symplectic cobordism
$\bigr([0,1]\times M,d(e^{ct}\lambda_t)\bigr)$ from $(M,d\lambda)$ to
$(M,e^cd\lambda)$ for a sufficiently large constant $c>0$. The
corresponding statement for stable Hamiltonian structures fails (see
e.g.~Corollary~\ref{cor:exotic} below). In this section we investigate
the relation between stable homotopies and various generalized notions
of symplectic cobordism. We discuss obstructions to symplectic
cobordisms, as well as obstructions to stable homotopies arising from
Rabinowitz Floer homology and symplectic field theory. 

Recall that for all homotopies of HS $\om_t$ we assume that the
cohomology class of $\om_t$ is constant. 

\subsection{Various notions of cobordism}\label{ss:cob-def}

For a SHS $(\om,\lambda)$ with $\ker(\om)=\LL$ we introduce the
following notation: 
\begin{align*}
   \delta^+_{(\om,\lambda)} &:= \sup\{\tau>0\mid
   \ker(\om+t\,d\lambda)=\LL \text{ for all }t\in[0,\tau)\}, \cr
   \delta^-_{(\om,\lambda)} &:= \sup\{\tau>0\mid
   \ker(\om+t\,d\lambda)=\LL \text{ for all }t\in(-\tau,0]\}, \cr
   \delta_{(\om,\lambda)} &:=
   \min\{\delta^+_{(\om,\lambda)},\delta^-_{(\om,\lambda)}\}, \cr 
   I_{(\om,\lambda)} &:=
   (-\delta^-_{(\om,\lambda)},\delta^+_{(\om,\lambda)}), \cr
   C_{(\om,\lambda)} &:= \{\om+t\,d\lambda\mid t\in
   I_{(\om,\lambda)}\}. 
\end{align*}
Moreover, for a HS $\om$ we define
$$
   D_\om^+ := \{\hat\om\in \Omega^2(M)\mid
   d\hat\om=0,\;[\hat\om]=[\om]\in
   H^2(M),\;\ker\hat\om=\ker\om,\;\hat\lambda\wedge\hat\om^{n-1}>0\},
$$
where $\hat\lambda$ is any 1-form with $\hat\lambda|_{\ker(\om)}>0$. Note that
the definition of $D_{\om}^+$ only depends on the oriented kernel foliation
$\ker(\om)$ and the cohomology class of $\om$. Also note that
$C_{(\om,\lambda)}\subset D_{\om}^+$ for a SHS $(\om,\lambda)$. 

Now we turn to symplectic cobordisms. We will only consider cobordisms
that are {\em topologically trivial}, i.e.~of the form $[a,b]\times
M$. 

\begin{definition}\label{def:cob1}
A {\em (topologically trivial) symplectic cobordism} between HS
$\om_a$ and $\om_b$ on $M$ is a symplectic manifold $([a,b]\times M,\Om)$
such that $\Om|_{\{i\}\times M}=\om_i$ for $i=a,b$.  
A symplectic cobordism $([a,b]\times M, \Om)$ is called {\em trivial} if
$$
   \Om=\om+d\bigl(f(t)\lambda\bigr)
$$ 
for some SHS $(\om,\lambda)$ on $M$ and an increasing function
$f:[a,b]\to I_{(\om,\lambda)}$.   
A {\em homotopy of symplectic cobordisms} between $\om_a$ and $\om_b$
is a smooth family $\{\Om^s\}_{s\in [0,1]}$ of symplectic forms on
$[a,b]\times M$ with $\Om^s|_{\{i\}\times M}=\om_i$ for $i=a,b$ and
all $s$. 
\end{definition}

Note that if $([a,b]\times M, \Om)$ is a symplectic cobordism, then 
$\om_t:=\Om|_{\{t\}\times M}$ defines a homotopy of HS from $\om_a$ to
$\om_b$. Note that the $\om_t$, so there exists a smooth path of
1-forms $\mu_t$ with $\mu_0=0$ and $\om_t=\om_a+d\mu_t$. In the
following {\em all homotopies will come with a chosen smooth path of
  primitives $\mu_t$}.  
Conversely, a homotopy of HS $\om_t=\om_a+d\mu_t$ induces
a closed 2-form 
$$
   \Om := \om_a + d\mu_t + dt\wedge\dot\mu_t
$$
on $[a,b]\times M$ which is symplectic iff
\begin{equation}\label{eq:pos}
   \dot\mu_t|_{\ker(\om_t)} > 0. 
\end{equation}
The above definition of symplectic cobordism is too rigid. For
example, a HS is not cobordant to itself with this
definition. Therefore, we now introduce two more flexible notions. 

\begin{definition}
A {\em strong symplectic cobordism} between SHS
$(\om_a,\lambda_a)$ and $(\om_b,\lambda_b)$ on $M$ is a symplectic
manifold $([a,b]\times M,\Om)$ such that $\Om|_{\{i\}\times M}\in
C_{(\om_i,\lambda_i)}$ for $i=a,b$.  
A {\em weak symplectic cobordism} between HS
$\om_a$ and $\om_b$ on $M$ is a symplectic manifold $([a,b]\times M,\Om)$
such that $\Om|_{\{i\}\times M}\in D^+_{\om_i}$ for $i=a,b$. 
Two SHS $(\om_a,\lambda_a)$ and $(\om_b,\lambda_b)$ are called {\em
  (strongly resp.~weakly) bicobordant} if they are (strongly
resp.~weakly)  cobordant both ways. Homotopies of strong resp.~weak
cobordisms are defined analogously to Definition~\ref{def:cob1}. 
\end{definition}

Note that a strong cobordism is in particular a weak cobordism, but
not the other way around. Also note that a trivial cobordism is in
particular a strong cobordism from $(\om,\lambda)$ to itself and such
trivial cobordisms exist for all SHS. Finally we point out that, while
cobordisms in the sense of Definition~\ref{def:cob1} can obviously be
composed, this is not true for strong or weak cobordisms. We will
discuss this in Section~\ref{subsec:largeh}.

\subsection{Short homotopies and strong
  cobordisms}\label{subsec:shorth}

Consider now a homotopy of {\em stable} HS
$(\om_t=\om_a+d\mu_t,\lambda_t)$ with Reeb vector fields $R_t$. We
try to build a symplectic form on $[a,b]\times M$ by 
$$
   \Om := \om_a + d\nu_t + dt\wedge\dot\nu_t,\qquad \nu_t:= \mu_t +
   f(t)\lambda_t 
$$
for some smooth function $f:[a,b]\to\R$. The form $\Om$ is symplectic
iff $\om_t+f(t)d\lambda_t$ is maximally nondegenerate for all $t$ and $\nu_t$
satisfies~\eqref{eq:pos}. The first condition holds if $|f(t)| <
\delta_{(\om_t,\lambda_t)}$ (as defined in Section~\ref{ss:cob-def}). 
Condition~\eqref{eq:pos} for $\nu_t$ is equivalent to 
$$
   (\dot\mu_t + \dot f(t)\lambda_t + f(t)\dot\lambda_t)(R_t)>0. 
$$
We associate to the homotopy
$\gamma:=\{\om_t=\om_a+d\mu_t,\lambda_t\}$ the following quantities:
$$
   \delta:=\delta(\gamma) := \min_t\delta_{(\om_t,\lambda_t)},\quad 
   A:=A(\gamma) := \max_t|\dot\lambda_t(R_t)|,\quad 
   B:=B(\gamma) := \max_t|\dot\mu_t(R_t)|.
$$
Then symplecticity of $\Om$ is implied by
$$
   |f(t)|<\delta,\qquad \dot f(t) > A |f(t)| + B. 
$$
These conditions are solved by the linear function 
$$
   f(t) = (A\delta+B+(b-a)^{-1}\delta)(t-a/2-b/2)
$$ 
provided that
$$
   (A\delta+B)(b-a)<\delta. 
$$
Define the {\em length} of the homotopy $\gamma$ as 
$$
   L(\gamma) := (A+B/\delta)(b-a). 
$$
Then we have shown

\begin{lemma}
To every homotopy $\gamma:=\{(\om_t,\lambda_t)\}_{t\in [a,b]}$ of
length $L(\gamma)<1$ we can canonically associate a symplectic
cobordism  
$$
   C(\gamma):=([a,b]\times M,\Om)
$$ 
such that for all $t\in [a,b]$ we have $\Om\bigl|_{\{t\}\times M}\in
C_{(\om_t,\lambda_t)}$. If $\gamma$ is a constant homotopy, then  
$C(\gamma)$ is a trivial cobordism and if $\{\gamma^s\}_{s\in[0,1]}$ is a 
path of homotopies, then $C(\gamma^s)$ is a homotopy of cobordisms.
\label{basicshort}
\end{lemma}

\begin{example}\label{ex:length}
If $[a,b]=[0,1]$ and $\om_t=\om$ is independent of $t$ we have $B=0$ and
$L(\gamma)=A$. Moreover, in this case the stabilizing 1-forms
$\lambda_0$ and $\lambda_1$ can be connected by a linear homotopy
$\lambda_t=(1-t)\lambda_0+t\lambda_1$. A short computation yields
$$
   R_t = \frac{R_0}{1-t+t\lambda_1(R_0)},\qquad
   \dot\lambda_t(R_t)=\frac{\lambda_1(R_0)-1}{1-t+t\lambda_1(R_0)}.
$$
So $L(\gamma)$ is small whenever $\lambda_1$ is $C^0$-close to
$\lambda_0$. 
\end{example}

Let us collect the dependence of the length on certain natural
operations on homotopies. 

(Reparametrization). 
Given a homotopy $\gamma=\{\gamma_t\}_{t\in[a,b]}$ and a smooth
function $\tau:[a',b']\to [a,b]$ consider the reparametrized homotopy
$\gamma\circ\tau=\{\gamma_{\tau(t)}\}_{t\in[a',b']}$. With
$K:=\max_t|\dot\tau(t)|$ the constants change as follows:
\begin{gather*}
   A(\gamma\circ\tau)\leq K\,A(\gamma),\qquad 
   B(\gamma\circ\tau)\leq K\,B(\gamma),\cr 
   \delta(\gamma\circ\tau) = \delta(\gamma),\qquad 
   L(\gamma\circ\tau)\leq \frac{K(b'-a')}{(b-a)}L(\gamma). 
\end{gather*}
In particular, reparametrization by a nonconstant linear function
$\tau$ does not change the length. Moreover, we can reparametrize a
homotopy to make it constant near $a$ and $b$, or to turn a piecewise
smooth homotopy into a smooth one, with arbitrarily small change in
length. 

(Rescaling). 
Replacing the stabilizing 1-forms $\lambda_t$ by $c\lambda_t$ for a
constant $t$ has the following effect on the constants:
$$
   R_t\mapsto R_t/c,\qquad A\mapsto A,\qquad B\mapsto B/c,\qquad
   \delta\mapsto \delta/c,\qquad L\mapsto L. 
$$
Replacing the 2-forms $\om_t$ by $c\om_t$ for a
constant $t$ has the following effect on the constants:
$$
   \mu_t\mapsto c\mu_t,\qquad A\mapsto A,\qquad B\mapsto cB,\qquad
   \delta\mapsto c\delta,\qquad L\mapsto L. 
$$
In particular, the length is invariant under both rescalings. 

(Concatenation). 
Two homotopies $\gamma=\{\gamma_t\}_{t\in[a,b]}$ and
$\gamma'=\{\gamma_t'\}_{t\in[b,c]}$ with $\gamma_b=\gamma_b'$ can be
concatenated in the obvious way to a homotopy $\gamma\#\gamma'$ on the
interval $[a,c]$. Then we have
\begin{gather*}
   A(\gamma\#\gamma') = \max\{A(\gamma),A(\gamma')\},\qquad
   B(\gamma\#\gamma') = \max\{B(\gamma),B(\gamma')\},\cr
   \delta(\gamma\#\gamma') =
   \min\{\delta(\gamma),\delta(\gamma')\},\qquad 
   L(\gamma\#\gamma') \geq L(\gamma)+L(\gamma'). 
\end{gather*}
(Restriction). Restricting a homotopy $\gamma$ to a subinterval
$[a',b']\subset[a,b]$ does not increase $A$ and $B$ and does not
decrease $\delta$, so
$$
    L(\gamma|_{[a',b']}) \leq \frac{b'-a'}{b-a}L(\gamma).  
$$
In particular, any homotopy $\gamma$ can be decomposed into homotopies
of arbitrarily small length. 

(Reversal).
For a homotopy $\gamma=\{\gamma_t\}_{t\in [a,b]}$ the reversed
homotopy $\gamma^{-1}=\{\gamma_{2b-t}\}_{t\in [b,2b-a]}$ retains the
same constants and in particular the same length. The homotopy
$\gamma\#\gamma^{-1}$, which runs from $\gamma_1$ to $\gamma_b$ and
back to $\gamma_a$ over the interval $[a,2b-a]$, has length 
$$
   L(\gamma\#\gamma^{-1}) = 2 L(\gamma). 
$$

We see that that for concatenation there is no triangle type
inequality in general. Nevertheless, the following lemma provides an
upper bound on $L(\gamma\#\gamma')$ in terms of $L(\gamma)$ and
$L(\gamma')$ for certain concatenations.
For a stable homotopy \\ $\gamma=\{\om_t=\om_0+d\mu_t,\lambda_t\}_{t\in[0,1]}$ 
we define the quantity
$$
   \|\gamma\|_{C^1} :=
   \max_{t\in[0,1]}(\|\dot\mu_t\|_{C^1}+\|\dot\lambda_t\|_{C^1}). 
$$
We say that $\gamma$ is {\em $C^1$-small} if $\|\gamma\|_{C^1}$ is
small. 

\begin{lemma}\label{lem:concat}
For every SHS $(\om,\lambda)$ and every $\eps>0$ there exists $\rho>0$
with the following property. Whenever
$\gamma=\{\om_t,\lambda_t\}_{t\in[-1,0]}$ is a stable homotopy ending
at $(\om,\lambda)$ and $\gamma'=\{\om_t,\lambda_t\}_{t\in[0,1]}$ is a
stable homotopy starting at $(\om,\lambda)$ such that
$\|\gamma'\|_{C^1}<\rho$, then $L(\gamma\#\gamma')\leq
L(\gamma)+\eps$.  
\end{lemma}

\begin{proof}
It follows imediately from the definition that the assignment
$(\om,\lambda)\mapsto \delta_{(\om,\lambda)}$ is lower semi-continuous
with respect to the $C^0\times C^1$-topology, i.e.~for each SHS
$(\om,\lambda)$ and $\eps>0$ there exists a $C^0\times
C^1$-neighbourhood $V$ of $(\om,\lambda)$ such that for all 
$(\tilde\om,\tilde\lambda)\in V$ we have 
$$
   \delta_{(\tilde\om,\tilde\lambda)}\ge
   \delta_{(\om,\lambda)}-\eps/2.  
$$
Note that for each $t\in[0,1]$ we have
$\|\om_t-\om\|_{C^0}+\|\lambda_t-\lambda\|_{C^1}\leq\|\gamma'\|_{C^1}$, 
hence for $\rho>\|\gamma'\|_{C^1}$ sufficiently small we have
$(\om_t,\lambda_t)\in V$ and thus 
$$
   \delta(\gamma')\geq \delta_{(\om,\lambda)}-\eps/2 \geq
   \delta(\gamma)-\eps/2. 
$$
Moreover, by the (Concatenation) properties we have
$$
   A(\gamma\#\gamma')\leq A(\gamma)+A(\gamma')\leq A(\gamma)+\rho,
   \qquad  
   B(\gamma\#\gamma')\leq B(\gamma)+B(\gamma')\leq B(\gamma)+\rho. 
$$
These estimates combine to $L(\gamma\#\gamma')\leq L(\gamma)+\eps$ for
$\rho$ sufficiently small. 
\end{proof}

The main result of this subsection is 

\begin{proposition} 
Let $\gamma=\{\om_t,\lambda_t\}_{t\in [0,1]}$ be a homotopy of SHS's of
length $L(\gamma)<1/3$. Then there exists a symplectic cobordism
$([0,3]\times M,\Om)$ with the following properties:
\begin{itemize}
\item $\Om|_{\{t\}\times M}\in C_{(\om_{\tau(t)},\lambda_{\tau(t)})}$,
    where $\tau:[0,3]\to[0,1]$ is the function 
$$
   \tau(t) = \begin{cases}
      t &: t\in [0,1],\cr
      2-t &: t\in [1,2],\cr
      t-2 &: t\in [2,3].
   \end{cases}
$$
\item $([0,2]\times M,\Om)$ and $([1,3]\times M,\Om)$ are homotopic
to trivial cobordisms. 
\end{itemize}
\label{maincob}
\end{proposition}

More colloquially, this means that short stable homotopies give rise to
strong bicobordisms whose compositions are defined and homotopic to
trivial cobordisms. 

\begin{proof}
Since $|\tau'|\leq 1$, the reparametrized homotopy $\gamma\circ\tau$
has length $L(\gamma\circ\tau)\leq 3L(\gamma) < 1$. Hence by
Lemma~\ref{basicshort} there exists a symplectic cobordism
$C(\gamma\circ\tau) = ([0,3]\times M,\Om)$ with $\Om|_{\{t\}\times
  M}\in C_{(\om_{\tau(t)},\lambda_{\tau(t)})}$. By the same lemma, 
the path of homotopies $\{\gamma\circ\tau^s\}_{s\in[0,2]}$ with fixed
endpoints, where $\tau^s(t)=s\tau(t)$, induces a homotopy
of cobordisms $C(\gamma\circ\tau^s)$ from $([0,2]\times M,\Om)$ to a
trivial cobordism. An analogous argument applies to $([1,3]\times
M,\Om)$. 
\end{proof}

\subsection{Large homotopies and weak cobordisms}\label{subsec:largeh} 

If we drop the shortness condition on the homotopy both
Lemma~\ref{basicshort} and Proposition~\ref{maincob} fail. 

\begin{prop}\label{prop:hom-weakcob}
On any closed oriented 3-manifold $M$ there exist a homotopy
$(\om_t,\lambda_t)_{t\in [0,1]}$ of SHS such that there exists no
strong symplectic cobordism from $(\om_0,\lambda_0)$ to $(\om_1,\lambda_1)$.
\end{prop}

We will discuss in this subsections two obstructions to symplectic
cobordisms. The proof of Proposition~\ref{prop:hom-weakcob} is based
on the 

{\bf First obstruction: helicity. }
Suppose that $\dim M=3$ and recall from Section~\ref{ss:hel} the
helicity $\Hel(\om)=\int_M\alpha\wedge\om$ of an exact 2-form
$\om=d\alpha$. Consider a symplectic cobordism $([0,1]\times M,\Om)$
with $\Om|_{\{i\}\times M}=\om_i=d\alpha_i$ for $i=0,1$. 
This $\Om$ is exact, i.e.~$\Om=d\gamma$ for some $1$-form $\gamma$ on
$[0,1]\times M$, and we obtain
\begin{align*}
   0 &< \int_{[0,1]\times M}\Om\wedge\Om = \int_{[0,1]\times
     M}d(\gamma\wedge\Om) \cr
   &= \int_{\{1\}\times M}\gamma\wedge d\alpha_1 - \int_{\{0\}\times
     M}\gamma\wedge d\alpha_0 \cr
   &= \int_M\alpha_1\wedge \om_1 - \int_M\alpha_0\wedge \om_0 \cr
   &= \Hel(\om_1) - \Hel(\om_0). 
\end{align*}
Hence helicity is monotone under symplectic cobordisms. To obtain from
this an obstruction to strong cobordisms we need to control the
helicity $\Hel(\hat\om)$ for $\hat\om\in C_{(\om,\lambda)}$. In the
presence of  integrable regions, this can be done by choosing a
sufficiently "exotic" stabilizing 1-form:

\begin{lemma}\label{lem:hom-weakcom}
Let $(\om=d\alpha,\lambda_0)$ be an exact SHS on a closed 3-manifold
$M$. Suppose that there exists an integrable region $I\times T^2$ in
$M$ on which $(d\alpha,\lambda_0)$ has a constant slope $v\in S^1$ in
the sense of Section~\ref{subsec:t2inv}. 
Then for every $\eps>0$ there exists a stabilizing 1-form $\lambda$
for $\om$ such that  $|\Hel(\hat\om)-\Hel(\om)|<\eps$ for all
$\hat\om\in C_{(\om,\lambda)}$. 
\end{lemma}

\begin{proof}
Consider $\hat\om\in C_{(\om,\lambda)}$, i.e.~$\hat\om=\om+t\,d\lambda$
with $t\in I_{(\om,\lambda)}$. Its helicity is
\begin{align*}
   \Hel(\hat\om) 
   &= \int_M(\alpha+t\lambda)\wedge(d\alpha+t\,d\lambda) \cr
   &= \Hel(\om) + 2t\int_M\lambda\wedge\om + t^2\int_M\lambda\wedge
   d\lambda. 
\end{align*}
Now suppose that $(\alpha,\lambda_0)=(\alpha_h,\lambda_{g_0})$ has
a constant unit slope $v\in\C$ on $I\times T^2$,
i.e.~$ih'/|h'|\equiv v$ and $\la g_0,v\ra\equiv c>0$. 
Given $\eps>0$,
we deform $g$ rel $\p I$, keeping the condition $\la g,v\ra\equiv c$,
such that $g'(r_\pm)=\pm h'(r_\pm)/\eps$ for some $r_\pm\in I$. Now
$\om+t\,d\lambda\neq 0$ iff $h'(r)+tg'(r)\neq 0$ for all $r\in I$,
which implies $|t|\leq \eps$. The term $\lambda_g\wedge\om_h=\la
g,ih'\ra dr\wedge d\theta\wedge d\phi = c\,\vol$ is independent of
$\eps$, and the term $\lambda_g\wedge d\lambda_g = \la g,ig'\ra\vol$
is of order $1/\eps$. Hence the above formular for the helicity yields 
$\Hel(\hat\om) = \Hel(\om) + O(\eps)$ and the lemma is proved.  
\end{proof}

\begin{proof}[Proof of  Proposition~\ref{prop:hom-weakcob}]
By Proposition~\ref{prop:contact-region}, there exists a SHS
$(\om_0,\lambda_0)$ on $M$ ($(\om_1,\lambda_1)$ in the notation of Proposition~\ref{prop:contact-region}) restricting as $(d\alpha_{st}, \alpha_{st})$
to some embedded solid torus $S^1\times D^2$ with $\alpha_{st}=\alpha_{h_0}$ for $h_0=(r^2,1-r^2)$.    Set $\gamma:=S^1\times \{(0,0)\}$ and  write $S^1\times D^2\setminus \gamma=I\times T^2$. Let $h_1:I\to \C$ be an immersion which is homotopic to $h_0$ rel $\p I$
and such that the signed area $\Delta A$ enclosed between the curves $h_0$ and $h_1$ is positive.
We use Corollary \ref{cor:stablin} to find a connecting homotopy $(h_t,g_t)$, $t\in [0,1]$ satisfying \eqref{eq:g}. Thus the SHS
$(\om_t,\lambda_t)$ defined by $(h_t,g_t)$ satisfy 
$$
   \Hel(\om_1)-\Hel(\om_0) = \int_Ih_1^*(y\,dx-x\,dy) -
   \int_Ih_0^*(y\,dx-x\,dy) = -2\Delta A < 0,
$$
where $x+iy$ denote coordinates on $\C$. Now pick $0<\eps<\Delta A$
and replace $\lambda_j$, $j=0,1$, by the stabilizing forms (still
denoted $\lambda_j$) provided by
Lemma~\ref{lem:hom-weakcom}. Then for all $\hat\om_i\in
C_{(\om_i,\lambda_i)}$ we have
$$
   \Hel(\hat\om_1)-\Hel(\hat\om_0) \leq -2\Delta A+2\eps < 0, 
$$
so by monotonicity of  helicity there exists no
strong symplectic cobordism from $(\om_0,\lambda_0)$ to $(\om_1,\lambda_1)$.
\end{proof}

Proposition~\ref{prop:hom-weakcob} shows that the notion of strong
cobordism is too rigid for large homotopies. This motivates the following

\begin{question}[Large Cobordism Question]
Given a stable homotopy $(\om_t,\lambda_t)_{t\in[0,1]}$, are $\om_0$ and
$\om_1$ weakly bicobordant?
\label{LCQ}
\end{question} 

A positive answer to this question would be very useful for the
following reason: Obstructions to homotopies of SHS beyond the
homotopy class of almost contact structures are hard to construct (the
only known ones arise from rational symplectic field theory, and in a
more restricted setting from Rabinowitz Floer homology, see later
subsections). On the other hand, there are simple obstructions to weak
symplectic cobordisms, two of which we will now describe.  

The first obstruction is again based on monotonicity of helicity on
cobordisms. To apply this to weak cobordisms, 
suppose that we have two exact HS's $\om_0$ and $\om_1$ such that 
$$
   \Hel(\hat\om_0) > 0,\qquad \Hel(\hat\om_1) \leq 0
$$ 
for all $\hat\om_i\in D_{\om_i}^+$. Then $\om_0$ and $\om_1$ are not
weakly cobordant. This situation occurs in the example discussed in
Section~\ref{subsec:left}: Let
$M$ be a compact quotient of $PSL(2,\mathbb{R})$ and let $\alpha_+$ a
positive contact form coming from a left-invariant one on
$PSL(2,\mathbb{R})$. Since any $\hat\omega\in D_{d\alpha_+}^{+}$ is a
closed form contracting to zero with the Reeb vector field of
$\alpha_+$, equation~\eqref{psl2pos} holds with $\hat\om$ in place of
$\theta$, so $\hat\om = c\,d\alpha_+$ for some constant $c\neq 0$. Thus 
$$
   \Hel(\hat\omega) = \Hel(c\,d\alpha_+) = c^2\Hel(d\alpha_+) =
   c^2\int_M\alpha_+\wedge d\alpha_+>0.
$$
Let $\alpha_-$ be a negative contact form coming from a left-invariant
one on $PSL(2,\mathbb{R})$. Since $[d\alpha_-]$ generates $\ker\kappa
= H^2_\LL(M)$, any $\hat\om\in D_{d\alpha_-}^+$ has foliated
cohomology class $[\hat\om] = c[d\alpha_-]\in\ker\kappa$ for a
constant $c\in\R$, and therefore
$$
   \Hel(\hat\omega) = \Hel(c\,d\alpha_-) = c^2\Hel(d\alpha_-) \leq 0
$$ 
because $\alpha_-$ is a negative contact form. 
On the other  hand, $d\alpha_+$
and $d\alpha_-$ are homotopic as HS because any two nonzero
left-invariant 2-forms on $PLS(2,\R)$ are homotopic through nonzero
left-invariant 2-forms and this homotopy descends to any compact quotient.
Therefore, we have shown 

\begin{prop}
The stable Hamiltonian structures $(d\alpha_+,\alpha_+)$ and
$(d\alpha_-,-\alpha_-)$ above on $\Gamma\setminus PSL(2,\R)$ are
homotopic as HS, but there exists no weak symplectic cobordism from  
$d\alpha_+$ to $d\alpha_-$. 
\end{prop}

We conjecture that the two SHS $(d\alpha_+,\alpha_+)$ and
$(d\alpha_-,-\alpha_-)$ are not stably homotopic. 

The second obstruction to weak symplectic cobordisms is fillability:
We will show that the standard tight contact structure on $S^3$ is not
weakly cobordant to any overtwisted one. 
Before showing this, we first discuss 

{\bf Composability of weak cobordisms. }
Consider two symplectic cobordisms $([a,b]\times M,\Om_-)$
and $([b,c]\times M,\Om_+)$ with $\Om_\pm|_{\{b\}\times M}\in
D^+_\om$. Assume that there exists a constant $C>0$ and an
orientation preserving diffeomorphism $\Phi:M\to M$ such that 
$$
   \Phi^*(\Om_+)|_{\{b\}\times M} = C\Om_-|_{\{b\}\times M}.    
$$
Then 
$$
   W := ([a,b]\times M)\cup_\Phi([b,c]\times M)
$$
with the symplectic form given by $\Om_+$ on $[b,c]\times M$ and by
$C\Om_-$ on $[a,b]\times M$ defines a symplectic cobordism from
$(\{a\}\times M,C\Om_-)$ to $(\{c\}\times M,\Om_+)$. 
In general, however, such a diffeomorphism need not exist.  

\begin{example}
Consider a circle bundle $\pi:M\to W$ over a closed symplectic manifold
$(W,\bar\om)$ with the pullback (stabilizable) Hamiltonian structure
$\om=\pi^*\bar\om$. Then forms in $D_\om^+$ descend to the quotient and
$D_\om^+$ is homeomorphic to the space of positive symplectic forms on
$W$ whose pullback to $M$ is cohomologous to $\om$. In general, two
such symplectic forms need not be diffeomorphic.  
\end{example}

However, in dimension 3, Lemma~\ref{lem:Moser} and the preceding
discussion implies 

\begin{lemma}\label{lem:gluing}
Suppose that $\dim M=3$ and the foliated cohomology $H^2_\LL(M)$ for
$\LL=\ker(\om)$ is 1-dimensional. Then for all $\om_0,\om_1\in
D_\om^+$ there exists a constant $C>0$ and an 
orientation preserving diffeomorphism $\Phi:M\to M$ such that 
$$
   \Phi^*\om_1 = C\om_0. 
$$
Thus weak cobordisms as above can be composed at $(M,\om)$ after
rescaling. In 
particular, this situation occurs for a circle bundle $\pi:M\to W$
over a closed symplectic 2-manifold $(W,\bar\om)$ with the pullback
(stabilizable) Hamiltonian structure $\om=\pi^*\bar\om$. 
\end{lemma}

{\bf Second obstruction: fillability}. 
As a first application of this observation, we now have 

\begin{corollary}
There is no weak symplectic cobordism from
$(d\alpha_\st,\alpha_\st)$ to $(d\alpha_\ot ,\alpha_\ot )$, 
where $\alpha_\st$ is the standard tight contact form on $S^3$ and
$\alpha_\ot $ is an overtwisted contact form defining the same
orientation.  
\end{corollary}

\begin{proof}
Arguing by contradiction, suppose $([0,1]\times S^3,\Om)$ is a
symplectic cobordism with $\om_0:=\Om|_{\{0\}\times S^3}\in
D_{d\alpha_\st}^+$ and $\om_1:=\Om|_{\{1\}\times S^3}\in
D_{d\alpha_\ot}^+$. By Lemma~\ref{lem:gluing}, we can glue the
standard symplectic $4$-ball $(B^4,C\Om_\st)$, rescaled by some constant
$C>0$, to $([0,1]\times S^3,\Om)$ along $\{0\}\times S^3$ to get a 
symplectic form $\Om$ on $B^4$ with $\Om|_{\p B^4}\in
D_{d\alpha_\ot}^+$. By definition of $D_{d\alpha_\ot}^+$, the form
$\Om|_{\p B^4}$ restricts positively to $\ker\alpha_\ot $. Thus
$(B^4,\Om)$ is a weak symplectic filling of $(S^3,\alpha_\ot)$, 
which contradicts the theorem by Eliashberg and Gromov~\cite{El,Gr}
that weakly fillable contact manifolds are tight.
\end{proof}

Using rational symplectic field theory, we will show in
Section~\ref{subsec:sft} that $(d\alpha_\st,\alpha_\st)$ to
$(d\alpha_{ot}  ,\alpha_{ot}  )$ are not stably homotopic.

\subsection{$T^2$-invariant cobordisms in dimension three}

We return to the setup of Section~\ref{subsec:t2inv}. For an interval
$I\subset\R$ we consider the 4-manifold 
$$
   [0,1]\times I\times T^2
$$ 
with coordinates $(t,r,\theta,\phi)$, equipped with the $T^2$-action given
by shift in $(\theta,\phi)$, viewed as a topologically trivial
cobordism between $\{0\}\times I\times T^2$ and $\{0\}\times I\times
T^2$. 
Consider a $T^2$-invariant closed 2-form $\Om$ on $[0,1]\times I\times
T^2$. We can write it as $$\Om=dt\wedge\beta_t+\om_t$$ for $t$-dependent
$T^2$-invariant forms $\beta_t,\om_t$ on $I\times T^2$. Closedness of
$\Om$ is equivalent to $d\om_t=0$ and $d\beta_t=\dot\om_t$. By
Lemma~\ref{lem:t2inv}, $\om_t$ can be written as $\om_t=d\alpha_t$ for 
a smooth family of 1-forms $\alpha_t$, $t\in[0,1]$, of the form 
$$
   \alpha_t = h_1(t,r)d\theta + h_2(t,r)d\phi. 
$$
It follows that $d\beta_t=\dot\om_t=d\dot\alpha_t$ and thus
$\beta_t=\dot\alpha_t+\gamma_t$ for a family of $T^2$-invariant closed 1-forms
$\gamma_t$. After writing $\gamma_t=\dot\delta_t$ for a family of
closed 1-forms $$\delta_t=a_1(t)d\theta+a_2(t)d\phi+f_t(r)dr$$ and
replacing $\alpha_t$ by $$\alpha_t+a_1d\theta+a_2d\phi,$$ we thus have
$$
   \Om = dt\wedge(\dot\alpha_t+\dot f_t'dr)+d\alpha_t. 
$$
Denoting $t$-derivatives by $\dot h_i$ and $r$-derivatives by $h_i'$,
we compute
\begin{align*}
   \Om\wedge \Om 
   &= 2 dt\wedge \dot\alpha_t\wedge d\alpha_t \cr
   &= 2 dt\wedge (\dot h_1d\theta+\dot h_2d\phi)\wedge dr\wedge
   (h_1'd\theta+h_2'd\phi) \cr 
   &= 2(h_1'\dot h_2 - h_2'\dot h_1)dt\wedge dr\wedge d\theta\wedge
   d\phi. 
\end{align*}
Hence $\Om$ is a positive $T^2$-invariant symplectic form iff
\begin{equation*}
   h_1'\dot h_2 - h_2'\dot h_1 > 0
\end{equation*}
for all $(t,r)$. Geometrically, this condition means that the velocity
vector $h_t'=(h_1',h_2')$ and the $t$-derivative $\dot h_t=(\dot h_1,\dot
h_2)$ of the family of curves
$h_t=(h_1(t,\cdot),h_2(t,\cdot)):I\to\C$, $t\in[0,1]$, satisfy
\begin{equation}\label{eq:mon-hom}
   \la\dot h_t,ih_t'\ra>0,
\end{equation}
i.e.~$(\dot h_t,h_t')$ is a positive basis of $\C$ for all $t$. Note
that the functions $f_t$ do not enter
condition~\eqref{eq:mon-hom}. Therefore, from now on we will assume
$f_t=0$  and only consider $\Om$ of the form
\begin{equation}\label{eq:Om}
   \Om = dt\wedge\dot\alpha_t+d\alpha_t. 
\end{equation}

\begin{definition}
A {\em monotone homotopy} is a smooth family of 
curves $h_t:I\to\C$ satisfying condition~\eqref{eq:mon-hom}. 
\end{definition}

The preceding discussion shows

\begin{lemma}
Two $T^2$-invariant Hamiltonian structures $d\alpha_0$ and
$d\alpha_1$ on $I\times  T^2$ are cobordant by a $T^2$-invariant
cobordism iff there exists a monotone homotopy from $h_0:I\to\R^2$ to
$h_1:I\to\R^2$. 
\end{lemma}

Let us now restrict our attention to curves $h:[0,1]\to\C$ which are
{\em standardized} in the sense that $h(r)$ is a positive
multiple of $(r^2,1-r^2)$ near $\p[0,1]$.  
Note that each monotone homotopy $h_t$ is in particular a regular
homotopy, and if the $h_t$ are standardized then by
Corollary~\ref{cor:stablin} the corresponding homotopy of 
Hamiltonian structures $d\alpha_t$ can be stabilized by
$T^2$-invariant 1-forms. On the other hand, we will now construct
examples of standardized immersions $h_0,h_1$ which are regularly homotopic
(i.e.~have the same rotation number), but for which there exists no
standardized monotone homotopy from $h_0$ to a positive multiple of $h_1$. 

The first obstruction to monotone homotopies comes from winding
numbers as defined in Section~\ref{subsec:t2inv}. 

\begin{lemma}
Let $h_0,h_1:[0,1]\to\C\setminus\{0\}$ be standardized immersions
missing the origin. If there exists a standardized monotone
homotopy from $h_0$ to a positive multiple of $h_1$, then $w(h_1)\leq
w(h_0)$. Moreover, if in addition $w(h_1)<w(h_0)$, then the
corresponding symplectic cobordism contains an exact Lagangian 2-torus.  
\end{lemma}

\begin{proof}
Condition~\eqref{eq:mon-hom} implies that during a standardized
monotone homotopy $h_t$, the winding number around $0$ decreases by
$1$ each time $h_t$ 
passes through the origin. This proves the first statement. If in
addition $w(h_1)<w(h_0)$, then there exists a $(t,r)$ with $h_t(r)=0$,
so the primitive $\alpha$ of the symplectic form $\Om$ given
by~\eqref{eq:Om} vanishes on the torus $\{(t,r)\}\times T^2$.   
\end{proof}

\begin{example}\label{ex:nocob}
\begin{figure}[htb!]\label{fig:exot}
\centering
\includegraphics{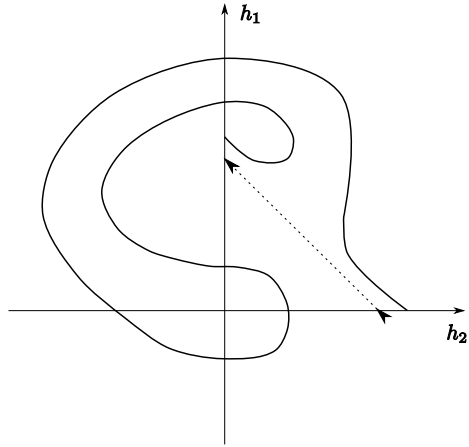}
\caption{Monotone homotopy}
\end{figure}
Figure \ref{fig:exot} shows a curve $h$ which is obtained from 
$h_\st(r)=(r^2,1-r^2)$ by a standardized monotone homotopy. (The curve $h_\st$ is
dashed, the curve $h$ is bold. The direction of increase of the
parameter $r$ on the curve $h_\st$ is designated by the two  
arrows. The curve $h$ is oriented accordingly.) On the other hand, it
has $w(h)=-1$, so there is no standardized monotone homotopy from $h$ to a positive
multiple of $h_\st$.  
\end{example}

\subsection{An exotic symplectic ball}

The following construction 
was communicated to the second author by Y.~Chekanov, who remembers
having seen it somewhere in the literature, but we were not able to
trace its origins. 

We consider $\R^4$ with its standard $T^2$-action and the standard
symplectic form
$$
   \Om_\st:=d\alpha_\st,\qquad \alpha_\st :=
   \frac{1}{2}\sum_{j=1}^2(x_jdy_j-y_jdx_j).  
$$
Denote by $B^4(r)$ the ball of radius $r$
around the origin and by $S^3(r):=\p B^4(r)$ its boundary sphere. 

\begin{proposition}\label{prop:exotic}
There exists an exact symplectic form $\Om$ on $B^4(1)$ with the following
properties: 
\begin{enumerate}
\item $\Om$ is $T^2$-invariant and $\Om=\Om_\st$ on $B^4(1/2)$. 
\item $\Om|_{S^3(r)}$, $r\in[1/2,1]$, can be stabilized to a
  $T^2$-invariant stable homotopy from $\frac{1}{2}(d\alpha_\st,\alpha_\st)$
  to the induced stable Hamiltonian structure $(\om,\lambda)$ on $S^3$.  
\item $(B^4,\Om)$ cannot be symplectically embedded into
  $(\R^4,\Om_\st)$. 
\end{enumerate}
\end{proposition}

\begin{remark}
%
    The existence of an exotic symplectic ball is of
    course not new: An example can be obtained by cutting a large ball
    out of an exotic $\R^4$. Historically, the first exotic symplectic
    structure on $\R^{2n}$ was found by Gromov in his
    paper~\cite{Gr}. Later an easy explicit construction of an
    exotic symplectic structure on $\R^4$ was given by Bates and
    Peschke~\cite{BP}. 
\end{remark}

\begin{proof}
The proof combines the construction in Example~\ref{ex:nocob} with a
result of Gromov.  
Let $h_t$, $t\in[1/2,1]$, be a standardized monotone homotopy from
$h_{1/2}(r):=\frac{1}{2}h_\st(r)=\frac{1}{2}(r^2,1-r^2)$ to the function $h_1=h$ in
Example~\ref{ex:nocob}. Due to the standardization, the associated
1-forms $\alpha_t$ extend to $S^3$, so~\eqref{eq:Om} defines an exact
symplectic structure $\Om$ on 
$[1/2,1]\times S^3$. Since $w(h)=-1$, $([1/2,1]\times S^3,\Om)$
contains an exact Lagrangian 2-torus. Since $h_{1/2}=\frac{1}{2}h_\st$, we can
glue a standard ball of radius $1/2$ to this cobordism to obtain an exact
symplectic manifold $(B^4,\Om)$. Properties (i) and (ii) in
the proposition follow immediately from the construction and
Corollary~\ref{cor:stablin}.  
Property (iii) follows from the fact that $(B^4,\Om)$ contains an
exact Lagrangian 2-torus, but $(\R^4,\Om_\st)$ does not by a theorem
of Gromov in~\cite{Gr}. 
\end{proof}

\begin{corollary}\label{cor:exotic}
There exists no symplectic cobordism $([1,2]\times S^3,\Om)$ with
$\Om_{\{1\}\times S^3}=\om$ the Hamiltonian structure in
Proposition~\ref{prop:exotic} and $\Om_{\{2\}\times S^3}\in
D_{d\alpha_\st}^+$. 
\end{corollary}

\begin{proof}
Suppose such a cobordism exists. Then we can glue this cobordism to
the ball in Proposition~\ref{prop:exotic} to obtain a symplectic
manifold $(B^4\cup [1,2]\times S^3,\Om)$, which in turn can be glued
using Lemma~\ref{lem:gluing} to $(\R^4\setminus S^3(R),\Om_\st)$
along a sphere $S^3(R)$ of suitable radius $R$. This yields an exact
symplectic structure $\Om$ on $\R^4$ with $\Om=\Om_\st$ outside
$B^4(R)$. By Gromov's uniqueness theorem for symplectic structures on
$\R^4$ standard at infinity~\cite{Gr}, $(\R^4,\Om)$ is
symplectomorphic to $(\R^4,\Om_\st)$, contradicting property (iii) in
Proposition~\ref{prop:exotic}. 
\end{proof}

\begin{cor}\label{cor:exotic2}
The SHS $(\om,\lambda)$ on $S^3$ in Corollary~\ref{cor:exotic} is 
stably homotopic to $(d\alpha_\st,\alpha_\st)$ but cannot be embedded
in $(\R^4,\Om_\st)$. 
\end{cor}

\begin{remark}\label{rem:exotic}
Note that this corollary does not contradict Lemma~\ref{lem:emb}
because the curve $h:[0,1]\to\C$ defining $\om$ does not remain in the
positive quadrant.  
\end{remark}

\begin{remark}
$(d\alpha_\st,\alpha_\st)$ and $(\om,\lambda)$ in
Proposition~\ref{prop:exotic} are candidates for stable Hamiltonian
structures that are stably homotopic and such that there exists no
weak cobordism from $\om$ to $d\alpha_\st$. However, the above
argument breaks down if we relax the assumption $\Om_{\{1\}\times
  S^3}=\om$ in Corollary~\ref{cor:exotic} to $\Om_{\{1\}\times S^3}\in 
D_\om^+$ because the foliated cohomology $H^2_\LL(S^3)$ for
$\LL=\ker(\om)$ is infinite dimensional
(Proposition~\ref{prop:folcoh3}), so we cannot apply 
Lemma~\ref{lem:gluing} to glue the cobordisms.  
\end{remark}

\subsection{Ambient homotopies and Rabinowitz Floer
  homology}\label{subsec:rfh} 

Any embedding $\psi:M\into W$ of a hypersurface in a symplectic
manifold $(W,\Om)$ induces a Hamiltonian structure $\om=\psi^*\Om$ on
$M$, and any smooth isotopy of embeddings $\psi_t:M\into W$ induces an
homotopy of HS $\om_t=\psi_t^*\Om$. We call such a homotopy
$\om_t$ an {\em ambient homotopy}. 
Corollary~\ref{cor:exotic2} shows that not every homotopy
of $\om=\psi^*\Om$ can be realized by an ambient homotopy.

Now we turn to the question of ambient stable homotopies: Given two 
stable hypersurfaces $M_0,M_1$ in a symplectic manifold $(W,\Om)$, are
they stably homotopic? Under an additional technical hypothesis, this 
question has a negative answer and can be addressed by Rabinowitz
Floer homology~\cite{CFP}. We first give the relevant definitions. 

For the remainder of this subsection, let $(W,\Om)$ be a fixed
symplectic manifold of dimension $2n$. We assume that
$\Om|_{\pi_2(M)}=0$ and $(W,\Om)$ is convex at infinity or
geometrically bounded (see~\cite{CFP}). For a hypersurface $M\subset
W$ denote by $\RR(M)$ the space of periodic orbits on $M$ that
are contractible in $W$. Recall that to $\gamma\in\RR(M)$ and a 1-form
$\lambda$ on $M$ we associate the energies 
$$
   E_\Om(\gamma) = \int_D^2\bar\gamma^*\om,\qquad
   E_\lambda(\gamma) = \int_\gamma\lambda,
$$
where $\bar\gamma:D^2\to W$ with $\bar\gamma|_{\p D^2}=\gamma$. All
hypersurfaces $M\subset W$ are assumed to be closed and {\em
  separating}, i.e.~$W\subset M$ consists of two connected
components. 

\begin{definition}
A hypersurface $M$ in $(W,\Om)$ is
called {\em stable} if the HS $\Om|_M$ is stabilizable. A smooth
homotopy of hypersurfaces $(M_t)_{t\in[0,1]}$ in $W$ is called {\em
stable} if there exists a smooth homotopy of stabilizing 1-forms
$\lambda_t$ for $\Om|_{M_t}$. Recall that a stable hypersurface $M$ is
called {\em tame} if for some (and hence every) stabilizing 1-form
$\lambda$ there exists a constant $c_\lambda>0$ such that
$E_\lambda(\gamma)\leq c_\lambda|E_\Om(\gamma)|$ for all
$\gamma\in\RR(M)$. A stable homotopy $M_t$ is called {\em tame} if there
exists a smooth homotopy of stabilizing 1-forms $\lambda_t$ and a
constant $c$ such that $E_{\lambda_t}(\gamma)\leq
c|E_\Om(\gamma)|$ for all $\gamma\in\RR(M_t)$ and all $t\in[0,1]$.  
\end{definition}

{\em Rabinowitz Floer homology (RFH)}~\cite{CF} associates to every
(closed, separating) stable tame hypersurface $M\subset W$ a
$\Z_2$-vector space $RFH(M)$ which is invariant under tame stable
homotopies. Moreover, it has the following properties:
\begin{enumerate}
\item If $M$ is displaceable (by a Hamiltonian isotopy) then
  $RFH(M)=0$. 
\item If $RFH(M)=0$ then $M$ carries a periodic orbit
  contractible in $W$. 
\end{enumerate}
In particular, if $RFH(M_0)\neq RFH(M_1)$ for two tame stable
hypersurfaces, then they are not tame stably homotopic. Using this,
many examples of smoothly but not tame stably homotopic hypersurfaces
are constructed in~\cite{CFP}. Here we just give one example. 

\begin{example}[\cite{CFP} Theorem 1.6]\label{ex:Heisenberg}
Let $G$ be the 3-dimensional Heisenberg group of matrices
\[\left(\begin{array}{ccc}
1&x&z\\
0&1&y\\
0&0&1\\
\end{array}\right),\]
where $x,y,z\in \R$. The 1-form $\gamma:=dz-xdy$ is left-invariant and we
let $\sigma:=d\gamma$ be the exact magnetic field. If $\Gamma$ is a co-compact
lattice in $G$, $Q:=\Gamma\setminus G$ is a closed 3-manifold and $\sigma$
descends to an exact 2-form on $Q$. Equip the cotangent bundle
$\tau:T^*Q\to Q$ with the symplectic form $\Om = dp\wedge dq +
\tau^*\sigma$ and the left-invariant Hamiltonian  
\[H:=\frac{1}{2}\bigl(p_{x}^2+(p_{y}+xp_{z})^2+p_{z}^2\bigr).\]
Then each level set $M_t=H^{-1}(t)$ for $t\neq 1/2$ is
stable and tame. For $t>1/2$, $M_t$ has no contractible periodic
orbits and thus $RFH(M_t)\neq 0$. For $t<1/2$, $M_t$ is displaceable
and thus $RFH(M_t)=0$. Therefore, two level sets
$M_s,M_t$ with $s<1/2<t$ are smoothly homotopic and
tame stable, but not tame stably homotopic.  
\end{example}

\begin{remark}
In the preceding example, $t=1/2$ is the Ma\~n\'e critical value of
the Hamiltonian system defined by $\Om$ and $H$ (see~\cite{CFP}). It
appears to be a 
general feature (though not a proven theorem) that level sets below
and above the Ma\~n\'e critical value are not tame stably homotopic. 
\end{remark}

We expect that the hypersurfaces $M_s,M_t$ with $s<1/2<t$ in the
preceding example are not stably homotopic (i.e.~without the tameness
assumption), and more generally $\Om|_{M_s},\Om|_{M_t}$ are not stably
homotopic through SHS on the unit cotangent bundle $S^*Q$. This
appears to lie outside the scope of established techniques such as
RFH, but may be approached using stronger invariants such as
symplectic field theory discussed in the next subsection.

\subsection{Symplectic field theory}\label{subsec:sft}

The motivation for this paper came from {\em symplectic field theory
(SFT)} introduced in~\cite{EGH}. We conclude this section by
discussing the relevance of some of our results to the foundations of
SFT, and conversely, the relevance of SFT (once it is defined) for the
homotopy classification of stable Hamiltonian structures. 

We begin by formalizing the expected TQFT properties of SFT and how
they give rise to a homotopy invariant for stable Hamiltonian
structures. 

\begin{definition}
A {\em (homological) SFT functor} is a contravariant
functor $\FF$ from the category of SHS to a category $\CC$. Thus
$\FF$ associates 
\begin{itemize}
\item to every SHS $(M,\om,\lambda)$ an object $\FF(M,\om,\lambda)$ in
  $\CC$; 
\item to every symplectic cobordism $(W,\Om)$ a morphism
  $\FF(W,\Om)$ in $\CC$; 
\item to homotopic symplectic cobordisms the same morphism in $\CC$;  
\item to a trivial cobordism $\bigl([0,1]\times
  M,\om+d(f(t)\lambda\bigr))$ an isomorphism in $\CC$
\end{itemize}
such that the usual functoriality properties hold. 
\end{definition}

\begin{remark}\label{rem:category}
(a) We are a little sloppy when speaking about the ``category of
SHS''. For example, this category has no identity morphisms, and
morphisms should be concatenations of cobordisms rather that single
cobordisms. See~\cite{El-ICM} for a more thorough discussion. 

(b) For concreteness, we will assume that objects in $\CC$ are vector
spaces with some additional structure, and morphisms are linear maps
preserving this structure, so that a morphism is invertible iff it is
injective and surjective. This is the case for the algebraic
formulation of SFT in~\cite{CL}. 
\end{remark}

For two HS $\om_0,\om_1$ with the same kernel $\LL$ and a stabilizing
1-form $\lambda$ for $\LL$ we write $\om_0<_\lambda\om_1$ iff
$\om_1=\om_0+\tau d\lambda$ for a $\tau>0$ such that $\ker(\om_0+t
d\lambda)=\LL$ for all $t\in[0,\tau]$. Note that for fixed $\lambda$
this defines a partial ordering on HS with kernel $\LL$, and a
complete ordering on each equivalence class $C(\om,\lambda)$. 

\begin{prop}\label{prop:SFT-functor}
For an SFT functor $\FF$ and SHS $(\om_i,\lambda_i)$ on $M$ the
following holds. 

(a) For $\om_0<_\lambda\om_1$ there exist canonical isomorphisms
$$
   \psi_{(\om_0,\om_1;\lambda)}:\FF(\om_1,\lambda_1)\to
   \FF(\om_0,\lambda_0)
$$ 
such that for $\om_0<_\lambda\om_1<_\lambda\om_2$ we have 
$$
   \psi_{(\om_0,\om_1;\lambda)}\psi_{(\om_1,\om_2;\lambda)} =
   \psi_{(\om_0,\om_2;\lambda)}.  
$$
(b) Each stable homotopy $\gamma=\{\om_t,\lambda_t\}_{t\in[0,1]}$
induces a canonical isomorphism 
$$
   \phi_\gamma:\FF(\om_1,\lambda_1)\to
   \FF(\om_0,\lambda_0)
$$ 
such that for stable homotopies $\gamma$, $\gamma'$ with
$\gamma_1=\gamma_0'$ we have 
$$
   \phi_{\gamma}\phi_{\gamma'} = \phi_{\gamma\#\gamma'}. 
$$
\end{prop}

\begin{proof}
(a) For $\om_0<_\lambda\om_1$ there exists a trivial cobordism
from $\om_0$ to $\om_1$ which induces an isomorphism
$\psi_{(\om_0,\om_1;\lambda)}:\FF(\om_1,\lambda_1)\to
\FF(\om_0,\lambda_0)$. This isomorphism does not depend on the trivial
cobordism since any two trivial cobordisms from $\om_0$ to $\om_1$ are
homotopic. The composition property follows from
functoriality of $\FF$ and the fact that the composition of trivial
cobordisms for the same $\lambda$ is again a trivial cobordism. 

(b) Consider a stable homotopy
$\gamma=\{\om_t,\lambda_t\}_{t\in[0,1]}$. Assume first that
$L(\gamma)<1/3$ and let $([0,3]\times M,\Om)$ be the symplectic
cobordism provided by Proposition~\ref{maincob}. This is the
composition $W_1W_2W_3$ of three cobordisms from $\om_0^-$ to
$\om_1^-$ to $\om_0^+$ to $\om_1^+$, where $\om_i^\pm\in
C_{(\om_i,\lambda_i)}$. Without loss of generality we may assume
$\om_i^-<_{\lambda_i}\om_i<_{\lambda_i}\om_i^+$. Denote by $\FF(W_i)$
the induced morphisms. Since $W_1W_2$ and $W_2W_3$ are homotopic
to trivial cobordisms, the compositions
$\FF(W_1)\FF(W_2)$ and $\FF(W_2)\FF(W_3)$ are isomorphisms, hence
$$
   \FF(W_2):\FF(\om_0^+,\lambda_0)\to\FF(\om_1^-,\lambda_1)
$$
is an isomorphism (see Remark~\ref{rem:category}). Define the isomorphism
$$
   \phi_\gamma := \psi_{(\om_1^-,\om_1;\lambda_1)}^{-1} \FF(W_2)
   \psi_{(\om_0,\om_0^+;\lambda_0)}^{-1} : \FF(\om_0,\lambda_0)\to  
   \FF(\om_1,\lambda_1). 
$$
By the composition property in (a), this isomorphism does not depend
on the choice of $\om_0^+$ and $\om_1^-$ and is thus canonically
defined. 
 
Now let us drop the assumption $L(\gamma)<1/3$. By the (Restriction)
property in Section~\ref{subsec:shorth}, $\gamma$ can be written as a
concatenation $\gamma=\gamma_1\#\dots\#\gamma_N$ of homotopies of
length $L(\gamma_i)<1/3$. We define 
$$
   \phi_\gamma := \phi_{\gamma_1}\cdots\phi_{\gamma_N}:
   \FF(\om_0,\lambda_0)\to \FF(\om_1,\lambda_1)
$$
with the isomorphisms $\phi_{\gamma_i}$ defined above. We need to show
that this is independent of the decomposition of $\gamma$ into short
homotopies $\gamma_i$. After taking a common refinement of two
decompositions, this reduces to showing
$\phi_\gamma=\phi_{\gamma_1}\phi_{\gamma_2}$ for a homotopy
$\gamma=\gamma_1\#\gamma_2$ of length $L(\gamma)<1/3$. After
reparametrization, we may assume that
$\gamma=\{(\om_t,\lambda_t)\}_{t\in[0,2]}$ and $\gamma_1,\gamma_2$ are
the restrictions to the intervals $[0,1]$ and $[1,2]$,
respectively. By definition,
$$
   \phi_\gamma = \psi_{(\om_0^-,\om_0;\lambda_0)}^{-1} \FF(W)
   \psi_{(\om_2,\om_2^+;\lambda_2)}^{-1} : \FF(\om_2,\lambda_2)\to  
   \FF(\om_0,\lambda_0)
$$
for a cobordism $(W,\Om)$ from $\om_0^-$ to $\om_2^+$ provided by
Proposition~\ref{maincob}. Without loss of generality we may assume
that $\Om_{\{1\}\times M}=\om_1$, so $W=W_1W_2$ for cobordisms
$W_1$ from $\om_0^-$ to $\om_1$ and $W_2$ from $\om_1$ to
$\om_2^+$. Then we have 
$$
   \phi_{\gamma_1} = \psi_{(\om_0^-,\om_0;\lambda_0)}^{-1}
   \FF(W_1),\qquad 
   \phi_{\gamma_2} = \FF(W_2)\psi_{(\om_2^+,\om_2;\lambda_2)}^{-1},  
$$
which together with $\FF(W_1)\FF(W_2)=\FF(W)$ yields
$$
   \phi_{\gamma_1}\phi_{\gamma_2}  
   = \psi_{(\om_0^-,\om_0;\lambda_0)}^{-1}
   \FF(W_1)\FF(W_2)\psi_{(\om_2^+,\om_2;\lambda_2)}^{-1} =
   \phi_\gamma. 
$$
This proves that the isomorphism $\phi_\gamma$ is independent of the
decomposition of $\gamma$ into short homotopies and thus canonically
defined. The same argument also shows the composition property. 
\end{proof}

\begin{remark}
More generally, one could define a {\em homotopical} SFT functor from
the 2-category of SHS to a 2-category $\CC$, with homotopies of
cobordisms giving rise to homotopies in $\CC$ and a trivial cobordism
inducing a homotopy equivalence. Then the arguments in the proof of
Proposition~\ref{prop:SFT-functor} show that stable homotopies give rise
to homotopy equivalences.   
\end{remark}

Let us now discuss how $\SFT$ as introduced in~\cite{EGH} should give 
rise to an SFT functor in the sense above. For this, let us first
consider only SHS $(M,\om,\lambda)$ that are {\em Morse-Bott}. 

An $\R$-invariant almost complex structure $J$ on $\R\times M$ is
called {\em adjusted to $(\om,\lambda)$} if $J(\p_r)=R$, $J(\xi)=\xi$,
and $J|_\xi$ is {\em tamed} by $\om|_\xi$ in the sense that
$\om(v,Jv)>0$ for all $0\neq v\in\xi$. $\SFT$ associates to such $(M,J)$
an algebraic object (Poisson algebra, Weyl algebra etc, depending on
the version of SFT) $\SFT(M,J)$ by counting suitable holomorphic
curves in $(\R\times M,J)$. 

Next consider a symplectic cobordism $(W,\Om)$ with stable boundaries
$(M_\pm,\om_\pm=\Om|_{M_\pm},\lambda_\pm)$. Note that we have
canonical vector fields $X_\pm$ along $M_\pm$ defined by
$i_{X_\pm}\Om=\lambda_\pm$, so $X_+$ is outward pointing and $X_-$ is
inward pointing. An almost complex structure $J$ on $W$ is
called {\em adjusted to $(\Om,\lambda_\pm)$} if it is tamed by $\Om$
and $J(X_\pm)=R_\pm$, $J(\xi_\pm)=\xi_\pm$ along $M_\pm$. Denote by
$(\hat W,\hat J)$ the almost complex manifold obtained by gluing
semi-infinite collars $R_\pm\times M_\pm$ to $W$ along $M_\pm$ and
extending $J$ $\R$-invariantly by $J_\pm$ to the collars. $\SFT$
associates to such $(W,J)$ a morphism $\SFT(W,J)$ from $(M_+,J_+)$ to
$(M_-,J_-)$ by counting suitable holomorphic curves in $(\hat W,\hat
J)$. 

Composition of cobordisms induces composition of morphisms and
homotopies of $J_t$ adjusted to the same $(\Om,\lambda_\pm)$ give
rise to the same morphisms. Since the space of $J$ adjusted to
$(\Om,\lambda_\pm)$ is contractible, this implies that $\SFT$ is in
fact independent of the adjusted $J$, so objects and morphisms can be
written as $\SFT(M,\om,\lambda)$ and $\SFT(W,\Om)$. 

By construction, $\SFT$ will satisfy the first three axioms of an SFT
functor given above. For the last axiom, consider a trivial cobordism 
$\bigl(W=[0,1]\times M,\Om=\om+d(f(t)\lambda\bigr)$ between SHS
$\bigl(\om_0=\om+f(0)d\lambda,\lambda\bigr)$ and
$\bigl(\om_1=\om+f(1)d\lambda\bigr)$ on $M$. We wish to show that
$\SFT(W,\Om)$ is an isomorphism. After cutting $[0,1]$ into smaller
intervals and using functoriality, we may assume that the
$\om_t:=\Om|_{\{t\}\times M}$, $t\in[0,1]$, are arbitrarily
$C^1$-close. Since taming is an open condition, we then find an almost
complex structure $J$ on $M$ that is tamed by all the $\om_t$. For
this $t$-independent $J$ the morphism $\SFT(W,J)$ is the identity and
the last axiom follows. 

\begin{remark}\label{rem:SFT}
At the time of writing, $\SFT$ has not yet been rigorously defined.
The solution of the relevant transversality problems for holomorphic
curves is work in progress by Hofer, Wysocki and
Zehnder~\cite{HWZ}. Apart from these well-known problems, we wish to
point out another issue that needs to be addressed. A priori, the
construction of $\SFT$ only works for stable Hamiltonian structures
that are {\em Morse-Bott}. On the other hand, in the proof of homotopy
invariance in Proposition~\ref{prop:SFT-functor} we need to cut a stable
homotopy into short ones and we cannot guarantee that the SHS at which
we cut are Morse-Bott. So in order to define $\SFT$ as a homotopy
invariant of SHS, one needs to define $\SFT$ also for SHS that are not
Morse-Bott. Moreover, in view of Theorem~\ref{thm:main} one cannot
approximate a given SHS by {\em Morse} ones. The counterexample in
Theorem~\ref{thm:main} is still Morse-Bott, which would suffice to
define SFT, but it appears hopeless to try to prove that any SHS can be
approximated by Morse-Bott ones. One way out may be perturbing the
holomorphic curve equation used in $\SFT$ by a suitable Hamiltonian
term. 
\end{remark}

{\bf Dimension 3. } 
In dimension three, the problem in Remark~\ref{rem:SFT} can be
overcome due to Theorem~\ref{thm:Morse-Bott} and its following
consequence. 
 
\begin{cor}\label{cor:Morse-Bott}
Let $M$ be a closed oriented 3-manifold $M$. Then any stable homotopy
$\gamma$ between Morse-Bott SHS on $M$ can be written as a concatenation
$\gamma=\gamma_1\#\dots\#\gamma_N$ of stable homotopies $\gamma_i$
between Morse-Bott SHS of length $L(\gamma_i)<1/3$, $i=1,\dots,N$. 
\end{cor}

\begin{proof}
By the (Restriction) property in Section~\ref{subsec:shorth}, $\gamma$
can be written as a concatenation
$\gamma=\bar\gamma_1\#\dots\#\bar\gamma_N$ of homotopies of length
$L(\bar\gamma_i)<1/3$. Denote the end points of $\gamma_i$ by
$(\om_{i-1},\lambda_{i-1})$ and $(\om_i,\lambda_i)$. Thus
$(\om_0,\lambda_0)$ and $(\om_N,\lambda_N)$ are Morse-Bott but the
other ones need not be. Since the length in invariant under linear
rescaling, we may assume that each $\gamma_i$ is parametrized over an
interval of length 1. 

According to Theorem~\ref{thm:Morse-Bott}, there exists for each
$i=1,\dots,N=1$ a stable homotopy $\delta_i$ (parametrized over
$[0,1]$) from $(\om_i,\lambda_i)$ to a Morse-Bott SHS with arbitrarily
small $\|\delta_i\|_{C^1}$. We let $\delta_0=(\om_0,\lambda_0)$ and
$\delta_N=(\om_N,\lambda_N)$ be the constant homotopies. Applying
Lemma~\ref{lem:concat} twice for each $i=1,\dots,N$, we can thus
achieve that each concatenation
$\gamma_i:=\delta_{i-1}^{-1}\#\bar\gamma_i\#\delta_i$ has length
$L(\gamma_i)<1/3$. Since the end points of each $\gamma_i$ are
Morse-Bott, this concludes the proof. 
\end{proof}

Using Corollary~\ref{cor:Morse-Bott}, the proof of
Proposition~\ref{prop:SFT-functor} can be carried out using in the
full subcategory of Morse-Bott SHS to obtain an invariant of
Morse-Bott SHS up to stable homotopies (through SHS that need not
be Morse-Bott). Since every SHS is stably homotopic to a Morse-Bott
one, this invariant extends to all SHS and we have shown:

\begin{thm}\label{thm:sft}
Suppose transversality for holomorphic curves can be achieved. Then $\SFT$
in~\cite{EGH} gives rise to a homotopy invariant of SHS in dimension 3. 
\end{thm}

\begin{cor}\label{cor:sft}
Suppose transversality for holomorphic curves can be achieved. 
Then the SHS $(d\alpha_\st,\alpha_\st)$ and $(d\alpha_\ot,\alpha_\ot)$ on
$S^3$ are not stably homotopic, where $\alpha_\st$ is the standard
contact form and $\alpha_\ot$ is an overtwisted contact form defining
the same orientation. 
\end{cor}

\begin{proof}
By the assumption on transversality and Theorem~\ref{thm:sft}, rational
symplectic field theory $RSFT$ in~\cite{EGH} gives rise to a homotopy
invariant of SHS in dimension 3. 
It is well-known (see e.g.~\cite{Ya,BN}) that
$RFST(S^3,\alpha_\ot)=\{0\}$. On the other hand, for 
the standard contact form $\alpha_\st$ all closed Reeb orbits have
even degree, so the RSFT Hamiltonian (without differential forms) for
$\alpha_\st$ vanishes $RSFT(S^3,\alpha_\st)$ is the free Poisson
algebra generated by the closed Reeb orbits. Hence
$(d\alpha_\st,\alpha_\st)$ and $(d\alpha_\ot,\alpha_\ot)$ have
different RSFT and therefore are not homotopic. 
\end{proof}

\begin{remark}
$\SFT$ is certailny {\em not} invariant under {\em non-exact} stable
homotopies $(\om_t,\lambda_t)$, i.e.~with varying cohomology class
$\om_t$. For example, consider $T^3$ with coordinates $(x,y,z)$ and
the non-exact stable homotopy 
$$
   \om_t = dx\wedge dy + t\,dx\wedge dz,\qquad \lambda_t=dz. 
$$    
(This corresponds to the mapping torus of a rotation of $T^2$ by angle
$t$). Note that the Reeb vector field is $R_t=\p_z-t\p_y$. For $t=0$ all
Reeb orbits are closed and rational SFT is nontrivial. On the other hand,
for $t$ irrational there are no closed Reeb orbits and hence rational
SFT is trivial. 
\end{remark}


\end{document}